\documentclass[11pt,reqno]{amsart}
\usepackage[a4paper, hmargin={2.7cm,2.7cm},vmargin={3.3cm,3.3cm}]{geometry}
\usepackage[english]{babel}
\usepackage{color}
\usepackage{amsmath}
\usepackage{amssymb}
\usepackage{amsfonts}
\usepackage{amsthm}
\usepackage{enumerate}
\usepackage{amsthm}
\usepackage[noadjust]{cite}
\usepackage{dsfont}
\usepackage{etoolbox}
\usepackage[unicode=true]{hyperref}
\usepackage[noabbrev]{cleveref}

\usepackage{todonotes}

\makeatletter
\patchcmd{\ttlh@hang}{\parindent\z@}{\parindent\z@\leavevmode}{}{}
\patchcmd{\ttlh@hang}{\noindent}{}{}{}
\makeatother

\newcommand\numberthis{\addtocounter{equation}{1}\tag{\theequation}}

\makeatletter
\@namedef{subjclassname@2020}{\textup{2020} Mathematics Subject Classification}
\makeatother

\theoremstyle{plain}
\newtheorem{theorem}{Theorem}[section]
\newtheorem{lemma}[theorem]{Lemma}
\newtheorem{proposition}[theorem]{Proposition}
\newtheorem{corollary}[theorem]{Corollary}

\newtheorem{innercustomgeneric}{\customgenericname}
\providecommand{\customgenericname}{}
\newcommand{\newcustomtheorem}[2]{%
  \newenvironment{#1}[1]
  {%
   \renewcommand\customgenericname{#2}%
   \renewcommand\theinnercustomgeneric{##1}%
   \innercustomgeneric
  }
  {\endinnercustomgeneric}
}

\newcustomtheorem{customthm}{Theorem}
\newcustomtheorem{customlemma}{Lemma}

\theoremstyle{definition}
\newtheorem{definition}[theorem]{Definition}

\theoremstyle{remark}
\newtheorem{remark}[theorem]{Remark}
\newtheorem*{remark*}{Remark}

\numberwithin{equation}{section}

\DeclareMathOperator*{\loc}{loc}

\DeclareMathOperator*{\Span}{span}
\DeclareMathOperator*{\supp}{supp}
\DeclareMathOperator*{\diag}{diag}
\DeclareMathOperator*{\esssup}{ess\,sup}

\DeclareMathOperator{\Co}{Co}

\usepackage{xparse}

\newcommand{\WLwr}{\mathcal{W}(L^r_w)}

\newcommand{\Schwartz}{\mathcal{S}}
\newcommand{\SP}{\mathcal{S}' (\mathbb{R}^d) / \mathcal{P} (\mathbb{R}^d)}

\newcommand{\MQ}{M_Q}
\newcommand{\Indicator}{\mathds{1}}
\newcommand{\TL}{\dot{\mathbf{F}}^{\alpha}_{p,q}}
\NewDocumentCommand\TLseq{O{\alpha}}{\dot{\mathbf{f}}^{#1}_{p,q}}
\NewDocumentCommand\PT{O{\alpha}D<>{\beta}}{\dot{\mathbf{P}}^{#1, #2}_{p,q}}
\newcommand{\TLA}{\dot{\mathbf{F}}^{\alpha}_{p,q}(A)}
\newcommand{\TLseqa}{\dot{\mathbf{f}}^{\alpha}_{p,q}}
\newcommand{\TLseqA}{\dot{\mathbf{f}}^{\alpha}_{p,q}(A)}
\NewDocumentCommand\PTseq{O{\alpha}D<>{\beta}}{\dot{\mathbf{p}}^{#1, #2}_{p,q}}

\newcommand{\vertiii}[1]{{\left\vert \kern-0.25ex
                            \left\vert \kern-0.25ex
                              \left\vert #1\right\vert\kern-0.25ex
                            \right\vert \kern-0.25ex
                          \right\vert}}

\NewDocumentCommand\DoubleStar{O{\varphi}m}{#1_{#2,\beta}^{\ast\ast}}

\newcommand{\eps}{\varepsilon}
\renewcommand{\emptyset}{\varnothing}
\newcommand{\PosPart}[1]{#1^+}
\newcommand{\NegPart}[1]{#1^-}
\newcommand{\GroupMaximal}{(F^{**})_{\beta}}

\newcommand{\CalP}{\mathcal{P}}
\newcommand{\Measure}{\mathrm{m}}
\newcommand{\Lebesgue}[1]{\Measure(#1)}
\newcommand{\GL}{\operatorname{GL}}

\DeclareFontFamily{U}{mathx}{\hyphenchar\font45}
\DeclareFontShape{U}{mathx}{m}{n}{
      <5> <6> <7> <8> <9> <10>
      <10.95> <12> <14.4> <17.28> <20.74> <24.88>
      mathx10
      }{}
\DeclareSymbolFont{mathx}{U}{mathx}{m}{n}
\DeclareFontSubstitution{U}{mathx}{m}{n}
\DeclareMathAccent{\widecheck}{0}{mathx}{"71}
\DeclareMathAccent{\wideparen}{0}{mathx}{"75}

\newcommand{\R}{\mathbb{R}}
\newcommand{\SC}{\mathcal{S}}
\newcommand{\CC}{\mathbb{C}}
\newcommand{\N}{\mathbb{N}}
\newcommand{\Z}{\mathbb{Z}}
\newcommand{\PTalt}[1]{\dot{\mathbf{P}}^{#1, \beta}_{p,q}}

\title[Anisotropic Triebel-Lizorkin spaces and wavelet coefficient decay, I]
      {Anisotropic Triebel-Lizorkin spaces and wavelet coefficient decay over one-parameter dilation groups, I}

\author[S. Koppensteiner]{Sarah Koppensteiner}
\address{Faculty of Mathematics,
University of Vienna,
Oskar-Morgenstern-Platz 1,
A-1090 Vienna, Austria}
\email{sarah.koppensteiner@univie.ac.at}
\email{jordy-timo.van-velthoven@univie.ac.at}
\email{felix@voigtlaender.xyz}

\author[J.T. van Velthoven]{Jordy Timo van Velthoven}

\author[F. Voigtlaender]{Felix Voigtlaender}

\address{Delft University of Technology,
Mekelweg 4, Building 36,
2628 CD Delft, The Netherlands.}
\email{j.t.vanvelthoven@tudelft.nl}

\address{
Katholische Universit\"at Eichst\"att-Ingolstadt,
Mathematical Institute for Machine Learning and Data Science (MIDS),
Research group \emph{Reliable Machine Learning},
Ostenstrasse 26,
85072 Eichst\"att,
Germany
}
\email{felix.voigtlaender@ku.de}


\begin{document}

\maketitle

\begin{abstract}
This paper provides maximal function characterizations of anisotropic Triebel-Lizorkin spaces
associated to general expansive matrices for the full range of parameters $p \in (0,\infty)$,
$q \in (0,\infty]$ and $\alpha \in \mathbb{R}$.
The equivalent norm is defined in terms of the decay of wavelet coefficients,
quantified by a Peetre-type space over a one-parameter dilation group.
As an application, the existence of dual molecular frames and Riesz sequences is obtained;
the wavelet systems are generated by translations and anisotropic dilations of a single function,
where neither the translation nor dilation parameters are required to belong to a discrete subgroup.
Explicit criteria for molecules are given in terms of mild decay, moment, and smoothness conditions.
\end{abstract}

\section{Introduction}
Let $A \in \mathrm{GL}(d, \mathbb{R})$ be an expansive matrix; that is,
all eigenvalues $\lambda \in \CC$ of $A$ satisfy $|\lambda| > 1$.
Choose a Schwartz function $\varphi \in \mathcal{S} (\mathbb{R}^d)$
whose Fourier transform $\widehat{\varphi}$ has compact support
\begin{align}\label{eq:support1}
  \supp \widehat{\varphi}
  = \overline{
      \{
        \xi \in \mathbb{R}^d : \widehat{\varphi} (\xi) \neq 0
      \}
    }
  \subset \mathbb{R}^d \setminus \{ 0 \}
\end{align}
and satisfies
\begin{align}\label{eq:support2}
  \sup_{j \in \mathbb{Z}}
    | \widehat{\varphi} ((A^*)^j \xi) |
  > 0,
  \quad \xi \in \mathbb{R}^d \setminus \{0\},
\end{align}
where $A^*$ denotes the transpose of $A$.
Following Bownik and Ho \cite{bownik2006atomic},
we define the (homogeneous) \emph{anisotropic Triebel-Lizorkin space} $\TL (\mathbb{R}^d; A)$,
with $p \in (0, \infty)$, $q \in (0,\infty)$ and $\alpha \in \mathbb{R}$,
as the collection of all tempered distributions $f \in \mathcal{S}' (\mathbb{R}^d)$
(modulo polynomials) satisfying
\begin{align*}
  \| f \|_{\TL}
  := \bigg\|
       \bigg(
         \sum_{j \in \mathbb{Z}} (|\det A|^{j\alpha} |f \ast \varphi_j |)^q
       \bigg)^{1/q}
     \bigg\|_{L^p}
  < \infty,
\end{align*}
where $\varphi_j := |\det A|^j \varphi (A^j \cdot)$.
The space $\dot{\mathbf{F}}^{\alpha}_{p,\infty} (\mathbb{R}^d; A)$
is defined via the usual modifications.

The dilation group $\{A^j : j \in \mathbb{Z} \} \leq \mathrm{GL}(d, \mathbb{R})$
generated by an expansive matrix $A$ induces the structure
of a space of homogeneous type on $\mathbb{R}^d$,
which differs from the usual isotropic homogeneous structure on $\mathbb{R}^d$,
unless $A$ is $\mathbb{C}$-diagonalizable with all eigenvalues equal in absolute value,
 \cite{bownik2003anisotropic}.
A particular motivation for the study of function spaces defined
through such non-isotropic structures is the analysis of mixed homogeneity properties
of functions and operators.
The scale of spaces $\TL (\mathbb{R}^d; A)$ considered here contains, among others,
the anisotropic and parabolic Hardy spaces
$H^p(\mathbb{R}^d; A) \cong \dot{\mathbf{F}}^{0}_{p, 2} (\mathbb{R}^d; A)$ for $p \in (0,1]$
and the Lebesgue spaces $L^p(\mathbb{R}^d)  \cong \dot{\mathbf{F}}^{0}_{p, 2} (\mathbb{R}^d; A)$
for $ p \in (1,\infty)$; see Section \ref{sec:hardy}.
We refer to Bownik
\cite{bownik2003anisotropic,bownik2006atomic,bownik2005atomic,bownik2007anisotropic,bownik2008duality},
Calder\'on and Torchinsky \cite{calderon1975parabolic, calderon1977parabolic, calderon1977atomic},
and Stein and Wainger \cite{stein1978problems} for more background and motivation
regarding anisotropic dilations and associated function spaces.

The purpose of the present paper is to derive various characterizations of the spaces
$\TL (\R^d; A)$, with $p \in (0,\infty)$ and $q \in (0,\infty]$,
in terms of Peetre-type maximal functions.
Our main motivation for such characterizations is that they allow to identify
a Triebel-Lizorkin space as a coorbit space \cite{feichtinger1989banach}
associated with a Peetre-type space on an affine-type group.
This identification will be used to obtain decompositions of the spaces $\TL (\R^d; A)$
in which both the analyzing and synthesizing functions are ``molecular systems''
(see Section~\ref{sec:molecule_intro});
the recent discretization results \cite{romero2020dual, velthoven2022quasi}
are used for this purpose.

Similar results for the endpoint case of $p = \infty$ are obtained
in the subsequent paper \cite{koppensteiner2022anisotropic2}.

\subsection{Maximal characterizations}

Throughout, in addition to $A$ being expansive, we assume that $A$ is exponential,
i.e., $A = \exp (B)$ for some $B \in \mathbb{R}^{d \times d}$,
so that $A^s = \exp (s B)$ is well-defined for all $s \in \mathbb{R}$;
see Remark~\ref{rem:ExponentialRemark} for additional comments on this assumption.
Given $\varphi \in \mathcal{S} (\mathbb{R}^d)$, $s \in \mathbb{R}$ and $\beta > 0$,
we define the Peetre-type maximal function of $f \in \mathcal{S}' (\mathbb{R}^d)$ as
\[
  \DoubleStar{s} f (x)
  := \sup_{z \in \mathbb{R}^d}
       \frac{|f \ast \varphi_s (x + z)|}
            {(1 + \rho_A (A^{s} z))^{\beta}},
  \quad x \in \mathbb{R}^d,
\]
where $\varphi_s := |\det A|^s \varphi (A^s \cdot)$ and $\rho_A$
is an $A$-homogeneous quasi-norm on $\mathbb{R}^d$;
see Section~\ref{sec:expansive_triebel}.

Our first main result (Theorem~\ref{thm:norm_equiv}) is the following characterization.

\begin{theorem}\label{thm:intro1}
  Let $A \in \mathrm{GL}(d, \mathbb{R})$ be expansive and exponential.
  Suppose that $\varphi \in \mathcal{S} (\mathbb{R}^d)$
  has compact Fourier support and satisfies conditions \eqref{eq:support1} and \eqref{eq:support2}.
  Then, for all $p \in (0,\infty)$, $q \in (0, \infty]$,
  $\alpha \in \mathbb{R}$ and $\beta > \max \{1/p, 1/q\}$,
  the norm equivalences
  \begin{align} \label{eq:equiv_norms_intro}
    \| f \|_{\TL}
    &\asymp \bigg\|
              \bigg(
                \int_{\mathbb{R}}
                \big(
                  |\det A|^{\alpha s}
                   \DoubleStar{s} f
                \big)^q ds
              \bigg)^{1/q}
            \bigg\|_{L^p}
    \asymp \bigg\|
             \bigg(
               \sum_{j \in \mathbb{Z}}
               \big(
                 |\det A|^{\alpha j}
                  \DoubleStar{j} f
               \big)^q
             \bigg)^{1/q}
           \bigg\|_{L^p}
  \end{align}
hold for all $f \in \SP$, with the usual modification for $q = \infty$.
\end{theorem}

Theorem~\ref{thm:intro1} is classical in the setting of \emph{isotropic} Triebel-Lizorkin spaces,
where it has been obtained under varying conditions on the multiplier
$\varphi \in \mathcal{S} (\mathbb{R}^d)$.
Among others, it can be found in
 Triebel \cite{triebel1988characterizations},
Bui, Paluszy\'{n}ski and Taibleson \cite{bui1996maximal,bui1997characterization},
and Rychkov \cite{rychkov1991on, rychkov2001littlewood};
see Ullrich \cite{ullrich2012continuous} for a self-contained overview of these characterizations.

In the setting of \emph{anisotropic} spaces, a maximal characterization of discrete type
(i.e., a characterization involving the right-most term in \eqref{eq:equiv_norms_intro})
was obtained by Farkas \cite{farkas2000atomic} for diagonal dilations $A = \diag(2^{ a_1}, ..., 2^{a_d})$
with anisotropy $(a_1, ..., a_d) \in (0,\infty)^d$.
For general expansive matrices, a discrete maximal characterization
of \emph{inhomogeneous} anisotropic Triebel-Lizorkin spaces
has been obtained by Liu, Yang, and Yuan \cite{liu2019littlewood1}.
However, in contrast to Theorem~\ref{thm:intro1}, the smoothness parameter $\alpha \in \mathbb{R}$
in \cite[Theorem~3.4]{liu2019littlewood1} is restricted to the range $0 < \alpha < \infty$.
In particular, the results in \cite{liu2019littlewood1} do not apply to the Lebesgue spaces $L^p$
for $1<p<\infty$ (which correspond to $\alpha = 0$),
whereas Theorem~\ref{thm:intro1} is applicable to these spaces.

Our proof of Theorem~\ref{thm:intro1} is inspired by the approach
in Rychkov \cite{rychkov2001littlewood} (see also \cite{ullrich2012continuous}),
which combines Fefferman-Stein vector-valued maximal inequalities
with a sub-mean-value property of the convolution products  $(f \ast \varphi_s)_{s \in \mathbb{R}}$
for $f \in \mathcal{S}' (\mathbb{R}^d)$ and $\varphi \in \mathcal{S} (\mathbb{R}^d)$.
This method is a variation of a technique originally due to Str\"omberg and Torchinsky
\cite[Chapter~V]{stroemberg1989weighted}, and is extended here to anisotropic matrix dilations.

In addition to \Cref{thm:intro1}, we also provide a maximal characterization
for the Triebel-Lizorkin sequence spaces; see Theorem~\ref{thm:maximal_sequence}.

\subsection{Wavelet transforms}
\label{sec:coorbit_intro}

The continuous maximal characterization provided by
Theorem~\ref{thm:intro1} can be naturally rephrased in terms of decay properties
of wavelet transforms associated to the quasi-regular representation
\begin{align} \label{eq:dilate_tranlate}
  \pi(x,s) f = |\det A|^{-s/2} f(A^{-s} (\cdot - x)),
  \quad (x,s) \in \mathbb{R}^d \times \mathbb{R}, \; f \in L^2 (\mathbb{R}^d),
\end{align}
of the semi-direct product group $G_A = \mathbb{R}^d \rtimes_A \mathbb{R}$;
see \Cref{sec:admissible} for basic properties.

To be more explicit, given an analyzing vector $\psi \in \mathcal{S} (\mathbb{R}^d)$,
the associated \emph{wavelet transform} of a distribution $f \in \mathcal{S}' (\mathbb{R}^d)$
is the function on $\mathbb{R}^d \times \mathbb{R}$ defined by
\[
  W_{\psi} f :
  G_A \to \mathbb{C}, \quad
  (x,s) \mapsto \langle f, \pi (x,s) \psi \rangle.
\]
Here, we use the sesquilinear dual pairing $\langle f, \varphi \rangle := f(\overline{\varphi})$
for $f \in \mathcal{S}'(\R^d)$ and $\varphi \in \mathcal{S}(\R^d)$.
A function $\psi$ is called \emph{admissible}
if $W_{\psi} : L^2 (\mathbb{R}^d) \to L^{\infty} (G_A)$ defines an isometry into $L^2 (G_A)$.
Given a suitable admissible vector $\psi \in \mathcal{S}(\mathbb{R}^d)$,
a common procedure for constructing an associated function space is by (formally) defining
\begin{align} \label{eq:coorbit_intro}
  \Co (Y)
  = \big\{ f \in \mathcal{S}'(\mathbb{R}^d) / \CalP(\R^d) \; : \; W_{\psi} f \in Y \big\},
\end{align}
where  $Y$ is an adequate
translation-invariant (quasi)-Banach function space on $G_A$.
The function spaces such defined form so-called \emph{coorbit spaces}, see, e.g.,
\cite{feichtinger1989banach,fuehr2015coorbit,velthoven2022quasi,rauhut2007coorbit,christensen2011coorbit}.
Generally, the definition of abstract coorbit spaces in the quasi-Banach range
\cite{rauhut2007coorbit, velthoven2022quasi} requires an additional local property
of the wavelet transform, but we show that it is automatically satisfied in the concrete setting
of the present paper (see \Cref{rem:coorbit_improving} for details).

In this paper we prove several admissibility properties
of functions $\psi \in \mathcal{S} (\mathbb{R}^d)$ and establish various decay and norm estimates
of their associated wavelet transforms.
In particular, it is shown in Proposition~\ref{prop:TL_coorbit} that membership
of $f \in \mathcal{S}'(\mathbb{R}^d)$ in the Triebel-Lizorkin space $\TL$
can be characterized trough decay properties of its wavelet transform $W_{\psi} f$,
in the sense that
\begin{align} \label{eq:TL_coorbit_intro}
    \TLA = \Co_\psi (Y^{\alpha}_{p,q}),
\end{align}
for a Peetre-type function space $Y^{\alpha}_{p,q}$ on $G_A$
and arbitrary $p \in (0, \infty)$, $q \in (0,\infty]$ and $\alpha \in \mathbb{R}$.
Such a coorbit realization is new for non-isotropic Triebel-Lizorkin spaces
and complements the realizations of anisotropic Besov spaces
\cite{cheshmavar2020classification,barrios2011characterizations,bownik2005atomic}
obtained in \cite{fuehr2020coorbit,FuehrVoigtlaenderCoorbitAsDecomposition}.

The isotropic Triebel-Lizorkin spaces have been identified as coorbit spaces \eqref{eq:coorbit_intro}
from the very beginning \cite{groechenig1991describing}.
The function spaces $Y$ used in the identification \cite{groechenig1991describing}
are the tent spaces of Coifman, Meyer and Stein \cite{coifman1985some}.
It was later shown by Ullrich \cite{ullrich2012continuous, liang2012new}
that alternatively one could use
so-called Peetre-type spaces, which allow for
a simpler and more transparent treatment (cf.\ \cite[Section~4.1]{ullrich2012continuous}).
Our use of Peetre spaces in Section~\ref{sec:coorbit} is inspired by \cite{ullrich2012continuous}.

Lastly, it is worth mentioning that the classical papers
\cite{feichtinger1989banach, groechenig1991describing} considered only coorbit spaces
associated with Banach spaces, while for treating Triebel-Lizorkin spaces $\TL$
in the range $\min \{ p,q \} < 1$ it is essential to deal with general \emph{quasi}-Banach spaces.
The framework \cite{rauhut2007coorbit, rauhutcoorbitpreprint} was used for this purpose in \cite{liang2012new}.
However, the theory\footnote{The published paper \cite{rauhut2007coorbit} is restricted
to so-called IN groups, in contrast to the preprint \cite{rauhutcoorbitpreprint}.
The affine group is not an IN group.}
\cite{rauhutcoorbitpreprint, rauhut2007coorbit} is based on an incorrect convolution relation
occurring in \cite{rauhut2007wiener}; in particular, it
does not apply to the affine group (cf.\ \cite[Example~3.13]{velthoven2022quasi}),
although it is used for this purpose in \cite{liang2012new}.
The present paper uses the framework \cite{velthoven2022quasi}
instead of \cite{rauhut2007coorbit}, and it is thus expected that our results in \Cref{sec:admissible}
and \Cref{sec:coorbit} provide a relevant contribution even for isotropic dilations.

\subsection{Molecular decompositions}
\label{sec:molecule_intro}

The identification \eqref{eq:TL_coorbit_intro} of anisotropic Triebel-Lizorkin spaces $\TL$
as suitable coorbit spaces $\Co_\psi (Y^{\alpha}_{p,q})$ (cf.\ Proposition~\ref{prop:TL_coorbit})
enables us to apply general results on the latter spaces
to obtain new molecular decompositions of $\TL$.
However, as was already observed in \cite{gilbert2002smooth},
the classical results \cite{feichtinger1989banach, groechenig1991describing} on coorbit spaces
do not guarantee the same form of localization of both the analyzing and synthesizing functions
as the decomposition theorems of Triebel-Lizorkin spaces in
\cite{gilbert2002smooth, frazier1990discrete, frazier1991littlewood} do.
For this reason, the recent results \cite{romero2020dual, velthoven2022quasi} on molecular decompositions
will be used, which bridge a gap between \cite{feichtinger1989banach, groechenig1991describing}
and \cite{gilbert2002smooth, frazier1990discrete, frazier1991littlewood}.

For $p \in (0,\infty), q \in (0, \infty]$, let $r = \min \{1, p , q\}$.
Given a countable, discrete set $\Gamma \subset G_A$,
a family $(\phi_{\gamma})_{\gamma \in \Gamma}$ of vectors $\phi_{\gamma} \in L^2 (\mathbb{R}^d)$
is a \emph{(coorbit) molecular system} (with respect to the window $\psi$)
if there exists an \emph{envelope} $\Phi \in \WLwr \subset L^1 (G_A)$
satisfying
\begin{align}\label{eq:molecule_intro}
  |W_{\psi} \phi_{\gamma} (g)|
  = | \langle \phi_{\gamma}, \pi(g) \psi \rangle|
  \leq \Phi (\gamma^{-1} g), \quad \gamma \in \Gamma,
  \; g \in G_A;
\end{align}
here, $\WLwr$ denotes a so-called \emph{Wiener amalgam space}
(cf.\ Section~\ref{sub:NormEstimates}).

This notion of molecules depends on a so-called \emph{control weight}
$w = w^{\alpha}_{p,q} : G_A \to [1,\infty)$ for the space $Y^{\alpha}_{p,q}$
occurring in \eqref{eq:TL_coorbit_intro}; see Sections~\ref{sub:ControlWeightConvolutionRelations}
and \ref{sub:MoleculesAndOperators} for details.
Note also that the functions $\phi_{\gamma}$ need not be of the simple form $\pi(\gamma) \phi$
given by translates and dilates of a fixed function (as in \eqref{eq:dilate_tranlate});
rather, the wavelet transform of $\phi_\gamma$ satisfies appropriate size estimates
\emph{as if} it was obtained in this manner.

\begin{theorem}\label{thm:intro2}
  Let $A \!\in\! \mathrm{GL}(d, \mathbb{R})$ be expansive and exponential.
  For $p \in (0, \infty)$, $q \in (0,\infty]$ and $\alpha \in \mathbb{R}$,
  let $r \!=\! \min\{1,p,q\}$, $\alpha' = \alpha + 1/2-1/q$, and $\beta > \max \{ 1/p, 1/q \}$.

  Suppose $\psi \in L^2 (\mathbb{R}^d)$ is an admissible vector satisfying
  $W_{\psi} \psi \in \WLwr$ for the standard control weight
  $w = w^{-\alpha',\beta}_{p,q} : G_A \to [1,\infty)$
  defined in Lemma~\ref{lem:ControlWeights}.
  Moreover, suppose $W_\varphi \psi \in \WLwr$ for some
  (thus all) admissible $\varphi \in \mathcal{S}_0(\R^d)$.
  Then there exists a compact unit neighborhood $U \subset G_A$ such that,
  for any $\Gamma \subset G_A$ satisfying
  \begin{align} \label{eq:udense_relatively}
   G_A = \bigcup_{\gamma \in \Gamma} \gamma U \quad
   \text{and} \quad \sup_{g \in G_A} \# (\Gamma \cap g U) < \infty,
   \end{align}
  there exist two molecular systems
  $(\phi_{\gamma})_{\gamma \in \Gamma} \subset L^2(\R^d)$
  and $(f_{\gamma})_{\gamma \in \Gamma} \subset L^2(\R^d)$ such that
  any $f \in \TL$  can be represented as
  \[
    f = \sum_{\gamma \in \Gamma} \langle f, \pi(\gamma) \psi \rangle \phi_{\gamma}
      = \sum_{\gamma \in \Gamma} \langle f, \phi_{\gamma} \rangle  \pi(\gamma) \psi
    \quad \text{and} \quad
    f = \sum_{\gamma \in \Gamma} \langle f, f_{\gamma} \rangle f_{\gamma},
  \]
  with unconditional convergence in the weak-$\ast$ topology
  of $\mathcal{S}'(\mathbb{R}^d) / \mathcal{P}(\mathbb{R}^d)$.

  (The dual pairings $\langle f, \pi(\gamma) \psi \rangle$
  and $\langle f, \phi_\gamma \rangle$ are defined suitably;
  see Definition~\ref{def:ExtendedDualPairing}.)
\end{theorem}

The novelty of Theorem~\ref{thm:intro2} is that it applies to possibly irregular sets $\Gamma$%
---i.e., arising from non-lattice translations---and that \emph{both}
$\{ \pi(\gamma)\psi  :  \gamma \in \Gamma \}$ and $\{ \phi_{\gamma} : \gamma \in \Gamma \}$
are molecular systems. It resembles the classical results for lattice translations
by Frazier and Jawerth \cite[Remark~9.17]{frazier1990discrete}
and Gilbert, Han, Hogan, Lakey, Weiland, and Weiss \cite[Theorem~1.5]{gilbert2002smooth},
and the work of Ho \cite{ho2003frames} for general expansive dilations.
In contrast to \Cref{thm:intro2}, the notion of molecules
used in \cite{frazier1990discrete,gilbert2002smooth,bownik2006atomic,ho2003frames}
is defined via explicit smoothness and moment conditions
rather than decay estimates of their wavelet transform as in  \Cref{eq:molecule_intro}.
For comparison, we provide explicit smoothness criteria
for coorbit molecular systems in Section~\ref{sec:criteria_molecules}.

It should be mentioned that for specific vectors $\psi$ and particular construction methods,
the validity of wavelet frame expansions in Hardy and Lebesgue spaces have,
among others, been obtained by Bui and Laugesen \cite{bui2013wavelet}
and Cabrelli, Molter and Romero \cite{cabrelli2013non}.
The results in \cite{bui2013wavelet,cabrelli2013non} provide criteria and constructions
that work for index sets $\Gamma$ satisfying \eqref{eq:udense_relatively}
for \emph{some} neighborhood $U$, whereas Theorem~\ref{thm:intro2} above requires $U$ to be
sufficiently small.
We mention that even for a molecular frame for $L^2 (\mathbb{R}^d)$, the extension
of the canonical $L^2$-frame expansions to Hardy and Lebesgue spaces is non-automatic in general,
and that such frames might fail to yield decompositions of $L^p$ for $p \neq 2$,
see, e.g., Tao \cite{tao} and Tchamitchian \cite{tchamitchian1987biorthogonalite}.

Lastly, we complement Theorem~\ref{thm:intro2} with a dual result on Riesz sequences. \Cref{thm:intro3}
shows that a solution to the interpolation or moment problem in discrete sequence spaces
$\dot{\mathbf{p}}^{-\alpha', \beta}_{p, q} (\Gamma) \leq \mathbb{C}^{\Gamma}$
associated to a discrete $\Gamma \subset G_A$ and the Triebel-Lizorkin spaces $\TL$ can be obtained using
molecular dual Riesz sequences; see \Cref{def:peetre_seq} and  Remark~\ref{rem:sequence_regular} for details.

\begin{theorem}\label{thm:intro3}
 Under the same assumptions of Theorem~\ref{thm:intro2}, the following holds:

  There exists a compact unit neighborhood $U \subset G_A$ such that,
  for any $\Gamma \subset G_A$ satisfying
  \begin{align} \label{eq:separated_intro3}
    \gamma U \cap \gamma' U = \emptyset,
    \quad \text{for all } \gamma, \gamma' \in \Gamma \text{ with } \gamma \neq \gamma',
  \end{align}
  there exists a molecular system
  \(
    (\phi_{\gamma})_{\gamma \in \Gamma}
    \subset \overline{\Span \{ \pi(\gamma) \psi : \gamma \in \Gamma \}}
    \subset L^2(\R^d)
  \)
  such that the moment problem
  \begin{align}\label{eq:momentproblem}
    \langle f, \pi (\gamma) \psi \rangle = c_{\gamma},
    \quad \gamma \in \Gamma,
  \end{align}
  admits the solution $f: = \sum_{\gamma \in \Gamma} c_{\gamma} \phi_{\gamma} \in \TL$ for any given
  \(
    (c_{\gamma})_{\gamma \in \Gamma}
    \in \dot{\mathbf{p}}^{- \alpha', \beta}_{p, q} (\Gamma)
    \leq \mathbb{C}^{\Gamma}.
  \)
\end{theorem}

Theorem~\ref{thm:intro3} seems to be the first result on Riesz sequences in
anisotropic Triebel-Lizorkin spaces
and it is new even for regular index sets arising from lattice translations.
We mention that for regular index sets, the sequence space appearing in \Cref{thm:intro3}
coincides with the standard anisotropic Triebel-Lizorkin sequence spaces
defined in \cite{bownik2006atomic}; see Remark~\ref{rem:sequence_regular}.

\subsection{General notation}

We write $\PosPart{s} := \max\{ 0, s \}$ and $\NegPart{s} := - \min \{ 0, s \}$
for $s \in \mathbb{R}$.

Given functions $f,g : X \to [0,\infty)$, we write $f \lesssim g$ if
there exists $C>0$ satisfying $f(x) \leq C g(x)$ for all $x \in X$.
We write $f \asymp g$ for $f \lesssim g$ and $g \lesssim f$.  The
notation $\lesssim_\alpha$ is sometimes used to indicate that the implicit
constant depends on a quantity $\alpha$.
If $G$ is a group, we write $f^{\vee} (x) = f(x^{-1})$  for $x \in G$.
The characteristic function of $\Omega \subset X$ is denoted by $\mathds{1}_{\Omega}$.
For a measurable $\Omega \subset \mathbb{R}^d$,
its Lebesgue measure is denoted by $\Lebesgue{\Omega}$.

For a matrix $A \in \R^{d \times d}$, its transpose is denoted by $A^*$.
The norm $\|A\|_{\infty}$ denotes the operator norm of the induced map
$A : \mathbb{R}^d \to \mathbb{R}^d$.
The function $\| \cdot   \| : \mathbb{R}^d \to \mathbb{R}$ will denote the Euclidean
norm on $\mathbb{R}^d$.

The space of Schwartz functions will be denoted by
$\mathcal{S} (\mathbb{R}^d)$ and the space of tempered distributions
by $\mathcal{S}' (\mathbb{R}^d)$. Moreover, the set
$\mathcal{P}(\mathbb{R}^d)$ denotes the space of all polynomials of
$d$ real variables, and $\SP$ denotes the space of equivalence classes
of tempered distributions modulo polynomials.  The Fourier transform
$\mathcal{F} : \mathcal{S}(\mathbb{R}^d) \to
\mathcal{S}(\mathbb{R}^d)$ is normalized as
$\widehat{f} (\xi) = \int_{\mathbb{R}^d} f(x) e^{-2\pi ix \cdot \xi} \;
dx$.  Its inverse $\mathcal{F}^{-1} f := \widehat{f}(- \, \cdot \,)$ will also
be denoted by $\widecheck{f}$.
Similar notations will be used for the unitary Fourier-Plancherel transform
$\mathcal{F} : L^2 (\mathbb{R}^d) \to L^2 (\mathbb{R}^d)$ and its inverse.
For $f : \R^d \to \CC$ and $y \in \R^d$, we define $T_y f : \R^d \to \CC, x \mapsto f(x - y)$.

Lastly, if $V$ is a topological vector space consisting of (equivalence classes of)
functions such that the conjugation map $V \to V, \; \varphi \mapsto \overline{\varphi}$
is a well-defined, continuous map, then the associated map
\[
  V' \to V^\ast, \quad f \mapsto \underline{f}
  \qquad \text{with} \qquad
  \underline{f}(\varphi) := f(\overline{\varphi})
\]
between the dual space $V'$ and the anti-dual space $V^\ast$ is a canonical isomorphism.
In this setting, we will not distinguish between $f \in V'$ and $\underline{f} \in V^\ast$.
In particular, the dual pairings $\langle \cdot,\cdot \rangle = \langle \cdot,\cdot \rangle_{V',V}$
and $\langle \cdot,\cdot \rangle = \langle \cdot,\cdot \rangle_{V^\ast,V}$ will always be taken
to be \emph{anti}-linear in the second component, i.e., $\langle f,\varphi \rangle := f(\overline{\varphi})$ for $f \in V'$
and $\langle f, \varphi \rangle := f(\varphi)$ for $f \in V^\ast$.
The two most important cases where this applies is for $V = \mathcal{S}(\R^d)$ and
$
  V
  = \mathcal{S}_0(\R^d)$ (cf. \Cref{def:VanishingMomentSchwartz}).

\section{Expansive matrices and Triebel-Lizorkin spaces}
\label{sec:expansive_triebel}

This section provides background on expansive matrices and associated
function spaces.

\subsection{Expansive matrices}
\label{sec:expansive}

A matrix $A \in \R^{d \times d}$ is called \emph{expansive}
if $\min_{\lambda \in \sigma(A)} |\lambda|>1$,
where $\sigma(A) \subset \CC$ denotes the spectrum of $A$.
The significance of an expansive matrix is that it induces the structure
of a space of homogeneous type on $\mathbb{R}^d$;
see \cite{CoifmanSpacesOfHomogeneousType,coifman1977extensions} for background.

The following lemma is collected from \cite[Definitions~2.3 and 2.5]{bownik2003anisotropic}
and \cite[Lemma~2.2]{bownik2003anisotropic}.

\begin{lemma}[\cite{bownik2003anisotropic}]  \label{lem:QuasiNorm}
Let $A \in \mathrm{GL}(d, \mathbb{R})$ be expansive.
\begin{enumerate}[(i)]
  \item There exist an ellipsoid $\Omega_A$ (i.e., $\Omega_A$ is the image of the
        open Euclidean unit ball under an invertible matrix) and $r > 1$ such that
        \[
          \Omega_A \subset r\Omega_A \subset A\Omega_A
        \]
        and $\Lebesgue{\Omega_A} = 1$.
        The map $\rho_A : \mathbb{R}^d \to [0,\infty)$ given by
        \begin{align}\label{eq:step_norm}
        \rho_A (x)
        = \begin{cases}
            |\det A|^j, \quad & \text{if} \;\; x \in A^{j+1} \Omega_A \setminus A^j \Omega_A, \\
            0, \quad          & \text{if} \;\; x = 0,
          \end{cases}
        \end{align}
        is called the \emph{step homogeneous quasi norm} associated to $A$.
        It is measurable and there exists $C \geq 1$ such that it satisfies the following properties:
        \begin{equation}
          \begin{alignedat}{3}
            \rho_A (-x)  & = \rho_A(x),
            && \quad x \in \mathbb{R}^d ,\\
            \rho_A (x)   &> 0,
            && \quad x \in \mathbb{R}^d \setminus \{0\} , \\
            \rho_A (Ax)  &= |\det A|  \rho_A (x),
            && \quad x \in \mathbb{R}^d , \\
            \rho_A (x+y) &\leq C \big( \rho_A (x) + \rho_A (y) \big),
            && \quad x,y \in \mathbb{R}^d.
          \end{alignedat}
          \label{eq:QuasiNormProperties}
        \end{equation}

  \item Define $d_A : \R^d \times \R^d \to [0,\infty), (x,y) \mapsto \rho_A(x-y)$
        and let $\Measure$ denote the Lebesgue measure on $\R^d$.
        Then the triple $(\mathbb{R}^d, d_A, \Measure)$
        is a \emph{space of homogeneous type}.
\end{enumerate}
\end{lemma}

For $y \in \mathbb{R}^d$ and $r>0$, the $d_A$-ball will be denoted by
$B_{\rho_A}(y, r) := \{ x \in \R^d \colon \rho_A (x-y) < r \}$.

The following lemma shows that the homogeneous quasi-norm can be
estimated from above and below by (powers of) the Euclidean norm;
cf.\ \cite[Equation~(2.7) and Lemma~3.2]{bownik2003anisotropic}.

\begin{lemma}[\cite{bownik2003anisotropic}]  \label{lem:quasi-norm-bounds}
  Let $A \in \mathrm{GL}(d, \mathbb{R})$ be expansive.
  Let $\lambda_-, \lambda_+$ satisfy
  $1 < \lambda_- < \min_{\lambda \in \sigma(A)} |\lambda|$ and
  $\lambda_+ > \max_{\lambda \in \sigma(A)} |\lambda|$.
  Define
  \[
  \zeta_- := \frac{\ln \lambda_-}{\ln |\det A|} \in (0, \tfrac{1}{d})
    \quad \text{and} \quad
     \zeta_+ := \frac{\ln \lambda_+}{\ln |\det A|} \in (\tfrac{1}{d}, \infty).
  \]
  Then there exists $C \geq 1$ such that for every $x \in \R^d$, we have
  \begin{align*}
    C^{-1}  [\rho_A (x)]^{\zeta_-} &\leq \| x \| \leq C  [\rho_A(x)]^{\zeta_+},
    \quad \text{if }\rho_A(x) \geq 1, \\
    C^{-1}  [\rho_A (x)]^{\zeta_+} &\leq \| x \| \leq C  [\rho_A(x)]^{\zeta_-},
    \quad \text{if }\rho_A(x) \leq 1.
  \end{align*}
\end{lemma}

We will also need the following fact about the integrability of powers
of the quasi norm $\rho_A$.

\begin{lemma}\label{lem:QuasiNormIntegrability}
  Suppose $A \in \mathrm{GL}(d, \mathbb{R})$ is expansive.
  Then for all $\eps > 0$ we have
  \begin{equation*}
    \int_{B_{\rho_A}(0,1)} [\rho_A(x)]^{\eps-1} dx < \infty
    \quad \text{and} \quad
    \int_{\R^d \setminus B_{\rho_A}(0,1)} [\rho_A(x)]^{-1-\eps} dx < \infty.
  \end{equation*}
\end{lemma}

\begin{proof}
  Directly from the definition of $\rho_A$, we see
  \begin{align*}
    \int_{\R^d \setminus B_{\rho_A}(0,1)}
      [\rho_A(x)]^{-1-\eps}
    d x
    & = \sum_{j=0}^\infty
          |\det A|^{-j(1+\eps)}  \Lebesgue{A^{j+1} \Omega_A \setminus A^j \Omega_A} \\
    & = \sum_{j=0}^\infty
          |\det A|^{-\eps j}  \Lebesgue{A \Omega_A \setminus \Omega_A}
      < \infty ,
  \end{align*}
  since $|\det A| > 1$.
  The proof for $\int_{B_{\rho_A}(0,1)} [\rho_A(x)]^{\eps-1} dx$ is similar.
\end{proof}

\subsection{Exponential matrices}
\label{sec:exponential}

A matrix $A \in \R^{d \times d}$ is called \emph{exponential}
if $A = \exp (B)$ for a matrix $B \in \R^{d \times d}$;
here, $\exp(B) = \sum_{n=0}^\infty B^n/n!$ denotes the usual matrix exponential.
If $A$ is expansive and has only positive eigenvalues,
then $A$ is exponential by \cite[Lemma~7.8]{cheshmavar2020classification}.
See \cite[Theorem~1]{culver1966on} for a precise characterization.

For an exponential matrix ${A = \exp(B)}$, the power $A^s = \exp(sB)$ is defined
for all $s \in \mathbb{R}$.
We have $\det A^s = \det (\exp(sB)) = e^{\mathrm{tr}(sB)} = (e^{\mathrm{tr}(B)})^s = (\det A)^s$,
see, e.g., \cite[Theorem~2.12]{HallLieGroups}.
The family $\{ A^s : s \in \mathbb{R} \}$ forms a
continuous one-parameter subgroup of $\mathrm{GL}(d, \mathbb{R})$.

The next lemma provides norm bounds for the powers $A^s$ of an exponential matrix $A$.
For integral powers, these bounds are folklore%
\footnote{Alternatively, they can be easily derived from the spectral radius formula.};
see, e.g., \cite[Equations~(2.1) and (2.2)]{bownik2003anisotropic}.

\begin{lemma}\label{lem:expansive_cont_powers}
  Let $A \in \mathrm{GL}(d, \mathbb{R})$ be expansive and exponential.
  Let $\lambda_-, \lambda_+$ be constants such that
  $1 < \lambda_- < \min_{\lambda \in \sigma(A)} |\lambda|$ and
  $\lambda_+ > \max_{\lambda \in \sigma(A)} |\lambda|$.
  Then there exists $C \geq 1$ such that
  \begin{align*}
    &C^{-1} \, \lambda_-^s \, \| x \| \leq \| A^s x \| \leq C \, \lambda_+^s \, \| x \|,
    \quad s \geq 0, \\
    &C^{-1} \, \lambda_+^s \, \| x \| \leq \| A^s x \| \leq C \, \lambda_-^s \, \| x \|,
    \quad s \leq 0,
  \end{align*}
  for all $x \in \mathbb{R}^d$.
\end{lemma}

\begin{proof}
  Since $t \mapsto A^t$ is continuous,
  there exists $C_A > 0$ such that $\|A^t\|_\infty \leq C_A$ for $t \in [-1,1]$.
  For $s \geq 0$, we write $ s = k + t$ with $k \in \N_0$ and
  $t \in [0,1)$, and use the result for integral powers
  \cite{bownik2003anisotropic} to conclude
  \begin{equation*}
    \|A^sx\|
    = \|A^tA^k x\|
    \leq \|A^t\|_{\infty} \, \|A^kx\|
    \leq C_A \, C \, \lambda_{+}^k \, \|x\|
    \leq C_A \, C \, \lambda_{+}^s \, \|x\| .
  \end{equation*}
  Similarly,
  \[
    C_A \, \| A^s x \|
    \geq \| A^{-t} \|_{\infty} \, \| A^s x \|
    \geq \| A^{-t} A^s x \|
    =    \| A^k x \|
    \geq C^{-1} \, \lambda_-^k \, \| x \|
    \geq (C \, \lambda_-)^{-1} \, \lambda_-^s \, \| x \| .
  \]
  The estimate for $s \leq 0$ is shown using similar arguments.
\end{proof}

\begin{corollary}\label{cor:quasi-norm_bound}
  Let $A \in \mathrm{GL}(d, \mathbb{R})$ be expansive and exponential.
  Then there exists $C \geq 1$ such that
  \begin{equation*}
    C^{-1} |\det A|^s \rho_A (x)
    \leq \rho_A (A^s x)
    \leq C  |\det A|^s \, \rho_A(x)
    \quad  x \in \R^d, \, s \in \R.
  \end{equation*}
\end{corollary}

\begin{proof}
  Due to the $A$-homogeneity of $\rho_A$, it suffices to verify the
  claim for ${x \in A \, \Omega_A \setminus \Omega_A}$
  (with $\Omega_A$ as in Lemma~\ref{lem:QuasiNorm}) and $s \in [0,1]$.
  By Lemma~\ref{lem:expansive_cont_powers} and by the compactness of
  ${\overline{A \, \Omega_A \setminus \Omega_A} \subset \R^d \setminus \{0\}}$,
  there exist $R_1,R_2 > 0$ such that
  \begin{equation*}
    R_1 \leq C^{-1} \, \lambda_-^s \, \| x \|
        \leq \| A^s x \|
        \leq C \, \lambda_+^s \, \| x \|
        \leq R_2
  \end{equation*}
  uniformly for all $x \in A\Omega_A \setminus \Omega_A$ and $s \in [0,1]$.
  Furthermore, there exists $k \in \N$ such that
  $A^{-k}\Omega_A \cap \{y \in \R^d : \|y\| \geq R_1 \} = \emptyset$.
  Thus, we see for $s \in [0,1]$ and $x \in A \Omega_A \setminus \Omega_A$
  that $A^s x \notin A^{-k} \Omega_A$ and hence
  \begin{equation*}
    \rho_A(A^sx)
    \geq |\det A|^{-k}
    = |\det A|^{-k} \rho_A(x)
    \geq |\det A|^{-k-1} |\det A|^s \rho_A(x),
  \end{equation*}
  where we have used that $\rho_A(x) = 1$ for all $x \in A\Omega_A \setminus \Omega_A$.
  This gives the lower bound with $C:= |\det A|^{k+1} \geq 1$.
  The upper bound follows by replacing $x$ with $A^{-s}x$.
\end{proof}

An alternative proof of \Cref{cor:quasi-norm_bound} can be obtained by using a homogeous quasi-norm associated to
the continuous one-parameter group $\{A^s : s \in \mathbb{R} \}$ (cf. \cite[Proposition 1-9]{stein1978problems}) and
the equivalence of all homogeneous quasi-norms associated to $A$ (cf. \cite[Lemma 2.4]{bownik2003anisotropic}).

\subsection{Analyzing vectors}
\label{sec:analyzing}

Let $A \in \mathrm{GL}(d, \mathbb{R})$ be expansive.
Suppose $\varphi \in \mathcal{S} (\mathbb{R}^d)$ is such that
 $\varphi$ has compact Fourier support
\begin{align}\label{eq:analyzing_support}
  \supp \widehat{\varphi}
  := \overline{\{ \xi \in \mathbb{R}^d : \widehat{\varphi}(\xi) \neq 0 \}}
  \subset \mathbb{R}^d \setminus \{0\}
\end{align}
and satisfies
\begin{align}\label{eq:analyzing_positive}
  \sup_{j \in \mathbb{Z}}
    \big| \widehat{\varphi} ((A^*)^j \xi) \big| > 0,
  \quad \xi \in \mathbb{R}^d \setminus \{0\}.
\end{align}
Then the function $\psi \in \SC(\R^d)$
defined through its Fourier transform as
\begin{equation*}
  \widehat\psi(\xi)
  = \begin{cases}
      \overline{\widehat\varphi(\xi)} / \sum_{k \in \Z}
                                          |\widehat{\varphi} ((A^*)^k \xi)|^2,
      \quad
      & \text{if} \;\; \xi \in \R^d \setminus \{0\}, \\
      0,
      \quad
      & \text{if} \;\; \xi = 0,
    \end{cases}
\end{equation*}
is well-defined and satisfies
\begin{align}\label{eq:analyzing_calderon}
  \sum_{j \in \mathbb{Z}}
    \widehat{\varphi} ((A^*)^j \xi) \, \widehat{\psi} ((A^*)^j \xi)
  = 1,
  \quad \xi \in \mathbb{R}^d \setminus \{0\}.
\end{align}
We refer to \cite[Lemma~3.6]{bownik2006atomic} for more details.

\subsection{Triebel-Lizorkin spaces}
\label{sub:TLSpaces}

Let $A \in \mathrm{GL}(d, \mathbb{R})$ be expansive and
suppose that $\varphi \in \Schwartz(\R^d)$
has compact Fourier support satisfying
\eqref{eq:analyzing_support} and \eqref{eq:analyzing_positive}.
For given $\alpha \in \mathbb{R}$, $0 < p < \infty$ and $0 < q \leq \infty$,
the associated (homogeneous) anisotropic Triebel-Lizorkin space $\TL = \TL(A,\varphi)$
is defined as in \cite{bownik2006atomic} as the set of all $f \in \SP$ for which
\[
  \| f \|_{\TL}
  := \bigg\|
       \bigg(
         \sum_{j \in \mathbb{Z}}
           (|\det A|^{j\alpha} |f \ast \varphi_j |)^q
       \bigg)^{1/q}
     \bigg\|_{L^p}
  < \infty,
 \]
where $\varphi_j := |\det A|^j \, \varphi(A^j \cdot)$,
with the usual modification for $q = \infty$.

As shown in \cite[Proposition~3.2]{bownik2006atomic},
the inclusion map $\TL \hookrightarrow \SP$ is continuous
and $\TL$ is complete with respect to the quasi-norm $\| \cdot \|_{\TL}$.
Moreover, \cite[Corollary~3.7]{bownik2006atomic} shows that the space $\TL$
is independent of the choice of $\varphi$;
we will thus simply write $\TLA$ instead of $\TL(A,\varphi)$.

The sequence space $\TLseq = \TLseqA$ on $\mathbb{Z} \times \mathbb{Z}^d$
associated to $\TL$ is defined as the collection of all
$c \in \mathbb{C}^{\mathbb{Z} \times \mathbb{Z}^d}$ satisfying
\begin{align} \label{eq:TL_sequence}
  \| c \|_{\TLseq}
  := \bigg\|
       \bigg(
         \sum_{j \in \mathbb{Z}}
           \sum_{k \in \mathbb{Z}^d}
             \big(
              |\det A|^{j (\alpha + 1/2)}
              |c_{j,k}|
              \mathds{1}_{A^{-j} ([0,1)^d +k)}
             \big)^q
       \bigg)^{1/q}
     \bigg\|_{L^p}
  < \infty,
\end{align}
with the usual modification for $q = \infty$.

\subsection{Anisotropic Hardy spaces}
\label{sec:hardy}

Denoting by $H^p_A$ the anisotropic Hardy space introduced in
\cite{bownik2003anisotropic}, it follows by
\cite[Theorem~7.1]{bownik2007anisotropic} and \cite[Remark on p.~16]{bownik2003anisotropic} that
\begin{align*}
  H^p_A &= \dot{\mathbf{F}}^{0}_{p, 2} (A), \quad p \in (0,1] \\
  L^p = H^p_A   &= \dot{\mathbf{F}}^{0}_{p, 2} (A), \quad p \in (1,\infty).
\end{align*}
Two expansive matrices $A_1, A_2 \in \mathrm{GL}(d, \mathbb{R})$ are said to be \emph{equivalent} if
$H^p_{A_1} = H^p_{A_2}$ for all $p \in (0,1]$. Given an expansive $A_1$,
there exists an equivalent matrix $A_2$ with all eigenvalues positive
and such that $\det A_2 = |\det A_1|$; see \cite[Lemma~7.7]{cheshmavar2020classification}
and \cite[Theorem~2.3 and Lemma~3.6]{bownik2020pde}.
Recall that such a matrix $A_2$ is exponential (cf.\ Section~\ref{sec:exponential}).

\section{Maximal function characterizations}
\label{sec:maximal}

This section provides maximal function characterizations of Triebel-Lizorkin spaces.
In Section \ref{sub:MaximalFunction} we provide preliminaries on maximal functions.
The characterizations of distribution and sequence spaces
will be proven in Sections~\ref{sec:distribution} and \ref{sec:sequence}, respectively.

\subsection{Anisotropic maximal functions}
\label{sub:MaximalFunction}

Let $A \in \mathrm{GL}(d, \mathbb{R})$ be expansive.
For $f : \R^d \to \CC$ measurable, the \emph{(anisotropic) Hardy-Littlewood maximal operator}
$M_{\rho_A}$ is defined as
\begin{equation}
  M_{\rho_A} f (x)
  = \sup_{B \ni x}
      \frac{1}{\Lebesgue{B}}
      \int_B
        |f(y)|
      \; dy,
  \qquad x \in \mathbb{R}^d,
  \label{eq:HardyLittlewoodMaximalFunctionDefinition}
\end{equation}
where the supremum is taken over all $\rho_A$-balls $B = B_{\rho_A}(y, r)$ that contain $x$.

The following simple observation is central for the remainder of this article.

\begin{lemma} \label{lem:maximal_dilation}
  Let $A \in \mathrm{GL}(d, \mathbb{R})$ be expansive.
  For $f : \R^d \to \CC$ measurable, it holds
  \begin{equation}
    M_{\rho_A} [f \circ A^j] = [M_{\rho_A} f] \circ A^j,
    \qquad  j \in \Z .
    \label{eq:MaximalFunctionDilation}
  \end{equation}
\end{lemma}

\begin{proof}
For $z \in \R^d$, the property
$A^j z \in B_{\rho_A} (y, r)$ is equivalent to $ z \in B_{\rho_A}(A^{-j} y, r / |\det A|^j)$.
Hence, the substitutions $z = A^{-j} y$ and $s = r / |\det A|^j$
and the change-of-variable ${v = A^{-j} w}$ show
\begin{align*}
  (M_{\rho_A} f)(A^j x)
  & = \sup_{\substack{y \in \R^d, r > 0 \\ A^j x \in B_{\rho_A} (y,r)}}
        \frac{1}{\Lebesgue{B_{\rho_A}(y,r)}}
        \int_{B_{\rho_A}(y,r)} |f(w)| \, d w \\
  & = \sup_{\substack{z \in \R^d, s > 0 \\ x \in B_{\rho_A}(z, s)}}
        \frac{1}{\Lebesgue{B_{\rho_A}(A^j z, |\det A|^j s)}}
        \int_{B_{\rho_A}(A^j z,|\det A|^j s)} |f(w)| \, d w \\
  & = \sup_{\substack{z \in \R^d, s > 0 \\ x \in B_{\rho_A}(z, s)}}
        \frac{|\det A|^j}{\Lebesgue{B_{\rho_A}(A^j z, |\det A|^j s)}}
        \int_{B_{\rho_A}(z,s)} |f(A^j v)| \, d v \\
   &= (M_{\rho_A} [f \circ A^j]) (x) ,
\end{align*}
as desired.
\end{proof}

A further central property is the vector-valued Fefferman-Stein inequality
\cite{fefferman1971some}, in the form stated in the following theorem.
It follows, e.g., from \cite[Theorem~1.2]{grafakos2009vector},
by using that $(\R^d, d_A, \Measure)$ is a space of homogeneous type.

\begin{theorem}[\cite{grafakos2009vector}]\label{thm:fefferman-stein}
  Let $A \in \GL(d,\R)$ be expansive.
  For $p \in (1, \infty), q \in (1, \infty]$, there exists $C=C(p,q,A,d) > 0$ such that
  \[
    \bigg\|
      \bigg(
        \sum_{i \in \mathbb{N}}
          [ M_{\rho_A} f_i ]^q
      \bigg)^{1/q}
    \bigg\|_{L^p }
    \leq C \bigg\|
             \bigg(
               \sum_{i \in \mathbb{N}}
                 |f_i |^q
             \bigg)^{1/q}
           \bigg\|_{L^p}
  \]
  for any sequence of measurable functions $f_i : \R^d \to \CC$, $i \in \N$,
  with the usual modification for $q = \infty$.
\end{theorem}

The following majorant property of the anisotropic maximal operator
can be found in \cite[Lemma~3.1]{borup2008on} in a slightly different setting.
Nevertheless, the proof given in \cite{borup2008on} applies verbatim in our setting.

\begin{lemma}[\cite{borup2008on}]\label{lem:majorant}
  Let $\theta : [0,\infty) \to [0,\infty)$ be non-increasing, and assume that
  ${\Theta : \R^d \to [0,\infty)}$ given by $\Theta(x) = \theta(\rho_A (x))$ is integrable.
  Suppose that $g \in L^1 (\mathbb{R}^d)$ satisfies $|g(x)| \leq \Theta(x)$
  for almost all $x \in \R^d$.
  Then, for $f \in L^1 (\mathbb{R}^d)$,
  \[
    |(f \ast g) (x) | \leq \|\Theta \|_{L^1} M_{\rho_A} f(x)
  \]
  for all $x \in \mathbb{R}^d$.
\end{lemma}

Given an exponential matrix $A \in \mathrm{GL}(d, \mathbb{R})$ and $s \in \mathbb{R}$,
we define the dilation of a function $\varphi : \R^d \to \CC$ by
$
  \varphi_s (x) := |\det A|^s  \varphi(A^s x).
$
For $\beta > 0$,
the \emph{Peetre-type maximal function} of $f \in \mathcal{S}' (\mathbb{R}^d)$
with respect to $\varphi \in \Schwartz(\R^d)$ is defined as
\begin{align}\label{eq:peetre_maximal}
  \DoubleStar{s} f (x)
  := \sup_{z \in \mathbb{R}^d}
       \frac{|(f \ast \varphi_{s}) (x+z)|}
            {(1 + \rho_{A}(A^s z))^{\beta}}
   = \esssup_{z \in \mathbb{R}^d}
       \frac{|(f \ast \varphi_{s}) (x+z)|}
            {(1 + \rho_{A}(A^s z))^{\beta}},
  \quad x \in \mathbb{R}^d;
\end{align}
see Lemma~\ref{lem:PeetreSupEssup} for the validity of the second equality
for the step homogeneous quasi-norm $\rho_A$.
If $A$ is not exponential, we define $\DoubleStar{s}$ also by \eqref{eq:peetre_maximal},
but only for $s \in \Z$.

The Peetre-type maximal function and the Hardy-Littlewood operator are related by
Peetre's inequality, cf.\ \cite[Lemma~3.4]{bownik2006atomic} for a proof.

\begin{lemma}[Anisotropic Peetre inequality]\label{lem:peetre}
  Let $K \subset \mathbb{R}^d$ be compact and $\beta > 0$.
  There exists $C = C(K,\beta,A) > 0$ such that for any $g \in \mathcal{S}'(\mathbb{R}^d)$
  with $\supp \widehat{g} \subset K$, we have
  \begin{align} \label{eq:peetre_inequality}
    \sup_{z \in \mathbb{R}^d}
      \frac{|g (x-z)|}{(1+\rho_{A} (z))^{\beta}}
    \leq C \bigl[(M_{\rho_A} |g|^{1/\beta})(x)\bigr]^{\beta}
  \end{align}
  for all $x \in \mathbb{R}^d$.
\end{lemma}

The expression $g(x)$  in \eqref{eq:peetre_inequality} makes sense,
since every tempered distribution with compact Fourier support is given by
(integration against) a smooth function, cf.\ \cite[Theorem~7.23]{RudinFA}.

\subsection{Function spaces}
\label{sec:distribution}

The following theorem is one of the main results of this paper.
It provides an anisotropic extension of corresponding results in
\cite{ullrich2012continuous, bui1996maximal, triebel1988characterizations}.

\begin{theorem}\label{thm:norm_equiv}
  Let $A \in \mathrm{GL}(d, \mathbb{R})$ be expansive and exponential.
  Assume that $\varphi \in \Schwartz(\R^d)$ has compact Fourier support
  and satisfies \eqref{eq:analyzing_support} and \eqref{eq:analyzing_positive}.
  Then, for all $p \in (0,\infty)$, $q \in (0, \infty]$, $\alpha \in \mathbb{R}$
  and $\beta > \max \{1/p, 1/q\}$, the norm equivalences
  \begin{align}\label{eq:norm_equiv}
    \| f \|_{\TL}
    & \asymp \bigg\|
               \bigg(
                 \int_{\mathbb{R}}
                   \big(
                     |\det A|^{\alpha s}
                      \DoubleStar{s} f
                   \big)^q
                 ds
               \bigg)^{1/q}
             \bigg\|_{L^p}
    &\asymp \bigg\|
              \bigg(
                \sum_{j \in \mathbb{Z}}
                  \big(
                    |\det A|^{\alpha j}
                     \DoubleStar{j} f
                  \big)^q
              \bigg)^{1/q}
            \bigg\|_{L^p}
  \end{align}
  hold for all $f \in  \SP$, with the usual modifications for $q = \infty$.

  (The function $\DoubleStar{s}f : \mathbb{R}^d \to [0,\infty]$  is well-defined
  for $f \in \SP$, since $\varphi$ has infinitely many vanishing moments and hence
  $P \ast \varphi_s = 0$ for every $P \in \CalP(\R^d)$.)
\end{theorem}

\begin{remark}\label{rem:ExponentialRemark}
  Let $A \in \mathrm{GL}(d, \mathbb{R})$ be expansive.
  \begin{enumerate}[(a)]
  \item The proof of Theorem~\ref{thm:norm_equiv} shows that the characterization
        \[
          \| f \|_{\TL}
          \asymp \bigg\|
                   \bigg(
                     \sum_{j \in \mathbb{Z}}
                     \big(
                       |\det A|^{\alpha j}
                        \DoubleStar{j} f
                     \big)^q
                   \bigg)^{1/q}
                 \bigg\|_{L^p},
          \quad f \in \SP,
        \]
        does not require $A$ to be exponential.
        Instead, it holds for arbitrary expansive matrices:
        the estimate ``$\lesssim$'' is trivial, whereas Step 3 of the proof shows ``$\gtrsim$''.

  \item For anisotropic Hardy spaces $H^p_A$ with $p \in (0,\infty)$, the matrix
        $A$ may be assumed to be exponential by the discussion in Section \ref{sec:hardy}.
  \end{enumerate}
\end{remark}

\begin{proof}[Proof of Theorem~\ref{thm:norm_equiv}]
  As seen in Section~\ref{sec:analyzing}, there exists $\psi \in \SC(\R^d)$
  with $\supp \widehat{\psi} \subset \supp \widehat{\varphi}$ and such that
  \[
    \sum_{j \in \mathbb{Z}}
      \widehat{\varphi}\bigl((A^*)^j \xi\bigr) \, \widehat{\psi}\bigl((A^*)^j \xi\bigr)
    = 1,
    \quad \; \xi \in \mathbb{R}^d \setminus \{0\}.
  \]
  Note that with $A$, also $A^\ast$ is expansive and exponential.
  By Lemma~\ref{lem:expansive_cont_powers}, it follows that there exist
  $0 < R_1 \leq R_2 < \infty$ such that
  \[
    R_1
    \leq \| (A^*)^{-t} \xi \|
    \leq R_2,
    \qquad t \in [-1,1], \quad \xi \in \supp \widehat{\varphi}.
  \]
  Choose $N > 0$ such that
  \(
    (A^*)^j \supp \widehat{\varphi}
    \cap \{ \xi \in \mathbb{R}^d : R_1 \leq \| \xi \| \leq R_2 \}
    = \emptyset
  \)
  for $|j|\geq N$, and define $ \Phi \in \Schwartz(\R^d)$ via its Fourier transform as
  \[
    \widehat{\Phi}(\xi)
    := \sum_{\ell = -N}^N
         \widehat{\varphi} \bigl((A^*)^\ell \xi\bigr) \, \widehat{\psi} \bigl((A^*)^\ell \xi\bigr) ,
  \]
  noting that $\widehat{\Phi}(\xi) = 1$ for $R_1 \leq \| \xi \| \leq R_2$.
  A direct calculation based on the preceding observations
  and using the convolution theorem shows that
  \begin{align}\label{eq:convolver}
    \varphi_{k} \ast \Phi_{{k+t}} = \varphi_{k}
    \quad \text{and} \quad
    \varphi_{k+t} \ast \Phi_k = \varphi_{k+t},
    \quad \; k \in \mathbb{Z}, \; t \in [0,1].
  \end{align}
  The remainder of the proof is split into three steps.
  For notational simplicity, we write throughout $\nu_{\beta} (y) := (1+\rho_{A} (y))^{\beta}$
  for $y \in \mathbb{R}^d$.
  By Equation~\eqref{eq:QuasiNormProperties},
  it follows that $\nu_{\beta}$ satisfies
  $\nu_{\beta}(x + y) \lesssim \nu_{\beta}(x) \; \nu_{\beta}(y)$
  for $x, y \in \mathbb{R}^d$, with implicit constant only depending on $A, \beta$.

  Let $f \in \SP$ be arbitrary.
  We prove the equivalences in~\eqref{eq:norm_equiv} in several steps.
 \\\\
  \textbf{Step 1.}
  In this step we show that $\| f \|_{\TL}$ can be estimated
  by the middle term of \eqref{eq:norm_equiv}.
  For arbitrary, but fixed $t \in [0,1]$, a direct calculation using \eqref{eq:convolver} gives
  \begin{align*}
    \hspace*{-0.28cm}
    \Big\| \!
      \Big(
        |\det A|^{\alpha j} \, |(f \ast \varphi_j)(x)|
      \Big)_{j \in \Z}
    \Big\|_{\ell^q} \!
     & = \Big\|
           \Big(
             |\det A|^{\alpha j}
             |(f \ast \Phi_{{j + t}} \ast \varphi_{j}) (x) |
           \Big)_{j \in \mathbb{Z}}
         \Big\|_{\ell^q} \\
     & \lesssim \!\! \sum_{\ell = - N}^N \!
                  \Big\| \!
                    \Big(
                      |\det A|^{\alpha j} \,
                      |
                       (
                        f
                        \ast \varphi_{{j + \ell + t}}
                        \ast \psi_{{j + \ell + t}}
                        \ast \varphi_{j}
                       ) (x)
                      |
                    \Big)_{\!j \in \mathbb{Z}}
                  \Big\|_{\ell^q}
     \numberthis \label{eq:triangle_sum}.
  \end{align*}
  To estimate \eqref{eq:triangle_sum}, note that for arbitrary $x \in \mathbb{R}^d$,
  \begin{align*}
    & |(f \ast \varphi_{{j + \ell + t}} \ast \psi_{{j + \ell + t}} \ast \varphi_{j}) (x) | \\
    & \leq \int_{\mathbb{R}^d}
             \frac{|(f \ast \varphi_{{j+\ell+t}} )(x+y)|}
                  {\nu_{\beta} (A^{j+\ell+t} y)}
             \nu_{\beta} (A^{j+\ell+t} y)
             |(\psi_{{j + \ell + t}} \ast \varphi_{j}) (-y) |
           \; dy \\
    & \leq \sup_{y \in \mathbb{R}^d}
             \frac{|(f \ast \varphi_{{j+\ell+t}} )(x+y)|}
                  {\nu_{\beta} (A^{j+\ell+t} y)}
             \int_{\mathbb{R}^d}
               \nu_{\beta} (A^{j+\ell+t} y)
               |(\psi_{{j + \ell + t}} \ast \varphi_{j}) (-y) |
             \; dy \\
             &= \DoubleStar{j+\ell+t} f (x)
                 \int_{\mathbb{R}^d}
                        \nu_{\beta} (A^{j+\ell+t} y)
                        |(\psi_{{j + \ell + t}} \ast \varphi_{j}) (-y) |
                      \; dy
    ,
    \numberthis \label{eq:pointwise}
  \end{align*}
  where $\DoubleStar{s} f$ is as in \eqref{eq:peetre_maximal}.
  We can estimate the integral in \eqref{eq:pointwise} by change-of-variables as
  \begin{align*}
    & \int_{\mathbb{R}^d}
        \nu_{\beta} (A^{j+\ell+t} y)
        |(\psi_{{j + \ell + t}} \ast \varphi_{j}) (-y) |
      \; dy \\
    & \leq \int_{\mathbb{R}^d}
             \nu_{\beta} (A^{j+\ell+t} y)
             \int_{\R^d}
               |\det A|^{{j + \ell + t}} \;
               |\psi( A^{j + \ell + t} w)| \;
               |\det A|^j  \;
               | \varphi( A^j( -y -w ) ) |
             \; dw
           \; dy \\
    & = \int_{\mathbb{R}^d}
          \int_{\R^d}
            \nu_{\beta} (A^{\ell+t} z) \;
            |\psi( v) | \;
            | \varphi( -z - A^{-(\ell + t)} v ) |
          \; d z
        \; d v \\
    & = \int_{\mathbb{R}^d}
          \int_{\R^d}
            \nu_{\beta} (A^{\ell+t} y - v) \;
            |\psi(v) | \;
            | \varphi(  -y ) |
          \; d y
        \; d v \\
    & \lesssim \int_{\mathbb{R}^d}
                 \nu_{\beta} (A^{\ell+t} y) \;
                 | \varphi(  -y ) |\;
               dy
               \int_{\R^d}
                 \nu_{\beta}(v) \;
                 |\psi(v) |
               \; d v
      \numberthis \label{eq:step1integral}.
  \end{align*}
  By Corollary~\ref{cor:quasi-norm_bound} and Lemma~\ref{lem:quasi-norm-bounds},
  we see for $-N \leq \ell \leq N$ and $t \in [0,1]$ that
  \begin{equation}
    \nu_{\beta} (A^{\ell + t} y)
    \lesssim |\det A|^{\beta(N+1)}
             (1+\rho_A(y))^{\beta}
    \lesssim |\det A|^{\beta(N+1)}
             (1 + \| y \| )^{\beta / \zeta_-}.
    \label{eq:weight_estimate}
  \end{equation}
  The integrals in \eqref{eq:step1integral} can therefore be bound independently of
  $-N \leq \ell \leq N$ and $t \in [0,1]$.
  Thus, \eqref{eq:pointwise} implies
  \begin{align*}
    |\det A|^{\alpha j} \,
    |(f \ast \varphi_{{j + \ell + t}} \ast \psi_{j+\ell+t} \ast \varphi_{j}) (x) |
    \lesssim |\det A|^{-\alpha (\ell+t)}
             |\det A|^{\alpha (j + \ell + t)}
              \DoubleStar{j+\ell+t} f(x)
  \end{align*}
  where we can estimate $|\det A|^{-\alpha (\ell+t)} \lesssim 1$ with
  implicit constants independent of $\ell, t$.
  Combining this with \eqref{eq:triangle_sum} gives
  \begin{align*}
    \Big\| \!
      \Big(
        |\det A|^{\alpha j}  |(f \ast \varphi_j)(x)|
      \Big)_{j \in \Z}
    \Big\|_{\ell^q} \!
    & \lesssim \sum_{\ell =-N}^N
                 \bigg\|
                   \bigg(
                     |\det A|^{\alpha(j+\ell+t)} \,
                      \DoubleStar{j+\ell+t} f (x)
                   \bigg)_{j \in \mathbb{Z}}
                 \bigg\|_{\ell^q} \\
    &\lesssim \bigg\|
                \bigg(
                  |\det A|^{\alpha(j+t)}
                   \DoubleStar{j+t} f (x)
                \bigg)_{j \in \mathbb{Z}}
              \bigg\|_{\ell^q}.
    \numberthis \label{eq:LHSindependt}
  \end{align*}

  Lastly, the left-hand side of \eqref{eq:LHSindependt} being independent of $t$,
  we average over $t \in [0,1]$.
  For this, let us assume $q < \infty$.
  Taking the $q$-th power of \eqref{eq:LHSindependt} and integrating gives
  \begin{align*}
    \sum_{j \in \mathbb{Z}}
      \bigl(|\det A|^{j\alpha}  |(f \ast \varphi_j)(x)|\bigr)^q
    & \lesssim \int_{0}^1
                 \sum_{j \in \mathbb{Z}}
                 \big(
                   |\det A|^{\alpha(j+t)}
                    \DoubleStar{j+t} f(x)
                 \big)^q  dt \\
    &= \int_{\mathbb{R}}
         \big(
           |\det A|^{\alpha s}
            \DoubleStar{s} f (x)
         \big)^q
       ds,
  \end{align*}
  and thus
  \[
    \bigg\|
      \bigg(
        \sum_{j \in \mathbb{Z}}
          (|\det A|^{j\alpha} |f \ast \varphi_j |)^q
      \bigg)^{1/q} \bigg\|_{L^p}
    \lesssim \bigg\|
               \bigg(
                 \int_{\mathbb{R}}
                   \big(
                     |\det A|^{\alpha s}
                      \DoubleStar{s} f
                   \big)^q
                 ds
               \bigg)^{1/q}
             \bigg\|_{L^p}.
  \]
  The case $q = \infty$ follows by the usual modifications.
\\\\
  \textbf{Step 2.} This step will show that the middle term can be bounded
  by the right-most term in \eqref{eq:norm_equiv}.
  Using the convolution identity \eqref{eq:convolver},
  we calculate for $x, z \in \R^d$, $j \in \Z$, and $t \in [0,1]$,
  \begin{align*}
    & \frac{|(f \ast \varphi_{{j+t}})(x+z)|}
           {\nu_{\beta} (A^{j+t} z)}
    \leq \sum_{\ell = -N}^N
           \int_{\mathbb{R}^d}
             \frac{|(f \ast \varphi_{{j+\ell}}) (x+y+z) |}
                  {\nu_{\beta} (A^{j+t} z)}
             |(\psi_{j+\ell} \ast \varphi_{{j+t}}) (-y)|
           \; dy \\
    & \quad \leq \sum_{\ell = -N}^N
                   \sup_{w \in \mathbb{R}^d}
                     \frac{|(f \ast \varphi_{{j+\ell}} ) (x+w)|}
                          {\nu_{\beta} (A^{j+\ell}w)}
                     \int_{\mathbb{R}^d}
                       \frac{\nu_{\beta} (A^{j + \ell} (z+y))}
                            {\nu_{\beta} (A^{j + t} z)}
                       | (\psi_{j+\ell} \ast \varphi_{{j+t}})(-y)|
                     \; dy.
    \numberthis \label{eq:integral_division}
  \end{align*}
  To estimate the integral in \eqref{eq:integral_division},
  note that the essential submultiplicativity of $\nu_{\beta}$ and a change-of-variable gives
  \begin{align*}
    & \int_{\mathbb{R}^d}
        \frac{\nu_{\beta} (A^{j + \ell} (z+y))}
             {\nu_{\beta} (A^{j + t} z)}
        | (\psi_{j+\ell} \ast \varphi_{{j+t}}) (-y)|
      \; dy \\
    & \quad \leq \int_{\mathbb{R}^d}
                   \frac{\nu_{\beta} (A^{j + \ell} (z+y))}
                        {\nu_{\beta} (A^{j + t} z)}
                   \int_{\R^d}
                     |\det A|^{j+\ell}
                     |\psi (A^{j+\ell}w)| \,
                     |\det A|^{j+t}
                     |\varphi (A^{j+t}(-y-w))|
                   \; dw
                 \; dy\\
    & \quad \leq \int_{\mathbb{R}^d}
                   \int_{\R^d}
                     \frac{\nu_{\beta} (A^{j + \ell} z+ A^{\ell - t} \zeta)}
                          {\nu_{\beta}(A^{j + t} z)} \;
                     |\psi (v)| \;
                     |\varphi (- \zeta - A^{t -\ell} v)|
                   \; d v
                 \; d  \zeta\\
    & \quad \lesssim \int_{\mathbb{R}^d}
                       \int_{\R^d}
                         \frac{\nu_{\beta} (A^{j + \ell}z)}
                              {\nu_{\beta}(A^{j + t} z)} \;
                         \nu_{\beta}(A^{\ell - t} \zeta) \;
                         |\psi (v)| \;
                         |\varphi (-\zeta)|
                       \; dv
                     \; d\zeta.
    \numberthis \label{eq:integral_quotient}
  \end{align*}

  Next, by Corollary~\ref{cor:quasi-norm_bound}, we have $\rho_A (A^t z) \gtrsim \rho_A (z)$
  for $t \in [0,1]$ and $z \in \mathbb{R}^d$.
  Therefore, we see for $-N \leq \ell \leq N$ and $t \in [0,1]$ that
  \begin{equation}
    \frac{1+ \rho_A (A^{\ell} z)}{1+\rho_A (A^t z)}
    \leq |\det A|^N \frac{1+ \rho_A (z)}{1+\rho_A (A^t z)}
    \lesssim 1,
    \label{eq:quotient_bound}
  \end{equation}
  with an implicit constant independent of $j, \ell, t$ and $z$.
  Combining \eqref{eq:quotient_bound} with \eqref{eq:weight_estimate},
  we then see that the integral \eqref{eq:integral_quotient} can
  be estimated independently of $j, \ell, t$.
  Therefore, \eqref{eq:integral_division} shows for $q < \infty$ that
  \begin{align*}
    \bigg(
      |\det A|^{\alpha (j+t)}  \DoubleStar{j+t} f (x)
    \bigg)^q
    & \lesssim \sum_{\ell = -N}^N
               \bigg(
                 \sup_{w \in \mathbb{R}^d}
                   \frac{|(f \ast \varphi_{{j+\ell}})(x+w)|}
                        {\nu_{\beta} (A^{j+\ell} w)}
                   |\det A|^{\alpha(j + \ell + t - \ell)}
               \bigg)^q \\
    &\lesssim  \sum_{\ell = -N}^N
               \bigg(
                 |\det A|^{\alpha (j + \ell)}
                  \DoubleStar{j+\ell} f (x)
               \bigg)^q.
    \numberthis \label{eq:det_Peetre}
  \end{align*}
  The right-hand side of \eqref{eq:det_Peetre} being independent of $t$,
  integrating \eqref{eq:det_Peetre} over $[0,1]$ shows that
  \begin{align*}
    \int_{\mathbb{R}}
      \big(
        |\det A|^{\alpha s}
         \DoubleStar{s} f(x)
      \big)^q
    \; ds
    &= \sum_{j \in \mathbb{Z}}
         \int_0^1
           \big(
             |\det A|^{\alpha (j+t)}
              \DoubleStar{j+t} f(x)
           \big)^q
         dt \\
     &\lesssim \sum_{j \in \mathbb{Z}}
                 \sum_{\ell = -N}^N
                   \bigg(
                     |\det A|^{\alpha (j + \ell)}
                      \DoubleStar{j+\ell} f(x)
                   \bigg)^q \\
   & \lesssim \sum_{j \in \mathbb{Z}}
                \big(
                  |\det A|^{\alpha j}
                   \DoubleStar{j} f(x)
                \big)^q, \numberthis \label{eq:det_Peetre2}
  \end{align*}
  and thus
  \[
    \bigg\|
      \bigg(
        \int_{\mathbb{R}}
          \big(
            |\det A|^{\alpha s}
             \DoubleStar{s} f
          \big)^q
        ds
      \bigg)^{1/q}
    \bigg\|_{L^p}
    \lesssim \bigg\|
               \bigg(
                 \sum_{j \in \mathbb{Z}}
                   \big(
                     |\det A|^{\alpha j}
                      \DoubleStar{j} f
                   \big)^q
               \bigg)^{1/q}
             \bigg\|_{L^p}.
  \]
  The case $q = \infty$ follows by the usual modifications.
\\\\
  \textbf{Step 3.} This final step will show that the right-most term in \eqref{eq:norm_equiv}
  can be estimated by $\| f \|_{\TL}$.
  Note first that
  \begin{align}\label{eq:peetre_dilation}
    \DoubleStar{j} f(x)
    = \sup_{z \in \mathbb{R}^d}
        \frac{|(f \ast \varphi_j) (x + A^{-j}z)|}
             {(1 + \rho_A (z))^{\beta}}
    =  \sup_{z \in \mathbb{R}^d}
         \frac{\bigl|[(f \ast \varphi_j) \circ A^{-j}] (A^j x + z)\bigr|}
              {(1 + \rho_A (-z))^{\beta}},
  \end{align}
  where the symmetry of $\rho_A$ is used.
  In order to estimate \eqref{eq:peetre_dilation}, we apply Peetre's inequality
  in Lemma~\ref{lem:peetre} to $g_j := (f \ast \varphi_j) \circ A^{-j}$.
  To this end, note with the (bilinear) dual pairing
  $\langle \cdot,\cdot \rangle_{\Schwartz',\Schwartz}$ that
  \begin{align*}
    \langle \widehat{g_j} , \phi \rangle_{\Schwartz', \Schwartz}
    & = \big\langle
          f \ast \varphi_j , \,\,
          |\det A^j| \, \widehat{\phi} \circ A^j
        \big\rangle_{\Schwartz',\Schwartz}
      = \big\langle
          \widehat{f \ast \varphi_j}, \,\,
          \phi \circ (A^*)^{-j}
        \big\rangle_{\Schwartz',\Schwartz} \\
    & = \big\langle
          \widehat{f}, \,\,
          (\widehat{\varphi} \circ \! (A^*)^{-j})
          \cdot (\phi \circ \! (A^*)^{-j})
        \big\rangle_{\Schwartz',\Schwartz}
      = 0
  \end{align*}
  for all $ \phi \in \Schwartz(\R^d)$
  with $\supp \phi \subset \R^d \!\setminus\! \supp \widehat{\varphi}$ .
  Thus, $\supp \widehat{g_j} \subseteq \supp \widehat{\varphi}$
  is contained in the same compact set for all $j \in \mathbb{Z}$.
  An application of Lemma~\ref{lem:peetre} therefore provides a uniform constant $C > 0$
  such that, for all $j \in \mathbb{Z}$,
  \[
    \DoubleStar{j} f(x)
    = \sup_{z \in \mathbb{R}^d}
        \frac{|g_j (A^j x + z)|}{(1 + \rho_A (-z))^{\beta}}
    \leq C \big[(M_{\rho_A} |g_j|^{1/\beta}) (A^j x) \big]^{\beta},
  \]
  where $M_{\rho_A}$ is as in~\eqref{eq:HardyLittlewoodMaximalFunctionDefinition}.
  Therefore, the right-hand side of \eqref{eq:norm_equiv} can be
  estimated using Lemma \ref{lem:maximal_dilation} and the vector-valued Fefferman-Stein inequality
  (Theorem~\ref{thm:fefferman-stein}) as follows:
  \begin{align*}
    \bigg\|
      \bigg(
        \sum_{j \in \mathbb{Z}}
          \big(
            |\det A|^{\alpha j}
             \DoubleStar{j} f
          \big)^q
      \bigg)^{1/q}
    \bigg\|_{L^p}
    & \lesssim \bigg\|
                 \bigg(
                   \sum_{j \in \mathbb{Z}}
                     \bigg(
                       |\det A|^{\alpha j}
                       \big[
                         (M_{\rho_A} |g_j|^{1/\beta}) (A^j \cdot)
                       \big]^{\beta}
                     \bigg)^q
                 \bigg)^{1/q}
               \bigg\|_{L^p} \\
    & = \bigg\|
          \bigg(
            \sum_{j \in \mathbb{Z}}
            \big(
              M_{\rho_A} \big[ (|\det A|^{\alpha j} |g_j|)^{1/\beta} \big] (A^j \cdot)
            \big)^{\beta q}
          \bigg)^{1/q}
        \bigg\|_{L^p} \\
    & = \bigg\|
          \bigg(
            \sum_{j \in \mathbb{Z}}
            \big(
              M_{\rho_A} (|\det A|^{\alpha j} |f \ast \varphi_j|)^{1/\beta}
            \big)^{\beta q}
          \bigg)^{1/(q\beta)}
        \bigg\|_{L^{p\beta}}^{\beta} \\
    & \lesssim \bigg\|
                 \bigg(
                   \sum_{j \in \mathbb{Z}}
                   \big(
                     |\det A|^{\alpha j}
                     |f \ast \varphi_j|
                   \big)^{ q}
                 \bigg)^{1/q}
               \bigg\|_{L^{p}}.
  \end{align*}
  The last step used that $p\beta, q \beta > 1$,
  so that Theorem~\ref{thm:fefferman-stein} is applicable.
\end{proof}

\subsection{Sequence spaces}
\label{sec:sequence}

This section provides a maximal function characterization of the sequence spaces
$\TLseq$ defined in Section~\ref{sub:TLSpaces}.
We start with a simple lemma.

\begin{lemma}\label{lem:pointwise_indicator}
  Let $A \in \GL(d,\R)$ be expansive,
  let $K \subset \R^d$ be bounded and measurable with positive measure,
  and let $\beta \geq 0$.
  For $\ell \in \Z$ and $z \in \R^d$, set $K_{\ell,z} := A^{-\ell} (K + z)$.
  Then
  \[
    \bigl(1 + \rho_A (A^\ell x - z)\bigr)^{-\beta}
    \lesssim \bigg(
               \Indicator_{K_{\ell,z}}
               \ast \frac{|\det A|^\ell}{(1 + \rho_A (A^\ell \cdot))^\beta}
             \bigg) (x)
    \qquad \, x \in \R^d ,
  \]
  where the implied constant only depends on $K, \beta, A$.
\end{lemma}

\begin{proof}
  Define $\nu(x) := (1 + \rho_A(x))^{-\beta}$.
    Note that
  \begin{align*}
    \Indicator_{K_{\ell,z}} (x)
     = |\det A|^{-\ell}  (\Indicator_K)_\ell (x - A^{-\ell} z)
      = |\det A|^{-\ell}  \big[ T_{A^{-\ell} z} (\Indicator_K)_\ell \big] (x).
  \end{align*}
  By applying similar manipulations to the left-hand side of the target estimate,
  and multiplying both sides of the target estimate by $|\det A|^\ell$,
  it is easily seen that the claim is equivalent to
  \[
    T_{A^{-\ell} z }\nu_\ell
    \lesssim [T_{A^{-\ell} z} (\Indicator_K)_\ell] \ast \nu_\ell .
  \]
  Since convolution commutes with translation,
  we can assume that $z = 0$, i.e., we need to show that
  $\nu_\ell \lesssim (\Indicator_K)_\ell \ast \nu_\ell$.
  Furthermore, using the identity
  $(f \circ A) \ast (g \circ A) = |\det A|^{-1} \cdot (f \ast g) \circ A$,
  it follows that it suffices to prove $\nu \lesssim \Indicator_K \ast \nu$.
  For this, note that
  since $\rho_A$ is  bounded on $K$, we have
  $1 + \rho_A(x - y) \lesssim 1 + \rho_A (x) + \rho_A(-y) \lesssim 1 + \rho_A(x)$ and hence
  $\bigl(1 + \rho_A(x-y)\bigr)^{-\beta} \gtrsim \bigl(1 + \rho_A(x)\bigr)^{-\beta}$
  for $x \in \R^d$ and $y \in K$.
  Therefore,
  \[
    \Indicator_K \ast \nu (x)
    = \int_{K}
        \bigl(1 + \rho_A (x-y)\bigr)^{-\beta}
      \, d y
    \gtrsim \int_K
              \bigl(1 + \rho_A(x)\bigr)^{-\beta}
            \, d y
    = \Lebesgue{K} \cdot \nu(x) ,
  \]
  which completes the proof.
\end{proof}

The following is a discrete counterpart of Theorem~\ref{thm:norm_equiv}
and will be used in Section~\ref{sec:peetre_discrete}.

\begin{theorem}\label{thm:maximal_sequence}
  Let $A \in \mathrm{GL}(d, \mathbb{R})$ be expansive and exponential.
  Then, for all $p \in (0, \infty)$, $q \in (0,\infty]$,
  $\alpha \in \mathbb{R}$ and $\beta > \max\{1/p,1/q\}$, the (quasi)-norm equivalence
  \begin{align*}
    &\| c \|_{\TLseqa}
     \asymp
    \\
    &\bigg\|
               \bigg(
                 \int_{\mathbb{R}}
                 \bigg( \!
                   \esssup_{z \in \mathbb{R}^d}
                     \frac{|\det A|^{-(\alpha + 1/2) s}}
                          {(1+\rho_A (A^{-s} z))^{\beta}}
                     \sum_{\ell \in \mathbb{Z}, k \in \mathbb{Z}^d} \!\!
                       |c_{\ell, k}| \,
                       \mathds{1}_{A^{-\ell} ([-1,1)^d + k)} (\cdot +z) \,
                       \mathds{1}_{-\ell + [-1, 1)} (s) \!
                 \bigg)^{\!q} ds
               \bigg)^{\!\frac{1}{q}}
             \bigg\|_{L^p}
  \end{align*}
  holds for all
  \(
    c
    = (c_{\ell, k})_{\ell \in \mathbb{Z}, k \in \mathbb{Z}^d}
    \in \mathbb{C}^{\mathbb{Z} \times \mathbb{Z}^d}
    ,
  \)
  with the usual modifications for $q = \infty$.
\end{theorem}

\begin{proof}
  We only prove the case $q < \infty$;
  the case $q = \infty$ can be proven by the usual modifications.
  For $\ell \in \mathbb{Z}$ and $k \in \mathbb{Z}^d$,
  define $Q_{\ell,k} := A^{-\ell} ([-1,1)^d + k)$ and $P_{\ell,k} := A^{-\ell} ([0,1)^d + k)$.
  Given $c = (c_{\ell,k})_{\ell \in \Z, k \in \Z^d} \in \CC^{\Z \times \Z^d}$,
  let $F : \mathbb{R}^d \times \mathbb{R} \to [0,\infty]$ be defined by
  \[
    F(x,s)
    := \sum_{\ell \in \mathbb{Z}, k \in \mathbb{Z}^d}
         |c_{\ell, k}| \,
         \mathds{1}_{Q_{\ell, k}} (x) \,
         \mathds{1}_{-\ell + [-1, 1)} (s),
    \quad (x,s) \in \mathbb{R}^d \times \mathbb{R}.
  \]
  Then we can re-write
  \begin{align*}
    I
    & := \int_{\mathbb{R}}
           \bigg(
             \esssup_{z \in \mathbb{R}^d}
               \frac{|\det A|^{-(\alpha + 1/2) s}}
                    {(1+\rho_A (A^{-s} z))^{\beta}}
               \sum_{\ell \in \mathbb{Z}, k \in \mathbb{Z}^d}
                 |c_{\ell, k}|
                 \mathds{1}_{A^{-\ell} ([-1,1)^d + k)} (\cdot +z)
                 \mathds{1}_{-\ell + [-1, 1)} (s)
           \bigg)^q
         ds \\
    & = \int_{\mathbb{R}}
          \bigg(
            |\det A|^{-(\alpha + 1/2)s } \esssup_{z \in \mathbb{R}^d}
            \frac{|F(\cdot +z, s)|}{(1+\rho_A (A^{-s} z))^{\beta}}
          \bigg)^q
        ds \\
    & = \sum_{j \in \mathbb{Z}}
          \int_{(0,1)}
            \bigg(
              |\det A|^{(\alpha + 1/2) (j+t)} \esssup_{z \in \mathbb{R}^d}
              \frac{|F(\cdot +z, -(j+t))|}{(1+\rho_A (A^{j+t} z))^{\beta}}
            \bigg)^q
          dt .
      \numberthis
      \label{eq:equality_periodizing}
  \end{align*}
  Note that for $j \in \mathbb{Z}$ and $t \in (0,1)$, we have
  \(
    F(x+z,-( j+t))
    \leq \sum_{\ell=j}^{j+1}
           \sum_{k \in \mathbb{Z}^d}
             |c_{\ell,k} | \mathds{1}_{Q_{\ell,k}} (x+z)
  \)
  for $x,z \in \mathbb{R}^d$.
  Moreover, for fixed $j \in \Z$, each $y \in \R^d$
  belongs to at most a fixed number of sets from the family $(Q_{j,k})_{k \in \Z^d}$; thus,
  \begin{equation}
    \sum_{k \in \Z^{d}}
      |c_{j,k}| \, \Indicator_{P_{j,k}}(x+z)
    \lesssim \bigl| F(x+z, -( j+t)) \bigr|^q
    \lesssim \sum_{\ell=j}^{j+1}
               \sum_{k \in \mathbb{Z}^d}
                 |c_{\ell,k} |^q \; \mathds{1}_{Q_{\ell,k}} (x+z).
    \label{eq:SequenceSpaceEquivalentNormFEstimate}
  \end{equation}
  Therefore,
  \begin{align*}
    I
    & \lesssim \sum_{m=0}^1
                 \sum_{j \in \mathbb{Z}}
                   \int_{(0,1)}
                     |\det A|^{(\alpha + 1/2) (j+m-m+t)q}
                     \esssup_{z \in \mathbb{R}^d}
                       \frac{
                              \sum_{k \in \mathbb{Z}^d}
                                |c_{j+m,k} |^q \mathds{1}_{Q_{j+m,k}} (\cdot+z)
                            }
                            {(1+\rho_A (A^{j+m-m+t} z))^{\beta q}}
                   \, d t\\
    &
    \lesssim \sum_{\ell \in \mathbb{Z}}
               |\det A|^{(\alpha + 1/2) \ell q}
               \esssup_{z \in \mathbb{R}^d}
                 \frac{ \sum_{k \in \mathbb{Z}^d} |c_{\ell,k} |^q \mathds{1}_{Q_{\ell,k}} (\cdot+z)}
                      {(1+\rho_A (A^{\ell} z))^{\beta q}} ,
   \numberthis \label{eq:withoutR}
  \end{align*}
  where the last step follows by using Corollary~\ref{cor:quasi-norm_bound}
  and noting that $|\det A|^{t-m}, |\det A|^{m-t} \lesssim 1$,
  with implicit constants independent of $t \in (0,1)$ and $m \in \{ 0,1 \}$.

  Next, since $\beta \min \{p,q\} > 1$, we can choose $r \in (0,\beta)$
  such that $r \, \min\{p,q\} > 1$, and estimate
  \begin{align*}
    I
    \lesssim \sum_{j \in \mathbb{Z}}
               \sum_{k \in \mathbb{Z}^d}
                 |\det A|^{(\alpha + 1/2) j q} \,
                 |c_{j,k} |^q \,
                 \bigg(
                   \esssup_{z \in \mathbb{R}^d}
                     \frac{ \mathds{1}_{Q_{j,k}} (\cdot+z)}
                          {(1+\rho_A (A^j z))^{\beta / r}}
                 \bigg)^{q r} .
  \numberthis \label{eq:withoutMaximal}
  \end{align*}

  To estimate \eqref{eq:withoutMaximal} further,
  note that $x + z \in Q_{j,k}$ for $x \in \mathbb{R}^d$,
  implies  $A^j(x+z) - k \in [-1,1]^d$, hence
  \begin{align*}
    1 + \rho_A (A^j x - k)
    & = 1 + \rho_A \bigl(A^j x + A^j z - k + (-A^j z)\bigr) \\
    & \lesssim (1 + \rho_A(A^j (x + z) - k))  (1 + \rho_A (-A^j z)) \\
    &  \lesssim 1 + \rho_A (A^j z) .
  \end{align*}
  Therefore, for arbitrary $x \in \mathbb{R}^d$,
  \begin{align*}
    \esssup_{z \in \mathbb{R}^d}
      \frac{ \mathds{1}_{Q_{j,k}} (x+z)}
           {(1+\rho_A (A^{j} z))^{\beta / r}}
    &\lesssim \frac{1}{(1+\rho_A (A^j x - k))^{\beta / r}} \\
   & \lesssim \bigg(
               \mathds{1}_{P_{j, k}} \ast \frac{|\det A|^j}{(1+\rho_A (A^j \cdot))^{\beta / r}}
             \bigg) (x), \numberthis \label{eq:upper_convolution}
  \end{align*}
  where the last inequality follows from Lemma~\ref{lem:pointwise_indicator}.
  The function $g_j := |\det A|^j  (1+\rho_A (A^j \cdot))^{-\beta / r}$ is in $L^1 (\mathbb{R}^d)$
  by Lemma~\ref{lem:QuasiNormIntegrability}.
  Moreover, we have $\| g_j \| = \|g_0 \|_{L^1}$ for every $j \in \mathbb{Z}$.
  Therefore, noting that $g_j(x) = |\det A|^j  (1 + |\det A|^j \, \rho_A(x))^{-\beta / r}$
  and applying the majorant property of the Hardy-Littlewood maximal function
  (see Lemma~\ref{lem:majorant}) to the right-hand side of \eqref{eq:upper_convolution} gives
  \begin{align}\label{eq:majorant_maximal}
    \esssup_{z \in \mathbb{R}^d}
      \frac{\mathds{1}_{Q_{j,k}} (x+z)}
           {(1+\rho_A (A^{j} z))^{\beta / r}}
    \lesssim M_{\rho_A} \mathds{1}_{P_{j,k}} (x),
    \quad x \in \mathbb{R}^d.
  \end{align}

  Combining \eqref{eq:withoutMaximal} and \eqref{eq:majorant_maximal} yields
  \begin{align*}
    I
    & \lesssim \sum_{j \in \mathbb{Z}}
                 \sum_{k \in \mathbb{Z}^d}
                   |\det A|^{(\alpha + 1/2) j q}
                    |c_{j,k} |^q
                    \Big( M_{\rho_A} \mathds{1}_{P_{j,k}} (\cdot) \Big)^{q r} \\
    & =        \sum_{j \in \mathbb{Z}}
                 \sum_{k \in \mathbb{Z}^d}
                   \bigg(
                     M_{\rho_A}
                     \Big[
                       |\det A|^{(\alpha + 1/2) j / r}
                        |c_{j,k} |^{1/r}
                        \mathds{1}_{P_{j,k}}
                     \Big] (\cdot)
                   \bigg)^{q r}.
  \end{align*}
  This, together with an application of the Fefferman-Stein inequality
  of Theorem~\ref{thm:fefferman-stein}, gives
  \begin{align*}
    \big\| I^{1/q} \big\|_{L^p}
    &\lesssim \bigg\|
                \bigg(
                  \sum_{j \in \mathbb{Z}}
                    \sum_{k \in \mathbb{Z}^d}
                       \bigg(
                         M_{\rho_A}
                         \big[
                           |\det A|^{(\alpha + 1/2) j / r}
                           |c_{j,k} |^{1/r}
                           \mathds{1}_{P_{j,k}}
                         \big] (\cdot)
                       \bigg)^{q r}
                \bigg)^{1/q}
              \bigg\|_{L^p} \\
    &= \bigg\|
         \bigg(
           \sum_{j \in \mathbb{Z}}
             \sum_{k \in \mathbb{Z}^d}
               \bigg(
                 M_{\rho_A}
                 \big[
                   |\det A|^{(\alpha + 1/2) j / r} \,
                   |c_{j,k} |^{1/r} \,
                   \mathds{1}_{P_{j,k}}
                 \big] (\cdot)
               \bigg)^{q r}
         \bigg)^{1 /(r q)}
       \bigg\|^{r}_{L^{p r}} \\
    & \lesssim \bigg\|
                 \bigg(
                   \sum_{j \in \mathbb{Z}} \sum_{k \in \mathbb{Z}^d}
                   \big(
                     |\det A|^{(\alpha + 1/2) j } |c_{j,k} | \mathds{1}_{P_{j,k}}  (\cdot)
                   \big)^{q}
                 \bigg)^{1/q}
               \bigg\|_{L^p}
      =        \| c \|_{\TLseqa}.
  \end{align*}

  The reverse estimate follows easily by combining the lower bound
  \[
    F(x,s)
    \lesssim \esssup_{z \in \mathbb{R}^d}
               \frac{|F(x+z, s)|}
                    {(1+\rho_A (A^{-s} z))^{\beta}},
    \quad (x, s) \in \mathbb{R}^d \times \mathbb{R},
  \]
  (see Lemma~\ref{lem:GroupMaximalProperties})
  with \eqref{eq:equality_periodizing} and \eqref{eq:SequenceSpaceEquivalentNormFEstimate}.
\end{proof}

\section{Admissible Schwartz functions and wavelet coefficient decay}
\label{sec:admissible}

Let $A \in \mathrm{GL}(d, \mathbb{R})$ be an exponential matrix.
Define the associated semi-direct product
\begin{equation}
  G_A = \mathbb{R}^d \rtimes_A \mathbb{R}
      = \{
          (x, s)
          :
          x \in \mathbb{R}^d, s \in \mathbb{R}
        \}
  \label{eq:SemidirectProductDefinition}
\end{equation}
with multiplication $(x,s) (y,t) = (x + A^sy, s+t)$ and inversion $(x,s)^{-1} = (-A^{-s} x, -s)$.
Left Haar measure on $G_A$ is given by $d \mu_{G_A} (x,s) = |\det A|^{-s} ds dx$,
and the modular function on $G_A$ is $\Delta_{G_A} (x,s) = |\det A|^{-s}$.
To ease notation, we often simply write $\mu := \mu_{G_A}$.

For $p \in (0, \infty)$, the Lebesgue space on $G_A$
is denoted by $L^p (G_A) = L^p (G_A, \mu_{G_A})$.
The left and right translation by $h \in G_A$ of a function $F : G_A \to \mathbb{C}$ are defined by
\[
  L_h F = F(h^{-1} \cdot)
  \qquad \text{and} \qquad
  R_h F = F (\cdot \, h)
\]
respectively.

\subsection{Admissible vectors}
\label{sub:AdmissibleVectors}

The \emph{quasi-regular representation} $(\pi, L^2 (\mathbb{R}^d))$ of
$G_A = \mathbb{R}^d \rtimes_A \mathbb{R}$ is given by
\[
  \pi(x,s) f = |\det A|^{-s/2} f (A^{-s} (\, \cdot \, - x)),
  \quad f \in L^2 (\mathbb{R}^d).
\]
For fixed $\psi \in L^2 (\mathbb{R}^d)$, the associated \emph{wavelet transform} is defined as
\[
  W_{\psi} : \quad
  L^2 (\mathbb{R}^d) \to L^{\infty} (G_A), \quad W_{\psi} f
  (x,s) = \langle f, \pi(x,s) \psi \rangle,
  \quad (x,s) \in G_A ,
\]
and $\psi$ is \emph{admissible} if $W_{\psi}$ defines an isometry into $L^2 (G_A)$.
This implies ${W_\psi^\ast W_\psi = \mathrm{id}_{L^2(\R^d)}}$, which gives rise
to the \emph{reconstruction formula}
\begin{equation}
  f
  = W_\psi^\ast W_\psi f
  = \int_{G_A}
      W_\psi f (g)
      \pi(g) \psi
    \, d \mu_{G_A} (g),
  \qquad  \, f \in L^2(\R^d) ,
  \label{eq:ReconstructionFormula}
\end{equation}
with the integral interpreted in the weak sense.
Furthermore, the \emph{reproducing formula}
\begin{equation}
  \label{eq:reproducing-formula-L2}
  W_{\varphi}f = W_{\psi}f \ast W_{\varphi}\psi, \quad f,\varphi \in L^2(\R^d)
\end{equation}
follows directly from the isometry of $W_\psi$ and the intertwining
property $W_\psi [\pi(g) f] = L_g [W_\psi f]$.

Admissibility of a vector can be conveniently characterized in terms of its Fourier transform,
see, e.g., \cite[Theorem~1.1]{laugesen2002characterization}
and \cite[Theorem~1]{fuehr2002continuous}.

\begin{lemma}[{\cite{laugesen2002characterization,fuehr2002continuous}}]
  \label{lem:AdmissibilityFourierCharacterization}
  A vector $\psi \in L^2 (\mathbb{R}^d)$ is admissible if, and only if,
  \begin{equation}
    \label{eq:admissible-vector}
    \int_{\mathbb{R}} \bigl|\widehat{\psi} ((A^*)^s \xi )\bigr|^2 ds = 1,
    \quad \text{a.e.} \; \xi \in \mathbb{R}^d.
  \end{equation}
\end{lemma}

The significance of $A$ being expansive is that this guarantees the existence of admissible
vectors with convenient additional properties:

\begin{theorem}[\cite{groechenig1992compact,kaniuth1996minimal,currey2016integrable, bownik2003anisotropic}]
  \label{thm:AdmissibleVectorsExistence}
  Let $A \in \mathrm{GL}(d, \mathbb{R})$ be an exponential matrix.
  Then the following assertions are equivalent:
  \begin{enumerate}[(i)]
    \item Either $A$ or $A^{-1}$ is expansive.

    \item There exists an admissible vector $\psi \in L^2 (\mathbb{R}^d)$ such that
          $\widehat{\psi} \in C_c^{\infty} (\mathbb{R}^d)$.
  \end{enumerate}
  If $A$ is expansive, there exists an admissible $\varphi \in \Schwartz(\R^d)$
  satisfying $\widehat{\varphi} \in C_c^\infty(\R^d \setminus \{ 0 \})$.
  In addition, it can be chosen to satisfy the support condition \eqref{eq:analyzing_positive}.
\end{theorem}

\begin{proof}
  The claimed equivalence is proven in \cite{groechenig1992compact,kaniuth1996minimal},
  see also \cite[p.~319]{schulz2004projections}.
  The final claim easily follows from \cite[Proposition~10]{currey2016integrable} or \cite[Chapter II, Theorem 4.2]{bownik2003anisotropic} and their proofs.
\end{proof}

In the sequel, a matrix $A \in \mathrm{GL}(d, \mathbb{R})$ will be assumed
to be expansive and exponential.

\subsection{Decay estimates}

This section concerns decay properties of the wavelet transform.
The derived decay estimates will play an important role in the subsequent sections,
but are also of independent interest.

We recall the following Fr\'echet space of Schwartz functions
with all moments vanishing.

\begin{definition}\label{def:VanishingMomentSchwartz}
  Let $\mathcal{S}_0 (\mathbb{R}^d)$ denote the space of all
  $\varphi \in \mathcal{S} (\mathbb{R}^d)$ satisfying
  \[
    \int_{\mathbb{R}^d} \varphi (x) x^{\alpha} dx = 0
  \]
  for all multi-indices $\alpha \in \mathbb{N}_0^d$.
  The space $\mathcal{S}_0 (\mathbb{R}^d)$ will be equipped with the subspace topology
  coming from $\Schwartz(\R^d)$.
  Its (topological) dual space will be denoted by $\mathcal{S}_0' (\mathbb{R}^d)$.
\end{definition}

The dual space $\mathcal{S}_0' (\mathbb{R}^d)$ can be identified with $\SP$;
see, e.g.,~\cite[Proposition~1.1.3]{grafakos2014modern}.

The following lemma will be helpful in establishing decay of the wavelet transform.
It is a generalization to the anisotropic setting of a well-known estimate,
see, e.g.,~\mbox{\cite[Appendix~B.1]{grafakos2014modern}}.

\begin{lemma}\label{lem:GrafakosStyleEstimate}
  If $s \geq 0$ and $L > 1$, then
  \[
    \int_{\R^d}
      \bigl(1 + \rho_A(y)\bigr)^{-L}
      \bigl(1 + \rho_A (A^{-s} (y - x))\bigr)^{-L}
    \, d y
    \lesssim_{d, A, L} \bigl(1 + \rho_A (A^{-s} x)\bigr)^{-L}
  \]
   for all $x \in \R^d$.
\end{lemma}
\begin{proof}
  Since $L > 1$, an application of Lemma~\ref{lem:QuasiNormIntegrability} shows
  $\int_{\R^d} (1 + \rho_A (y))^{-L} \, d y \lesssim 1$.
  Therefore, if $\rho_A (A^{-s} x) \leq 1$, then
  \begin{align*}
     \int_{\R^d}
      \bigl(1 + \rho_A(y)\bigr)^{-L}
      \bigl(1 + \rho_A (A^{-s} (y - x))\bigr)^{-L}
    \, d y
    &\leq \int_{\R^d} (1 + \rho_A (y))^{-L} \, d y \\
    &\lesssim (1 + \rho_A (A^{-s} x))^{-L}.
  \end{align*}
  In the remainder of the proof, it may therefore be assumed that $\rho_A (A^{-s} x) > 1$.

  Let $C_1 \geq 1$ with $\rho_A (x + y) \leq C_1  (\rho_A (x) + \rho_A (y))$,
  and let $C_2 \geq 1$ denote the constant in \Cref{cor:quasi-norm_bound},
  so that $\rho_A(A^s x) \leq C_2 \, |\det A|^s \rho_A(x)$ for all $x,y \in \R^d$ and $s \in \R$.
  Define
  \[
    U
    :=
    \bigl\{
      y \in \R^d
      \, \colon
      \rho_A(y) \geq (2 C_1C_2)^{-1} |\det A|^s  \rho_A(A^{-s} x)
    \bigr\}
  \]
  and
  $
    V
    :=
    \bigl\{
      y \in \R^d
      \, \colon
      \rho_A (A^{-s} (y-x)) \geq \rho_A (A^{-s} x) / (2 C_1)
    \bigr\} .
  $
  Then $\R^d = U \cup V$; otherwise,
  \begin{align*}
    \rho_A (A^{-s} x)
    & \leq C_1  \big( \rho_A (A^{-s} (x-y)) + \rho_A (A^{-s} y) \big) \\
    & \leq C_1  \big( \rho_A (A^{-s} (x-y)) + C_2 \, |\det A|^{-s} \rho_A (y) \big) \\
    & <    C_1  \big(
                     \rho_A(A^{-s} x) / (2C_1)
                     + C_2 \, |\det A|^{-s} (2 C_1 C_2)^{-1} |\det A|^s \rho_A(A^{-s} x)
                   \big) \\
     &= \rho_A(A^{-s} x),
  \end{align*}
  for any $y \in \mathbb{R}^d \setminus (U \cup V)$.

  On the one hand, it follows by $\rho_A(A^{-s} x) \geq 1$ and a change-of-variable that
  \begin{align*}
    & \int_U
        \bigl(1 + \rho_A (y)\bigr)^{-L} \,
        \bigl(1 + \rho_A (A^{-s} (y - x))\bigr)^{-L}
      \, d y \\
    & \leq \frac{(2 C_1C_2)^L \cdot |\det A|^{-L s}}{\rho_A (A^{-s} x)^L}
           \int_{\R^d}
             \big( 1 + \rho_A (A^{-s} (y - x)) \big)^{-L}
           \, d y \\
    & \leq
           \frac{(4 C_1C_2)^L \, |\det A|^{-(L-1)s}}{(1 + \rho_A(A^{-s} x))^L}
           \int_{\R^d}
             (1 + \rho_A(z))^{-L}
           \, d z \\
    &  \lesssim \frac{1}{(1 + \rho_A(A^{-s} x))^L} ,
  \end{align*}
  where the last inequality uses Lemma~\ref{lem:QuasiNormIntegrability} and
  ${|\det A|^{-(L-1)s} \leq 1}$ since $L > 1$ and $s \geq 0$.
  On the other hand, if $y \in V$, then
  \(
    1 + \rho_A(A^{-s} (y - x))
    \geq (2C_1)^{-1}  (1 + \rho_A(A^{-s} x)) .
  \)
  Therefore,
  \begin{align*}
    & \int_V
        \bigl(1 + \rho_A(y)\bigr)^{-L}
        \bigl(1 + \rho_A (A^{-s} (y - x))\bigr)^{-L}
      \, d y
      \lesssim \frac{1}{(1 + \rho_A(A^{-s} x))^L}
        \end{align*}
  by Lemma~\ref{lem:QuasiNormIntegrability}.
  Combining these estimates yields the claim.
\end{proof}

\begin{lemma}\label{lem:DiagonalWaveletDecay}
  Let
  $f_1,f_2 \in L^2(\R^d)$.
  \begin{enumerate}[(i)]
  \item If $|f_i(\cdot)| \leq C_i  (1 + \rho_A(\cdot))^{-L}$ a.e. for some $L > 1$ and
        all $i \in \{ 1, 2 \}$, then
        \begin{equation}
          \label{eq:wavelet-decay-x}
          |W_{f_1} f_2 (x,s)|
          \lesssim C_1 C_2 \,
                   |\det A|^{-|s|/2}
                    \big(1 + \rho_A (A^{- \PosPart{s}} x)\big)^{-L}
        \end{equation}
        for all $s \in \R$, where the implied constant only depends on $d,L, A$.

  \item If $f_1 \in C^N(\R^d)$ satisfies $|\partial^\alpha f_1 (x)| \leq C_3$
        for all $\alpha \in \N_0^d$ such that $|\alpha| \leq N$, and
        \[
          \int_{\R^d} \|x\|^N \, |f_2(x)| \, d x \leq C_4,
          \quad \text{and} \quad
          \int_{\R^d} x^\alpha \, f_2(x) \, d x = 0
          \text{ for } |\alpha| < N,
        \]
        then
        \begin{equation}\label{eq:wavelet-decay-s}
          |W_{f_1}f_2 (x,s)| \lesssim C_3 C_4 \, |\det A|^{-s/2} \, \| A^{-s} \|_{\infty}^N,
        \end{equation}
        for all $s \in \R$, where the implied constant only depends on $d,N$.
  \end{enumerate}
\end{lemma}

\begin{proof}
  (i)  $s \geq 0$, Lemma~\ref{lem:GrafakosStyleEstimate} implies
  \begin{align*}
    |W_{f_1} f_2(x,s)|
    & \leq C_1 C_2
           \int_{\R^d}
             (1 + \rho_A(y))^{-L} \,
             |\det A|^{-s/2} \,
             \big(1 + \rho_A(A^{-s} (y-x))\big)^{-L}
           \, d y \\
    & \lesssim C_1 C_2 \, |\det A|^{-s/2} \, \big(1 + \rho_A(A^{-s} x)\big)^{-L},
  \end{align*}
  as claimed.
  For $s \leq 0$, note that
  \begin{align*}
    \big| W_{f_1} f_2 (x, s) \big|
    & = \big| W_{f_2} f_1 (-A^{-s} x, - s) \big| \\
    & \lesssim C_1 C_2
                |\det A|^{-|-s|/2}
                \big( 1 + \rho_A(-A^{-\PosPart{(-s)}} A^{-s} x) \big)^{-L} \\
    & =        C_1 C_2
                |\det A|^{-|s|/2}
                \big( 1 + \rho_A (A^{-\PosPart{s}} x) \big)^{-L} .
  \end{align*}
  \medskip{}
  (ii) By Taylor's theorem, there exists a polynomial $P_x$ of degree $N-1$
  that satisfies
  \[ |f_1(x + z) - P_x(z)| \lesssim C_3  \|z\|^N \quad \text{for all} \quad z \in \R^d, \]
  with implied constant only depending on $d,N$.
  Since $\int_{\R^d} P(y) f_2(y) \, d y = 0$ for any polynomial $P$ with degree at most $N-1$,
  it follows that
   \begin{align*}
    |W_{f_1}f_2(x,s)|
    & = \Big|
          \int_{\R^d}
            f_2(y)
            |\det A|^{-s/2} \,
            \overline{f_1(A^{-s} (y - x))}
          \, d y
        \Big| \\
    & = \Big|
          \int_{\R^d}
            f_2(y)
            |\det A|^{-s/2}
            \overline{
              [f_1(A^{-s}y - A^{-s}x) - P_{-A^{-s}x}(A^{-s}y)]
            }
          \, d y
        \Big| \\
    & \lesssim C_3 \, |\det A|^{-s/2}
               \int_{\R^d}
                 |f_2(y)|
                  \|A^{-s}y\|^N
               \, d y \\
     & \leq C_3 C_4 \, |\det A|^{-s/2} \, \|A^{-s}\|_{\infty}^N,
  \end{align*}
  as required.
 \end{proof}

The following consequence is what we will actually use in most applications.

\begin{corollary}\label{cor:wavelet-decay}
  Let $\psi, \varphi \in \SC_0(\R^d)$ and  $1 < \lambda_{-} < \min_{\lambda \in \sigma(A)}|\lambda|$ be as
  in Lemma~\ref{lem:expansive_cont_powers}.
  Then, for every $L, N \in \N$,
  \begin{equation}
    |W_{\psi}\varphi (x,s)|
    \lesssim  \bigl(1 + \rho_A(x)\bigr)^{-L} \lambda_{-}^{- |s| N}  \| \psi \| \, \|\varphi\|,
  \end{equation}
  where  $\|\,\cdot\,\|$ is a suitable
  continuous Schwartz semi-norm.
  The implied constant and the choice of the semi-norms depend only on $L ,N, A, d,\lambda_-$.
\end{corollary}

\begin{proof}
  Note that $\psi, \varphi \in \SC_0(\R^d)$ guarantees that
  all assumptions of Lemma~\ref{lem:DiagonalWaveletDecay} are satisfied and
  the bounds $C_1, \dots, C_4$ can be replaced by suitable Schwartz
  semi-norms of $\psi$ or $\varphi$.

  We first use the estimate \eqref{eq:wavelet-decay-x}.
  Note that $\rho_A(A^{-\PosPart{s}}x) \gtrsim |\det A|^{- \PosPart{s}} \rho_A(x)$ by
  Corollary~\ref{cor:quasi-norm_bound}.
  Therefore, we see for any $K > 1$ that
  \begin{align}
    \label{eq:wavelet-decay-proof-eq1}
    |W_{\psi}\varphi(x,s)|
    & \lesssim \|\psi\| \, \|\varphi\|
               \, |\det A|^{-|s|/2}
                \bigl( 1 + |\det A|^{- \PosPart{s}} \rho_A(x) \bigr)^{-K}
               \notag \\
    & \lesssim \|\psi\| \, \|\varphi\|
               \, |\det A|^{-|s|/2}
               \, \max\bigl\{1, |\det A|^{K \, \PosPart{s}}\bigr\}
               \, \bigl(1 + \rho_A(x)\bigr)^{-K}
               \notag \\
    & \lesssim \|\psi\| \, \|\varphi\|
               \, |\det A|^{K \, \PosPart{s}}
               \bigl(1 \!+\! \rho_A(x)\bigr)^{-K}
      \leq     \|\psi\| \, \|\varphi\|
               \, |\det A|^{|s| K}
               \bigl(1 \!+\! \rho_A(x)\bigr)^{-K} ,
  \end{align}
  where  $\|\,\cdot\,\|$ is a suitable Schwartz semi-norm depending on $K,A$.

  We now show for arbitrary $M \in \N$ that
  \begin{equation}\label{eq:wavelet-decay-proof-eq2}
    |W_{\psi}\varphi (x,s)| \lesssim \lambda_{-}^{- |s| M} \| \psi \| \, \|\varphi\| .
  \end{equation}
  Indeed, if $s \geq 0$, then $\|A^{-s}\|_{\infty} \lesssim \lambda_{-}^{-s}$
  by Lemma~\ref{lem:expansive_cont_powers} and
  the claim follows immediately from \eqref{eq:wavelet-decay-s}.
  The claim for $s \leq 0$ follows from the case $s \geq 0$ via
  $ W_{\psi}\varphi(x,s)= \overline{W_{\varphi}\psi(-A^{-s}x,-s)}$.

  Finally, we interpolate between \eqref{eq:wavelet-decay-proof-eq1}
  and \eqref{eq:wavelet-decay-proof-eq2}.
  To this end, note that a priori the seminorms in \eqref{eq:wavelet-decay-proof-eq1}
  and \eqref{eq:wavelet-decay-proof-eq2} are distinct, but that we can assume
  that they are equal by possibly enlarging them.
  Now, since $\lambda_{-} > 1$, we can choose $H = H(A,\lambda_-) \in \N$
  such that $\lambda_{-}^H \geq |\det A|$.
  Taking $K = 2L$ and $M= 2(H L + N)$ yields that
  \begin{align*}
    |W_{\psi}\varphi(x,s)|
    & =        |W_{\psi}\varphi(x,s)|^{1/2}  |W_{\psi}\varphi(x,s)|^{1/2} \\
    & \lesssim \|\psi\| \, \|\varphi\|
               \, |\det A|^{|s| L}
               \, \bigl(1 + \rho_A(x)\bigr)^{-L}
               \, \lambda_{-}^{- |s| (H L + N)} \\
    & \lesssim \|\psi\| \, \|\varphi\|
                 (1 + \rho_A(x))^{-L}
                 \lambda_{-}^{- |s| N} ,
  \end{align*}
  as claimed.
\end{proof}

\subsection{Extended wavelet transform}
\label{sub:ExtendedWaveletTransform}

The wavelet transform can be extended via duality to $\SC_0'(\R^d) \cong \SP$.
Throughout, we will use the dual bracket defined by
\begin{equation}
  \langle \cdot , \cdot \rangle : \quad
  \SC_0'(\R^d) \times \SC_0(\R^d) \to \CC, \qquad
  \langle f, \varphi \rangle := f(\overline{\varphi}) .
  \label{eq:SesquilinearDualBracket}
\end{equation}
The bracket is a sesquilinear form naturally extending the $L^2$-inner product.

If $\psi \in \SC_0(\R^d)$, then the \emph{(extended) wavelet transform}
\begin{equation}
  W_{\psi} : \quad
  \SC_0'(\R^d) \to C (G_A), \quad
  W_{\psi} f (x,s) = \langle f, \pi(x,s) \psi \rangle,
  \quad (x,s) \in \mathbb{R}^d \times \mathbb{R},
  \label{eq:ExtendedWaveletTransform}
\end{equation}
is well-defined.
Here, we implicitly use the continuity of
$\R^d \times \R \to \Schwartz(\R^d), (x,s) \mapsto \pi(x,s) \psi$.
In addition to the wavelet transform, we also extend the representation $\pi$ to $\SC_0'(\R^d)$
by defining
\[
  \langle \pi(h)f, \varphi \rangle
  := \langle f, \pi(h^{-1})\varphi \rangle
  \qquad \text{for } f \in \SC_0'(\R^d) \text{ and } \varphi \in \SC_0(\R^d).
\]
The following lemma extends the reconstruction formula \eqref{eq:ReconstructionFormula} to
all of $\SC_0'(\R^d)$.

\begin{lemma}\label{lem:weak-continuity}
  Let $\psi \in \SC_0(\R^d)$ be admissible.
  Then
  \begin{equation}
    \label{eq:weak-continuity}
    \int_{G_A}
      W_\psi f (g) \,
      \overline{W_\psi \varphi (g)}
    \, d \mu_{G_A}(g)
    = \langle f, \varphi \rangle
  \end{equation}
  for all $f \in \SC_0'(\R^d)$ and  $\varphi \in \SC_0(\R^d) $.
\end{lemma}

\begin{proof}
  The proof follows \cite[Lemma~2.11]{fuehr2020coorbit}
  and \cite[Lemma~40]{FuehrVoigtlaenderCoorbitAsDecomposition},
  with suitable modifications.

  For $M, N \in \N_{\geq d+1}$, let $\SC_{M,N}(\R^d)$ denote the space of all functions
  $f \in C^N(\R^d)$ satisfying
  \begin{equation}\label{eq:SN-family}
    \|f\|_{M,N}
    := \max_{\beta \in \N_0^d, |\beta| \leq N}
        \sup_{x \in \R^d}
          (1+\|x\|)^{M} |\partial^{\beta} f(x)|
    < \infty .
  \end{equation}
  The function space $\SC_{M,N}(\R^d)$ equipped
  with the norm in \eqref{eq:SN-family} is a Banach space.
  Furthermore, $\SC(\R^d) \hookrightarrow \SC_{M,N}(\R^d)$.
  Since $G_A \to \SC(\R^d), g \mapsto \pi(g) \psi$ is continuous and $G_A$ is $\sigma$-compact,
  this implies that the map
  \begin{equation}
    G_A \to \SC_{M,N}(\R^d), \quad
    g \mapsto W_{\psi}\varphi (g) \, \pi(g) \psi
    \label{eq:ReproducingFormulaBochnerMap}
  \end{equation}
  is continuous  and has a $\sigma$-compact (and hence separable) range.
  Moreover, the decay estimates of Corollary~\ref{cor:wavelet-decay} show
  ${\int_{G_A} |W_\psi \varphi(g)| \, \| \pi(g) \psi \|_{M,N} \, d \mu_{G_A}(g) < \infty}$.
  Overall, this shows that the map in \eqref{eq:ReproducingFormulaBochnerMap}
  is Bochner integrable, for arbitrary $M,N \in \N_{\geq d+1}$.

  The reconstruction formula \eqref{eq:ReconstructionFormula} shows for
  $\varphi \in \SC_0(\R^d) \subset L^2(\R^d) \cap \SC_{M,N}(\R^d)$ that
  \begin{equation}\label{eq:weak-cont-pr1}
    \varphi
    = \int_{G_A}
        W_{\psi}\varphi(g) \, [\pi(g)\psi]
      \, d\mu_{G_A}(g)
  \end{equation}
  where the integral is understood in the weak sense in $L^2(\R^d)$.
  As shown above, the right-hand side also exists as a Bochner integral in $\SC_{M,N} (\R^d)$.
  Since $M \geq d+1$, we have
  $\SC_{M,N}  \hookrightarrow L^2(\R^d)$.
   Furthermore, if $\varphi \in \SC_{M,N}$ satisfies
  $\langle \varphi, f \rangle = 0$ for all $f \in L^2(\R^d)$, then $\varphi \equiv 0$.
  Hence the identity \eqref{eq:weak-cont-pr1} also holds in $\SC_{M,N}(\R^d)$.

  Lastly, if $f \in \SC_0'(\R^d)$, then $f$ extends to a continuous linear functional
  on $\SC(\R^d)$ by~\cite[Proposition~1.1.3]{grafakos2014modern}.
  Hence, there are $M,N \in \N_{\geq d+1}$, such that the restriction of $f$
  to $\SC_0(\R^d)$ is continuous with respect to $\| \cdot \|_{M,N}$;
  see \cite[Proposition~2.3.4]{GrafakosClassicalFourier}.
  Using the Hahn-Banach theorem, we can extend $f$ to a bounded linear functional $\widetilde{f}$
  on $\SC_{M,N}(\R^d)$.
  In view of \eqref{eq:weak-cont-pr1}, and using that the Bochner-integral
  can be interchanged with bounded linear functionals
  by \cite[V.5, Corollary~2]{yosida1980functional}, we obtain that
  \[
    \langle f, \varphi \rangle
    = \widetilde{f}(\overline{\varphi})
    = \widetilde{f}
      \Big(
        \int_{G_A}
          \overline{W_{\psi}\varphi(g)}  \overline{\pi(g)\psi}
        \, d\mu_{G_A}(g)
      \Big)
    = \int_{G_A}
        \overline{W_{\psi}\varphi(g)}
        \langle f, \pi(g) \psi \rangle
      \, d\mu_{G_A}(g)
  \]
  for any $\varphi \in \SC_0(\R^d)$.
\end{proof}

\begin{corollary}[Reproducing formula]\label{cor:reproducing-formula-distr}
  Let $\psi \in \SC_0(\R^d)$ be admissible.
  Then
  \begin{equation}
    \label{eq:reproducing-formula-distr}
    W_{\varphi}f = W_{\psi}f \ast W_{\varphi}\psi
  \end{equation}
  holds for all $f \in \SC_0'(\R^d)$ and $\varphi \in \SC_0(\R^d)$.
\end{corollary}

\begin{proof}
  Replacing $\varphi$ by $\pi(h) \varphi$ in Lemma~\ref{lem:weak-continuity}
  easily yields the claim.
\end{proof}

\section{Coorbit spaces associated to Peetre-type spaces}
\label{sec:coorbit}

This section is devoted to characterizations of anisotropic Triebel-Lizorkin spaces
in terms of wavelet transforms. Explicitly, it will be shown that Triebel-Lizorkin spaces can be identified with coorbit spaces associated to so-called Peetre-type spaces.

\subsection{Peetre-type spaces}
\label{sub:PeetreSpaces}

For $p,q \in (0,\infty]$, the mixed-norm Lebesgue space $L^{p,q}(G_A)$ consists
of all (equivalence classes of a.e.\ equal) measurable functions ${F : G_A \to \CC}$
satisfying
\begin{align} \label{eq:mixed_norm}
  \| F \|_{L^{p,q}}
  := \big\| x \mapsto \| F(x,\cdot) \|_{L^q(\nu)}  \big\|_{L^p(\R^d)} < \infty,
\end{align}
relative to the Borel measure $\nu$ on $\mathbb{R}$
defined by $\nu(M) = \int_M \frac{d s}{|\det A|^s}$.
The weighted space is given by
$L_w^{p,q}(G_A) = \{ F : G_A \to \CC \colon w \cdot F \in L^{p,q}(G_A) \}$,
with norm $\| F \|_{L_w^{p,q}} := \| w \cdot F \|_{L^{p,q}}$.

\begin{definition}
For $\alpha \in \mathbb{R}, \beta > 0$, and $p \in (0, \infty)$ and $q \in (0, \infty]$,
the \emph{Peetre-type space} $\PT (G_A)$ on $G_A$ is defined
as the collection of all (equivalence classes of a.e.\ equal)
measurable $F : G_A \to \mathbb{C}$ satisfying
\[
  \| F \|_{\PT}
  := \bigg\|
       x \mapsto
       \bigg(
         \int_{\mathbb{R}}
           \bigg(
             |\det A|^{\alpha s} \,
             \esssup_{z \in \mathbb{R}^d}
               \frac{|F(x+z,s)|}
                    {(1 + \rho_A (A^{-s} z))^{\beta}}
           \bigg)^q
         \frac{ds}{|\det A|^{s}}
       \bigg)^{1/q}
     \bigg\|_{L^p}
  < \infty,
\]
with the usual modification for $q = \infty$.
\end{definition}

An essential property of the Peetre-type spaces for our purposes is their two-sided translation invariance.
For proving this, the following lemma will be used.
Its proof is deferred to Appendix~\ref{sub:SpecialWeightProof}.

\begin{lemma}\label{lem:SpecialWeight}
  The weight function
  \begin{equation}
    v : \quad
    G_A \to [0,\infty), \quad
    (y,t) \mapsto \sup_{(z, u) \in G_A}
                    \frac{1 + \rho_A (A^{-u} z)}{1 + \rho_A (A^{-u} A^t z - y)}
    \label{eq:SpecialWeightDefinition}
  \end{equation}
  is well-defined, measurable, and submultiplicative.
  Furthermore, we have
  \begin{equation}
    v(y,t)
    \asymp \max \bigl\{ 1, |\det A|^{-t} \bigr\}
            \bigl(1 + \min \{ \rho_A (y), \rho_A (A^{-t} y) \}\bigr)
    \asymp 1 + |\det A|^{-t} + \rho_A (A^{-t} y) .
    \label{eq:VExplicitBound}
  \end{equation}
\end{lemma}

The basic properties of Peetre-type spaces are collected in the following lemma.

\begin{lemma}\label{lem:TranslationNormBounds}
  Let $\alpha \in \mathbb{R}, \beta > 0$, and $p \in (0, \infty)$ and $q \in (0, \infty]$.
  Then the Peetre-type space $\PT(G_A)$ is a
  solid quasi-Banach function space (Banach function space if $p,q \geq 1$).
  Furthermore, the operator norms of the translation operators $L_g$ and $R_g$ acting on $\PT(G_A)$
  can be bounded by
  \begin{align*}
    \vertiii{L_{(y,t)}}
    = |\det A|^{t (\alpha+1/p-1/q)}
    \quad \text{and} \quad
    \vertiii{R_{(y,t)}}
    \leq |\det A|^{- t ( \alpha-1/q)} (v(y,t))^{\beta},
  \end{align*}
  where $\vertiii{\,\cdot\,}:= \|\,\cdot\,\|_{\PT \to \PT}$ and $v$ is the weight function
  defined in Lemma~\ref{lem:SpecialWeight}.
\end{lemma}

\begin{proof}
  It is easy to see that $\| \cdot \|_{\PT}$ is a solid quasi-norm,
  as defined in \cite[Chapter~2]{VoigtlaenderPhDThesis}, and a solid \emph{norm} if $p,q \geq 1$.
  The positive definiteness of $\| \cdot \|_{\PT}$ follows from
  \Cref{lem:GroupMaximalProperties}.

  For the completeness of $\PT$, suppose that $(F_n)_{n \in \mathbb{N}}$ satisfies
  $\liminf_{n \to \infty} \| F_n \|_{\PT} < \infty$, and let $F \in \PT(G_A)$ be such that
  $|F(x, s)| \leq \liminf_{n \to \infty} |F_n (x, s)|$
  for a.e.\ $(x,s) \in \mathbb{R}^d \times \mathbb{R}$.
  Then it follows directly from Fatou's lemma and the definition of $\| \cdot \|_{\PT}$ that
  \begin{align}\label{eq:Fatou}
    \| F \|_{\PT}
    \leq \liminf_{n \to \infty}
           \| F _n \|_{\PT},
  \end{align}
  and thus $\PT$ satisfies the so-called \emph{Fatou property}, which
  in particular implies that $\PT$ is complete; see
  \cite[Section~65, Theorem~1]{zaanen1967integration}
  and \cite[Lemma~2.2.15]{VoigtlaenderPhDThesis} .

  We show the translation-invariance for $q \in (0, \infty)$.
  Let $F \in \PT(G_A)$ and ${(y,t) \in \mathbb{R}^d \times \mathbb{R}}$ be arbitrary.
  Then a direct calculation using the substitutions $\widetilde{x} = x - y$
  and $x = A^{-t} \widetilde{x}$, as well as $\widetilde{z} = A^{-t} z$ shows
  \begin{align*}
    \| L_{(y,t)} F \|_{\PT}
    &= \bigg\|
         x \mapsto
         \bigg(
           \int_{\mathbb{R}}
             \bigg[
               |\det A|^{\alpha s}
               \esssup_{z \in \mathbb{R}^d}
                 \frac{|F(A^{-t} (x+z-y), s-t) |}
                      {(1+\rho_A (A^{-s} z))^{\beta}}
             \bigg]^q
           \frac{ds}{|\det A|^{s}}
         \bigg)^{1/q}
       \bigg\|_{L^p} \\
    &= \bigg\|
         x \mapsto
         \bigg(
           \int_{\mathbb{R}}
             \bigg[
               |\det A|^{\alpha s}
               \esssup_{z \in \mathbb{R}^d}
                 \frac{|F(A^{-t} \widetilde{x} + A^{-t} z, s-t) |}
                      {(1+\rho_A (A^{-s} z))^{\beta}}
             \bigg]^q
           \frac{ds}{|\det A|^{s}}
         \bigg)^{1/q}
       \bigg\|_{L^p} \\
    &= |\det A|^{\frac{t}{p}}
       \bigg\|
         x \! \mapsto \!
         \bigg(
           \int_{\mathbb{R}}
             \bigg[
               |\det A|^{\alpha s}
               \esssup_{\widetilde{z} \in \mathbb{R}^d}
                 \frac{|F( x + \widetilde{z}, s-t) |}
                      {(1+\rho_A (A^{-(s-t)} \widetilde{z}))^{\beta}}
             \bigg]^q
           \frac{ds}{|\det A|^{s}}
         \bigg)^{1/q}
       \bigg\|_{L^p} \\
    &= |\det A|^{t/p} |\det A|^{t(\alpha-1/q)} \| F \|_{\PT}.
  \end{align*}

  For the right-translation, the substitutions $\widetilde{z} = z + A^s y$
  and $\widetilde{s} = s + t$ show that
  \begin{align*}
    \big\| R_{(y,t)} F \big\|_{\PT}
    &= \bigg\|
         x \! \mapsto \!
         \bigg(
           \int_{\mathbb{R}}
             \bigg[
               |\det A|^{\alpha s}
               \esssup_{z \in \mathbb{R}^d}
                 \frac{|F(x+z+A^sy, s+t) |}
                      {(1+\rho_A (A^{-s} z))^{\beta}}
             \bigg]^q
           \frac{ds}{|\det A|^{s}}
         \bigg)^{1/q}
       \bigg\|_{L^p} \\
    &= \bigg\|
         x \! \mapsto \!
         \bigg(
           \int_{\mathbb{R}}
             \bigg[
               |\det A|^{\alpha (\widetilde{s} - t)}
               \esssup_{\widetilde{z} \in \mathbb{R}^d}
                 \frac{|F(x+\widetilde{z}, \widetilde{s}) |}
                      {(1+\rho_A (A^{-\widetilde{s}}A^t \widetilde{z} - y))^{\beta}}
             \bigg]^q
           \frac{d\widetilde{s}}{|\det A|^{\widetilde{s}-t}}
         \bigg)^{1/q}
       \bigg\|_{L^p}.
  \end{align*}
  By \Cref{lem:SpecialWeight},
  $\bigl(1 + \rho_A (A^{-s} A^t z - y)\bigr)^{-1} \leq \frac{v(y,t)}{1 + \rho_A (A^{-s} z)}$
  for all $(z,s), (y,t) \in \R^d \times \R$, showing the desired estimate.
  The case $q = \infty$ follows via the usual modifications.
\end{proof}

Lastly, the following simple observation allows to apply
results of \cite{velthoven2022quasi} in the remainder.

\begin{lemma}\label{lem:PeetreNormIsRNorm}
  Let $\alpha \in \mathbb{R}, \beta > 0$.
  For $p \in (0, \infty)$, $q \in (0, \infty]$, let $r := \min\{1, p, q \}$.
  The quasi-norm $\| \cdot \|_{\PT}$ is an $r$-norm, i.e.,
  \[
    \|F_1 + F_2 \|^r_{\PT}
    \leq \| F_1 \|_{\PT}^r + \| F_2 \|^r_{\PT}
    \quad \text{for} \quad F_1, F_2 \in \PT.
  \]
\end{lemma}

\begin{proof}
  The case $p,q \geq 1$ follows directly by Lemma~\ref{lem:TranslationNormBounds},
  so let $p, q < 1$ throughout the proof.
  For $F_i \in \PT$ with $i = 1, 2$, define
  \[
    H_i (x, s)
    = |\det A|^{\alpha s}
      \esssup_{z \in \mathbb{R}^d}
        \frac{| F_i (x+z, s) |}{(1 + \rho_A (A^{-s} z))^{\beta}},
    \quad (x,s) \in \mathbb{R}^d \times \mathbb{R}.
  \]
  Using this notation and the inequalities $r = \min\{1, p, q\} < 1$ and $q/r, p/r \geq 1$,
  a direct calculation yields
  \begin{align*}
    \| F_1 + F_2 \|_{\PT}^r
    &= \bigg\|
         \bigg(
           \int_{\mathbb{R}}
             \bigg(
               |\det A|^{\alpha s} \,
               \esssup_{z \in \mathbb{R}^d}
                 \frac{|F_1(\cdot +z,s) + F_2 (\cdot+z, s)|}
                      {(1 + \rho_A (A^{-s} z))^{\beta}}
             \bigg)^{r \cdot \frac{q}{r}}
           \frac{ds}{|\det A|^{s}}
         \bigg)^{ \frac{r}{q} \cdot \frac{1}{r}}
       \bigg\|^r_{L^p}
       \\
    & \leq
      \bigg\|
        \bigg(
          \int_{\mathbb{R}}
            \bigg(
              H_1 (\cdot, s)^r + H_2 (\cdot, s)^r
            \bigg)^{\frac{q}{r}}
          \frac{ds}{|\det A|^{s}}
        \bigg)^{ \frac{r}{q} \cdot \frac{1}{r}}
      \bigg\|^r_{L^p} \\
    & \leq
      \bigg\|
        \big\| H_1^r \big \|_{L^{q/r} (\nu)}
        + \big\| H_2^r \big \|_{L^{q/r} (\nu)}
      \bigg\|_{L^{p/r}} \\
    & \leq \| F_1 \|_{\PT}^r + \| F_2 \|_{\PT}^r,
  \end{align*}
  where $\nu$ denotes the Borel measure on $\mathbb{R}$
  given by $\nu(M) = \int_M \frac{d s}{|\det A|^s}$ as in \Cref{eq:mixed_norm}.
\end{proof}

\subsection{Standard envelope and control weight}
\label{sub:ControlWeightConvolutionRelations}
The notion of a control weight plays an essential role in coorbit theory, see, e.g., \cite{feichtinger1989banach, groechenig1991describing, fuehr2015coorbit, velthoven2022quasi}. For the study of control weights in the setting of the present paper, the class of functions will be useful.

\begin{definition}
  \label{def:standard-env}
  For $\sigma = (\sigma_1, \sigma_2) \in (0,\infty)^2$ and $L \in \R$,
   define $\eta_L : G_A \to (0,\infty)$ and $\theta_\sigma : \R \to (0,\infty)$ by
  \begin{equation*}
    \eta_L(x,s)
    := \big(
         1 + \min
         \{
           \rho_A (x), \rho_A(A^{-s} x)
         \}
       \big)^{-L}
    \qquad \text{and} \qquad
    \theta_\sigma (s)
    := \begin{cases}
        \sigma_1^s , & \text{if } s \geq 0, \\
        \sigma_2^s , & \text{if } s < 0.
       \end{cases}
  \end{equation*}
  The \emph{standard envelope} $\Xi_{\sigma,L} : G_A \to (0,\infty)$ is given by
  $\Xi_{\sigma,L}(x,s):=  \theta_\sigma (s)  \eta_L(x,s)$.
\end{definition}

\begin{lemma}\label{lem:HFunctionAlternative}
  For each $L \in \R$,
  we have $\eta_L (x,s) \asymp \bigl(1 + \rho_A (A^{-\PosPart{s}} x)\bigr)^{-L}$
  for all $(x,s) \in G_A$.
\end{lemma}

\begin{proof}
  Corollary~\ref{cor:quasi-norm_bound} shows $\rho_A(A^{-s} x) \asymp |\det A|^{-s} \rho_A(x)$.
  Because of $|\det A| > 1$, this implies
  \[
    \min \{ \rho_A(x), \rho_A(A^{-s} x) \}
    \asymp \min \{ \rho_A(x), |\det A|^{-s} \rho_A(x) \}
    \asymp |\det A|^{-\PosPart{s}} \rho_A(x)
    \asymp \rho_A (A^{-\PosPart{s}} x) ,
  \]
  where Corollary~\ref{cor:quasi-norm_bound} was again used in the last step.
  This estimate easily implies the claim.
\end{proof}

The next lemma provides the existence of a so-called control weight for $\PT$
and shows how to estimate it by a standard envelope.

\begin{lemma}\label{lem:ControlWeights}
  Let $\alpha \in \R$, and $\beta > 0$. For $p \in (0, \infty)$, $q \in (0, \infty]$,
  let $r := \min\{1, p, q \}$.
  As in \Cref{lem:TranslationNormBounds},
  write $\vertiii{\,\cdot\,}:= \|\,\cdot\,\|_{\PT \to \PT}$.
  There exists a continuous, submultiplicative weight
  $w = w^{\alpha,\beta}_{p,q} : G_A \to [1,\infty)$ such that
  \begin{equation*}
    w(g) = \Delta^{1/r} (g^{-1}) \, w(g^{-1}),
     \qquad
     \vertiii{L_{g^{-1}}} \leq w(g),
    \qquad
    \vertiii{R_{g}}
    \leq w(g), \quad g \in G_A,
    \label{eq:ControlWeightProperties}
  \end{equation*}
  with implicit constant depending on $A,\beta$.
  The weight $w$ is called a \emph{standard control weight}.

  Furthermore, define
  $\sigma_1 := |\det A|^{1/r + |\alpha+1/p-1/q|}$ and
  $\sigma_2 := |\det A|^{-|\alpha+1/p-1/q|}$, as well as
  \begin{equation*}
    \kappa_1
    := \begin{cases}
         |\det A|^{1/r+\alpha+\beta-1/q} & \text{if } \alpha \geq -\frac{1/r+\beta-2/q}{2}, \\[0.1cm]
         |\det A|^{-(\alpha-1/q)}        & \text{otherwise} ,
       \end{cases}
  \end{equation*}
  and
  \begin{equation*}
    \kappa_2
    := \begin{cases}
         |\det A|^{-(\alpha+\beta-1/q)} & \text{if } \alpha \geq -\frac{1/r+\beta-2/q}{2}, \\[0.1cm]
         |\det A|^{1/r + \alpha - 1/q}      & \text{otherwise} .
       \end{cases}
    \label{eq:KappaDefinition}
  \end{equation*}
  Then the standard control weight $w$ satisfies
  \(
    w \asymp \Xi_{\sigma, 0} + \Xi_{\kappa, -\beta} .
  \)
\end{lemma}

\begin{proof}
  The weight $v : G_A \to [0,\infty)$ constructed in
  Lemma~\ref{lem:SpecialWeight} is submultiplicative, measurable, and
  locally bounded; see \Cref{eq:VExplicitBound}.
  Furthermore, $v \geq 1$.
  Thus, $v$ is a weight function in the sense of \cite[Definition~3.7.1]{reiter2000}
  and by the proof of \cite[Theorem~3.7.5]{reiter2000}, there exists a \emph{continuous},
  submultiplicative function $v_0 : G_A \to [1,\infty)$ satisfying $v \asymp v_0$.

  Let $\tau \in \R$ and set $a_{\tau}(g) = a_{\tau} (x,s) := |\det A|^{s \tau}$ for
  $g = (x,s) \in G_A$.
  Note that $a_{\tau}$ is multiplicative and that $\Delta = a_{-1}$.
  For $\gamma, \delta \in \R$, define the
  function $w_{\gamma, \delta} : G_A \to [1,\infty)$ by
  \begin{equation*}
    w_{\gamma, \delta} := \max \big\{ 1, \,\,\, a_{1/r}, \,\,\, a_{\gamma}, \,\,\,
    a_{-\gamma}, \,\,\, a_{\gamma + 1/r}, \,\,\, a_{1/r - \gamma}, \,\,\,
    a_{\delta + 1/r} \cdot (v_0^\vee)^\beta, \,\,\, a_{-\delta} \cdot
    v_0^\beta \big\}.
  \end{equation*}
  Then $w_{\gamma, \delta}$ is again continuous and submultiplicative.
  Since $a_{\tau}^{\vee} = a_{-\tau}$, it follows easily that
  $
    (\Delta^{1/r})^{\vee} \cdot w^{\vee}_{\gamma, \delta}
    = w_{\gamma, \delta} .
  $
  Choosing $\gamma:= \alpha + 1/p - 1/q$ and $\delta:= \alpha - 1/q$ and
  setting $w= w^{\alpha, \beta}_{p,q}:=w_{\gamma, \delta}$ yields,
  by Lemma~\ref{lem:TranslationNormBounds}, that
   $\vertiii{L_{g^{-1}}} = a_{-\gamma}(g) \leq w(g)$
  and
  \(
    \vertiii{R_{g}}
    \lesssim  \, a_{-\delta}(g) \, v_0(g)^{\beta}
    \leq w(g).
  \)

  For proving the second part of the lemma, note that $w \asymp w_1 + w_2$ for the weights given by
  ${w_1 := \max \{ a_0, a_{1/r}, a_\gamma, a_{-\gamma}, a_{1/r+\gamma}, a_{1/r-\gamma} \}}$ and
  \(
    w_2 := \max
           \bigl\{
              a_{\delta + 1/r} \cdot (v_0^{\vee})^\beta , \,\,
              a_{-\delta} \cdot  v_0^\beta
           \bigr\}
  \).
  It remains therefore to show that
  $w_1 \asymp \Xi_{\sigma,0}$ and $w_2 \asymp \Xi_{\kappa,-\beta}$, with
  $\kappa$ and $\sigma$ as in the statement of the lemma.
  To estimate $w_1$, note that if $I = \{ 0, 1/r, \gamma, -\gamma, 1/r + \gamma, 1/r - \gamma \}$,
  then
  \begin{align*}
    \max_{\tau \in I}
    a_{\tau} (x,s)
    &= \begin{cases}
      |\det A|^{s \cdot \max I}, & \text{if } s \geq 0, \\
      |\det A|^{s \cdot \min I}, & \text{if } s <    0, \\
    \end{cases} \\
     &= \begin{cases}
      |\det A|^{s \cdot (1/r + |\gamma|)}, & \text{if } s \geq 0, \\
      |\det A|^{- s |\gamma|}, & \text{if } s <    0. \\
    \end{cases}
  \end{align*}
  Hence, by the choice of $\gamma$ and $\sigma$, this yields
  $w_1 (x,s) = \max_{\tau \in I} a_{\tau} (x,s) = \theta_\sigma (s) = \Xi_{\sigma,0}(x,s)$.
  Lastly, for estimating $w_2$, note that the estimate for $v$
  in Lemma~\ref{lem:SpecialWeight} implies
  \begin{align*}
    v_0^\vee (x,s)
    &   \asymp \max \{ 1, |\det A|^s \}
              \big( 1 + \min \{ \rho_A (-A^{-s} x), \,\, \rho_A(-A^s A^{-s} x) \} \big) \\
    & =    |\det A|^{\PosPart{s}}  \big( 1 + \min \{ \rho_A(x), \rho_A(A^{-s} x) \} \big) \\
    &  =    |\det A|^{\PosPart{s}}  \eta_{-1} (x,s) .
  \end{align*}
  Similarly, one can show that $v_0(x,s) \asymp |\det A|^{\NegPart{s}} \, \eta_{-1} (x,s)$.
  In case of $s \geq 0$, this gives
  \begin{align*}
    w_2(x,s)
    &\asymp \big( \eta_{-1}(x,s) \big)^{\beta}
     \max \big\{
            |\det A|^{(1/r + \delta + \beta) s} ,
            |\det A|^{-\delta s}
          \big\} \\
    &= \eta_{-\beta}(x,s)  \kappa_1^s \\
    &= \Xi_{\kappa,-\beta} (x,s) ,
  \end{align*}
  since
  \(
    \max \{
           1/r + \delta + \beta,
           -\delta
         \}
    = \max
      \{
        1/r + \alpha + \beta - 1/q,
        -\alpha + 1/q
      \}
    = 1/r + \alpha + \beta - 1/q
  \)
  if and only if $\alpha \geq - \frac{1/r+\beta-2/q}{2}$.
  The estimate for $s < 0$ follows similarly.
\end{proof}

\subsection{Norm estimates}
\label{sub:NormEstimates}

Let $Q \subset G_A $ be a relatively compact unit-neighborhood.
The \emph{two-sided local maximal function} $M_Q F$ of a measurable
function $F \!:\! G_A \! \to \! \CC$ is defined by
\begin{equation}
  M_Q F (g)
  := \esssup_{u,v \in Q} |F(u g v)|.
  \label{eq:LocalMaximalFunction}
\end{equation}
Two properties of this maximal function that will be used below are its measurability
(see, e.g.,~\cite[Lemma~B.4]{holighaus2020schur})
and the estimate $|F| \leq M_Q F$ a.e.\ (see,  e.g., \cite[Lemma~2.3.3]{VoigtlaenderPhDThesis}).

For $p \in (0, \infty)$, $q \in (0, \infty]$, let $r := \min\{ 1, p, q\}$.
The (weighted) \emph{Wiener amalgam space} $\WLwr$ is defined by
\[
  \WLwr
  := \mathcal{W}_Q (L_w^r)
  := \bigg\{
       F \in C (G_A)
       :
       \MQ F \in L^r_w (G_A)
     \bigg\},
\]
where $w : G_A \to [1,\infty)$ is a standard control weight
for $\PT$ as provided by Lemma~\ref{lem:ControlWeights}.

The space $\WLwr$ is independent of the choice of $Q$.%
\footnote{Given two relatively compact unit neighborhoods $P,Q \subset G_A$, one can write
$Q \subset \bigcup_{i=1}^N x_i P$ and $Q \subset \bigcup_{j=1}^M P y_j$, and this easily implies
$M_Q F \leq \sum_{i,j} R_{x_i} L_{y_j^{-1}} (M_P F)$.
The result then follows by noting that $L_w^r(G_A)$ is invariant under left- and right-translations,
since $w$ is submultiplicative.}
In particular, this implies that
\begin{equation}
  F \in \WLwr
  \quad \text{if and only if} \quad
  F^{\vee} \in \WLwr;
  \label{eq:WienerSpaceSymmetry}
\end{equation}
since the condition $w(g) = \Delta^{1/r} (g^{-1}) w(g^{-1})$
in Lemma~\ref{lem:ControlWeights} implies $\| F^\vee \|_{L_w^r} = \| F \|_{L_w^r}$,
and by choosing $Q$ to be symmetric it follows that $M_Q (\Phi^{\vee}) = (M_{Q} \Phi)^{\vee}$.

The following norm estimate will be used repeatedly in the remainder.

\begin{lemma}\label{lem:wavelet-Lpq}
  Let $Q \subset G_A$
  be a relatively compact unit neighborhood.
  Let $\psi \in \SC_0(\R^d)$ and let $w: G_A \to [0, \infty)$ be any weight
  such that $w \lesssim \Xi$, where $\Xi$ is a linear combination of standard envelopes
  (see Definition~\ref{def:standard-env}).

  Then, for all $p,q \in (0,\infty]$, there exists a continuous Schwartz seminorm $\| \cdot \|$
  such that
  \[
    \| W_\psi \varphi \|_{L_w^{p,q}}
    \leq \| M_Q (W_\psi \varphi) \|_{L_w^{p,q}}
    \lesssim \| \varphi \|
  \]
  for all $\varphi \in \Schwartz_0(\R^d)$;
  in particular, $ W_{\psi} \varphi \in \WLwr$ for all $r \in (0, \infty]$.
\end{lemma}

\begin{proof}
  Let $1 < \lambda_- < \min_{\lambda \in \sigma(A)} |\lambda|$.
  By Corollaries~\ref{cor:wavelet-decay} and \ref{cor:quasi-norm_bound}
  and Lemma~\ref{lem:HFunctionAlternative}, it follows that
  for all $L,N \in \N$ and $\varphi \in \SC_0(\R^d)$,
  \begin{align*}
    |W_{\psi}\varphi (x,s)|
    & \lesssim \|\varphi\|
                (1 + \rho_A(x))^{-L} \,
                \lambda_{-}^{-|s| N}
      \lesssim \| \varphi \|
                \bigl(1 + |\det A|^{-\PosPart{s}} \, \rho_A(x)\bigr)^{-L}
                \lambda_{-}^{-|s| N} \\
    & \lesssim \| \varphi \|  \Xi_{L, \tau} (x,s),
  \end{align*}
  where $\tau := (\lambda_-^{-N}, \lambda_-^{N})$ and a suitable continuous Schwartz seminorm
  $\| \cdot \|$, depending on $L, N$.
  Lemma~\ref{lem:StandardEnvelopeWienerMaximalFunction}
  yields $M_Q \Xi_{L,\tau} \lesssim \Xi_{L,\tau}$, and hence
  $M_Q [W_\psi \varphi] \lesssim \| \varphi \|  \Xi_{L,\tau}$.
  In addition, Lemma~\ref{lem:HFunctionAlternative} shows that
  \(
    \eta_L (x,s)
    \asymp \bigl(1 + |\det A|^{-\PosPart{s}} \rho_A(x)\bigr)^{-L}
    \leq |\det A|^{|s| L}  (1 + \rho_A (x))^{-L} .
  \)
  Therefore,
  \begin{equation}
    M_Q [W_\psi \varphi] (x,s)
    \lesssim \| \varphi \|
              (1 + \rho_A(x))^{-L}
              (|\det A|^{|L|} / \lambda_{-}^N)^{|s|} ,
    \label{eq:WaveletLpqMainEstimate}
  \end{equation}
  where $L,N \in \N$ are arbitrary and $\| \cdot \|$ is a continuous Schwartz semi-norm
  depending on $N,L$.

  \smallskip{}

  It clearly suffices to prove the claim for the case where
  $w = \theta_{\sigma} \cdot \eta_M$ is a standard envelope, with
  $\sigma = (\sigma_1, \sigma_2) \in (0, \infty)^2$ and $M \in \R$.
  Define $\widetilde{\sigma}:= \max \{ \sigma_1, \sigma_2^{-1} \}$;
  then $\theta_{\sigma}(s) \leq \widetilde{\sigma}^{|s|}$ for all $s \in \R$.
  Since $\eta_M \leq 1 = \eta_0$ for $M \geq 0$, we may assume that $M \leq 0$.
  Then, Lemma~\ref{lem:HFunctionAlternative} and Corollary~\ref{cor:quasi-norm_bound} imply
  \[
    \eta_M(x,s)
    \lesssim \bigl(1 + \rho_A(A^{-\PosPart{s}} x)\bigr)^{-M}
    \lesssim \bigl(1 + |\det A|^{-\PosPart{s}} \rho_A(x) \bigr)^{|M|}
    \leq (1 + \rho_A(x))^{|M|}.
  \]
 Since $|\det A|^{-\frac{s}{q}} \leq |\det A|^{|s|/q}$, we thus have
  \begin{align} \label{eq:RHS_F}
    |\det A|^{-\frac{s}{q}}  w(x,s)  M_Q [W_\psi \varphi] (x,s)
    \lesssim \| \varphi \|
              \bigl(1 + \rho_A(x)\bigr)^{|M|-L}
              \bigl(
                     |\det A|^{|L| + \frac{1}{q}} \, \widetilde{\sigma} / \lambda_{-}^N
                   \bigr)^{|s|}.
  \end{align}
  Therefore, choosing $L, N$  sufficiently large,
  it is an easy consequence of Lemma~\ref{lem:QuasiNormIntegrability} that
  \[
    \| M_Q [W_\psi \varphi] \|_{L_w^{p,q}}
    = \big\|
        |\det A|^{-\frac{s}{q}} \,
        w(\cdot) \,
        M_Q[W_\psi \varphi] (\cdot)
      \big\|_{L^{p,q}}
    \lesssim \| \varphi \| ,
  \]
  which completes the proof.
\end{proof}

\subsection{Coorbit spaces}
\label{sub:CoorbitSpaces}

This section proves wavelet characterizations of anisotropic Triebel-Lizorkin spaces
by identifying them with so-called \emph{coorbit spaces}
(cf.\ \cite{feichtinger1989banach, velthoven2022quasi}).

For technical reasons, coorbit spaces associated with quasi-Banach function spaces
are commonly defined in terms of merely \emph{left} local maximal functions.
For a function $F \in L^{\infty}_{\loc} (G_A)$, its left maximal function is defined by
\[
  M_Q^L F (g) = \esssup_{u \in Q} |F(g u)|, \quad g \in G_A,
\]
where $Q \subset G_A$ is a relatively compact unit neighborhood.

\begin{definition}\label{def:PeetreCoorbitDefinition}
  Let $p \in (0,\infty),q \in (0,\infty]$, $\alpha \in \R$, and $\beta > 0$.
  Let $A \in \mathrm{GL}(d, \mathbb{R})$ be expansive  exponential
  and let $Q \subset G_A$ be a relatively compact, symmetric unit neighborhood.

  For an admissible vector $\psi \in \SC_0(\mathbb{R}^d)$,
  the \emph{coorbit space} $\Co (\PT) = \Co_{\psi} (\PT)$ is the collection
  of all $f \in \SC_0'(\R^d)$ satisfying
  \[
    \| f \|_{\Co(\PT)} = \| f \|_{\Co_{\psi} (\PT)}
    = \big\| M^L_Q (W_{\psi} f) \big\|_{\PT}
    < \infty,
  \]
  and equipped with the norm $\| \, . \, \|_{\Co(\PT)}$.
\end{definition}

In the above definition, note that there exist admissible vectors $\psi \in \SC_0(\R^d)$
by Theorem~\ref{thm:AdmissibleVectorsExistence}.

\begin{remark} \label{rem:basic_coorbit}
The space $\Co(\PT)$ defined in \Cref{def:PeetreCoorbitDefinition}
can be identified with the abstract coorbit spaces defined in \cite[Definition~4.7]{velthoven2022quasi}.
In particular, several basic properties of coorbit spaces, such as independence
of the defining vector $\psi \in \mathcal{S}_0 (\mathbb{R}^d)$ and neighborhood $Q$,
follow directly from the theory \cite{velthoven2022quasi}. See \Cref{lem:abstract_identification}
for details on the identification.
\end{remark}

Anisotropic Triebel-Lizorkin spaces are identified with coorbit spaces by the following proposition.

\begin{proposition}\label{prop:TL_coorbit}
  Let $p \in (0, \infty)$, $q \in (0, \infty]$, $\alpha \in \mathbb{R}$,
  and $\beta > \max \{ 1/p, 1/q \}$.
  Then
  \begin{equation*}
    \TL = \Co \bigl(\PTalt{-\alpha'}\bigr)
    \qquad \text{for} \qquad
    \alpha' = \alpha + \frac{1}{2} - \frac{1}{q}
    .
  \end{equation*}
\end{proposition}

\begin{proof}
Throughout, let $\psi \in \SC_0(\R^d)$ be admissible \eqref{eq:admissible-vector}
with compact Fourier support in $\R^d \setminus \{0\}$
satisfying the support condition \eqref{eq:analyzing_positive}.
The existence of such vectors is guaranteed by Theorem~\ref{thm:AdmissibleVectorsExistence}.
Furthermore, let $Q := [-1, 1)^d \times [-1,1)$.
The space $\Co(\PT)$ is independent of these choices by Remark~\ref{rem:basic_coorbit}.
To ease notation, set $\alpha' := \alpha+1/2-1/q$.

The proof is split into three steps.
\\\\
\textbf{Step 1.}
  Let $\psi \in \SC_0(\R^d)$ be an admissible vector \eqref{eq:admissible-vector}
  with compact Fourier support in $\R^d \setminus \{0\}$ satisfying \eqref{eq:analyzing_positive}.
  Then also $\psi^\ast \in \SC_0(\R^d)$ satisfies \eqref{eq:analyzing_positive},
  where $\psi^\ast (t) := \overline{\psi}(-t)$.
  Since
  \begin{equation*}
    \pi(x,s)\psi
    = |\det A|^{-s/2} \, \psi(A^{-s}( \,\cdot - x))
    = |\det A|^{s/2} \, T_x \psi_{-s},
  \end{equation*}
  it follows that
  \begin{align} \label{eq:wavelet_conv}
    W_{\psi}f(x,s) = \langle f, \pi(x,s) \psi \rangle
    = \langle f , |\det A|^{s/2} T_x \psi_{-s} \rangle
    = |\det A|^{s/2} f \ast \psi^\ast_{-s}(x) .
  \end{align}
  Therefore, using $\psi^\ast$ as the analyzing vector in Theorem~\ref{thm:norm_equiv} yields
  \begin{align*}
    \|f\|_{\TL}
    &\asymp \bigg\|
              x \mapsto
              \bigg(
                \int_{\mathbb{R}}
                  \bigg(
                    |\det A|^{-\alpha s}
                    \sup_{z \in \mathbb{R}^d}
                      \frac{|(f \ast \psi^\ast_{-s}) (x+z)|}
                           {(1 + \rho_{A}(A^{-s} z))^{\beta}}
                  \bigg)^q
                ds
              \bigg)^{1/q}
            \bigg\|_{L^p}  \\
       & = \bigg\|
             x \mapsto
             \bigg(
               \int_{\mathbb{R}}
                 \bigg(
                   |\det A|^{-(\alpha+1/2-1/q) s}
                   \sup_{z \in \mathbb{R}^d}
                     \frac{|W_\psi f (x+z,s)|}
                          {(1 + \rho_{A}(A^{-s} z))^{\beta}}
                 \bigg)^q
               \frac{ds}{|\det A|^{s}}
             \bigg)^{1/q}
           \bigg\|_{L^p} \numberthis \label{eq:TL_normequiv} \\
     &= \|W_\psi f\|_{\PTalt{-\alpha'}}
  \end{align*}
  for any $f \in \SC_0'(\R^d)$.
\\\\
\textbf{Step 2.} Since $|F| \leq M^L_Q F$ a.e.\ on $G_A$ for $F \in L^1_{\loc} (G_A)$
(see, e.g., \cite[Lemma~2.3.3]{VoigtlaenderPhDThesis}), it follows by Step~1 that
\[
  \| f \|_{\TL}
  \asymp \|W_\psi f\|_{\PTalt{-\alpha'}}
  \leq   \| M^L_Q (W_{\psi} f) \|_{\PTalt{-\alpha'}}
  =      \| f \|_{\Co (\PTalt{-\alpha'})},
\]
for $f \in \SC_0'(\R^d)$.
\\\\
\textbf{Step 3.} This step will show the remaining estimate
$\| f \|_{\Co (\PTalt{-\alpha'})} \lesssim \| f \|_{\TL}$ for $f \in \SC_0'(\R^d)$.
Parts of the used arguments resemble Step~2 in the proof of Theorem~\ref{thm:norm_equiv}
and will for this reason only be sketched.

First, a direct calculation using the involved definitions and a
change-of-variable yields that
\begin{align*}
  \esssup_{z \in \mathbb{R}^d}
  \frac{|M^L_Q (W_\psi f) (x+z,s)|}
  {(1 + \rho_{A}(A^{-s} z))^{\beta}}
  &= \esssup_{\substack {z \in \mathbb{R}^d \\  (y,t) \in Q}}
  \frac{|W_\psi f (x + z + A^s y,s+t)|}
  {(1 + \rho_{A}(A^{-s} z))^{\beta}} \\
  &= \esssup_{\substack {z \in \mathbb{R}^d \\  (y,t) \in Q}}
  \frac{|W_\psi f (x + z,s+t)|}
  {(1 + \rho_{A}(A^{-s} z - y))^{\beta}} \\
  &\lesssim_{A, \beta} \esssup_{\substack {z \in \mathbb{R}^d \\  t \in [-1,1)}}
  \frac{|W_\psi f (x + z,s+t)|}
  {(1 + \rho_{A}(A^{-s} z))^{\beta}},
  \numberthis \label{eq:maxfunction-estimate}
\end{align*}
where the last inequality follows from
$(1+ \rho_A (A^s z - y))^{-\beta}  \lesssim (1+ \rho_A (A^s z))^{-\beta}$ for $ y \in [-1,1)^d$.

For fixed $s \in \mathbb{R}$ and $t \in [-1,1)$, the identity \eqref{eq:wavelet_conv}
and \Cref{cor:quasi-norm_bound} allows to estimate
\begin{align*}
  |\det A|^{-s/2} \esssup_{\substack {z \in \mathbb{R}^d}}
  \frac{ |W_{\psi} f (x + z, s + t)|}
  {(1 + \rho_{A}(A^{-s} z))^{\beta}}
  &\lesssim_A
    \esssup_{\substack {z \in \mathbb{R}^d}}
    \frac{ |(f \ast \psi^*_{-(s+t)}) (x + z)|}
    {(1 + \rho_{A}(A^{-s} z))^{\beta}} \\
  &\lesssim_{A, \beta}
    \esssup_{\substack {z \in \mathbb{R}^d}}
    \frac{ |(f \ast \psi^*_{-(s+t)}) (x + z)|}
    {(1 + \rho_{A}(A^{-(s+t)} z))^{\beta}} \\
  &=  (\psi^*)^{**}_{-(s+t), \beta} f(x).
\end{align*}
This, combined with \eqref{eq:maxfunction-estimate} and
$|\det A|^{-\alpha t} \lesssim_{A,\alpha} 1$ for $t \in [-1,1)$,
yields that
\begin{align*}
  |\det A|^{-s/2}
  \esssup_{z \in \mathbb{R}^d}
  \frac{|M^L_Q (W_\psi f) (x+z,s)|}
  {(1 + \rho_{A}(A^{-s} z))^{\beta}}
  & \lesssim_{A,\beta}
    |\det A|^{-s/2}
    \esssup_{\substack {z \in \mathbb{R}^d \\  t \in [-1,1)}}
  \frac{|W_\psi f (x + z,s+t)|}
  {(1 + \rho_{A}(A^{-s} z))^{\beta}} \\
  & \lesssim_{A,\beta}
    \esssup_{t \in [-1,1)}
    (\psi^*)^{**}_{-(s+t), \beta} f(x) \\
  &\lesssim_{A, \alpha}
    \esssup_{t \in [-1,1)}
     |\det A|^{-\alpha t}
    (\psi^*)^{**}_{-(s+t), \beta} f(x).
\end{align*}

Now let $q < \infty$.  Then arguments similar to proving
\eqref{eq:det_Peetre} yield $N \in \mathbb{N}$ such that
\begin{align*}
  \bigg(
    |\det A|^{-\alpha (s+t)}
                 (\psi^*)^{**}_{-(s+t), \beta} f(x)
  \bigg)^q
  \lesssim \sum_{\ell = - N}^N
           \bigg(
             |\det A|^{- \alpha (s+\ell)} (\psi^*)^{**}_{-(s+\ell), \beta} f(x)
           \bigg)^q.
\end{align*}
The right-hand side being independent of $t \in [-1,1)$, it follows that
\begin{align*}
  &\bigg(|\det A|^{-(\alpha+1/2)s} \esssup_{\substack {z \in \mathbb{R}^d}}
  \frac{ | M^L_Q(W_{\psi} f) (x + z, s)|}
  {(1 + \rho_{A}(A^{-s} z))^{\beta}}  \bigg)^q \\
  &\quad \quad \quad \quad \quad \quad \quad \quad \lesssim \!\esssup_{t \in [-1,1)}
    \bigg(|\det A|^{-\alpha(s+t)}
    (\psi^*)^{**}_{-(s+t), \beta} f(x)
    \bigg)^q \\
  &\quad \quad \quad \quad \quad \quad \quad \quad \lesssim \sum_{\ell = - N}^N \bigg ( |\det A|^{- \alpha (s+\ell)} (\psi^*)^{**}_{-(s+\ell), \beta} f (x)
    \bigg)^q.
    \numberthis \label{eq:estimate-for-part2}
\end{align*}
Combining this estimate with \Cref{thm:norm_equiv} gives
\begin{align*}
   \| f \|_{\Co(\PTalt{-\alpha'})}
    &= \bigg\|
    \bigg(
    \int_{\mathbb{R}}
    \bigg(
    |\det A|^{-(\alpha+1/2-1/q) s}
    \esssup_{\substack {z \in \mathbb{R}^d}}
    \frac{ | M^L_Q(W_{\psi} f) (\cdot + z, s)|}
    {(1 + \rho_{A}(A^{-s} z))^{\beta}}
    \bigg)^q
    \frac{ds}{|\det A|^{s}}
    \bigg)^{1/q}
    \bigg\|_{L^p} \\
  & \quad \quad \quad
    \lesssim \bigg\|
    \bigg(
    \sum_{\ell = - N}^N
    \int_{\mathbb{R}}
    \bigg ( |\det A|^{- \alpha (s+\ell)} (\psi^*)^{**}_{-(s+\ell), \beta} f
    \bigg)^q
    ds
    \bigg)^{1/q}
    \bigg\|_{L^p}    \\
  & \quad \quad \quad
     \lesssim \bigg\|
             \bigg(
               \int_{\mathbb{R}}
                 \big ( |\det A|^{ \alpha s} (\psi^*)^{**}_{s, \beta} f  \big)^q
               ds
             \bigg)^{1/q}
           \bigg\|_{L^p} \\
  & \quad \quad \quad
    \asymp \| f \|_{\TL}.
\end{align*}
The case $q = \infty$ follows from \eqref{eq:estimate-for-part2} by similar arguments.
\end{proof}

\begin{remark} \label{rem:coorbit_improving}
For $p \in [1,\infty)$ and $q \in [1,\infty]$, the coorbit spaces $\Co(\PT)$
of \Cref{def:PeetreCoorbitDefinition} are genuine Banach spaces,
which are well-known to admit the equivalent norm
\[
  \| f \|_{\Co (\PT)}
  := \| M^L_Q (W_{\psi} f ) \|_{\PT}
  \asymp \|W_{\psi} f  \|_{\PT},
\]
see, e.g., \cite[Theorem~8.3]{feichtinger1989banach2} and \cite[Proposition~4.10]{velthoven2022quasi}.
The proof of Proposition~\ref{prop:TL_coorbit} shows that the same holds for the Peetre-type spaces
$\PT$ in the quasi-Banach range $\min \{ p,q \} < 1$.
\end{remark}

\section{Molecular characterizations}
\label{sec:molecular}

This section provides new molecular characterizations of anisotropic Triebel-Lizorkin spaces.
The results will be obtained from \cite{velthoven2022quasi, romero2020dual}
by exploiting the coorbit identification of Triebel-Lizorkin spaces
provided by \Cref{prop:TL_coorbit}.

\subsection{Peetre-type sequence space}
\label{sec:peetre_discrete}

Let $\Gamma \subset G_A$ be arbitrary and let $U \subset G_A$
be a relatively compact unit neighborhood.
The set $\Gamma$ is \emph{relatively separated} if
\[
  \sup_{g \in G}
    \# \big( \Gamma \cap g U)
  = \sup_{g \in G}
      \sum_{\gamma \in \Gamma}
        \Indicator_{\gamma U^{-1}} (g)
  < \infty
\]
and is called \emph{$U$-dense} if $G = \bigcup_{\gamma \in \Gamma} \gamma U$.
The set $\Gamma$ is \emph{$U$-separated} in $G$ if $\mu_{G_A} (\gamma U \cap \gamma' U) = 0$
for all $\gamma, \gamma' \in \Gamma$ with $\gamma \neq \gamma'$
and is \emph{separated} if it is $U$-separated for some unit neighborhood $U$.
Any separated set is relatively separated.
Furthermore, the notion of being relatively separated is independent of
the choice of the relatively compact unit neighborhood $U$.

\begin{definition} \label{def:peetre_seq}
  Let $\Gamma \subset G_A$ be relatively separated and
  let $U \subset G_A$ be a relatively compact unit neighborhood.
  For $p \in (0,\infty),q \in (0,\infty]$, $\alpha \in \R$, and $\beta > 0$,
  the sequence space $\PTseq (\Gamma, U)$ associated to the Peetre-type space $\PT (G_A)$
  is defined as the set of all
  $c = (c_{\gamma})_{\gamma \in \Gamma} \in \mathbb{C}^{\Gamma}$ such that
  \[
    \| c \|_{\PTseq}
    := \bigg\|
         \sum_{\gamma \in \Gamma}
           |c_{\gamma}| \Indicator_{\gamma U}
       \bigg\|_{\PT}
    < \infty
  \]
  and equipped with the (quasi)-norm $\| \cdot \|_{\PTseq}$.
\end{definition}

The sequence space $\PTseq (\Gamma, U)$ is a well-defined quasi-Banach space,
independent of the choice of the defining neighborhood $U$;
see, e.g.,~\cite{rauhut2007wiener,feichtinger1989banach}
or \cite[Lemma~2.3.16]{VoigtlaenderPhDThesis}.

\begin{remark}\label{rem:sequence_regular}
  The Triebel-Lizorkin space $\TLseq$ defined in \eqref{eq:TL_sequence}
  can be identified with a sequence space $\PTseq$ via Theorem~\ref{thm:maximal_sequence}.
  To be more explicit, if $\Gamma = \{(A^{-j} k, -j) : j \in \mathbb{Z}, k \in \mathbb{Z}^d \}$,
  then the map
  \[
    \PTseq\big(\Gamma, [-1,1)^d \times [-1,1)\big) \to \TLseq[-(\alpha + \frac{1}{2} - \frac{1}{q})],
    \quad
    (c_\gamma)_{\gamma \in \Gamma}
    \mapsto \big( c_{(A^{-j} k, -j)} \big)_{(j,k) \in \Z \times \Z^d}
  \]
  is an isomorphism of (quasi)-Banach spaces, for
  any $p \in (0,\infty)$, $q \in (0,\infty]$, $\alpha \in \mathbb{R}$
  and $\beta > \max\{1/p, 1/q \}$.
\end{remark}

\subsection{Molecular systems and the extended pairing}
\label{sub:MoleculesAndOperators}

Following \cite{grochenig2009molecules, velthoven2022quasi, romero2020dual},
the notion of molecular systems used in this paper is defined through properties
of the associated wavelet transform.

\begin{definition}
  Let $\Gamma \subset G_A$ be  relatively separated
  and let $\psi \in L^2 (\mathbb{R}^d)$ be an admissible vector
  such that $W_{\psi} \psi \in \WLwr$, where $w = w_{p,q}^{\alpha, \beta} : G_A \to [1, \infty)$
  is the standard control weight of Lemma~\ref{lem:ControlWeights}.

  A family $(\phi_{\gamma})_{\gamma \in \Gamma}$ of vectors $\phi_{\gamma} \in L^2 (\mathbb{R}^d)$
  is an \emph{$L^r_w$-molecular system} if there exists an \emph{envelope}
  $\Phi \in \WLwr$ such that
  \begin{align} \label{eq:molecule_condition}
    |W_{\psi} \phi_{\gamma} (x) |
    \leq L_{\gamma} \Phi (x),
    \end{align}
  for $x \in G_A$ and $\gamma \in \Gamma$.
\end{definition}

\begin{remark} \label{rem:molecule_independent}
  The condition \eqref{eq:molecule_condition} is independent
  of the choice of the window $\psi$ in the following sense:
  If $\psi,\varphi \in L^2(\R^d)$ are both admissible satisfying
  $W_\psi \varphi, W_\psi \psi, W_\varphi \varphi \in \WLwr$, then
  $(\phi_\gamma)_{\gamma \in \Gamma} \subset L^2(\R^d)$ is a molecular system
  with respect to the window $\psi$ if and only if the same holds with respect to the window
  $\varphi$; see \cite[Lemma~6.3]{velthoven2022quasi}.
\end{remark}

In order to treat molecular systems consisting of general vectors in $L^2 (\mathbb{R}^d)$
in a meaningful manner, we define the following extended dual pairing.

\begin{definition}\label{def:ExtendedDualPairing}
Let $\psi \in \mathcal{S}_0 (\mathbb{R}^d)$ be admissible.
For $f \in \mathcal{S}'_0 (\mathbb{R}^d)$ and $\phi \in L^2 (\mathbb{R}^d)$,
define the \emph{extended pairing} as
\[
  \langle f, \phi \rangle_{\psi}
  := \int_{G_{A}}
       \langle f, \pi(x,s) \psi \rangle \,\,
       \overline{\langle \phi, \pi (x,s) \psi \rangle_{L^2}} \,\,
     d\mu_{G_A} (x,s) ,
\]
provided that the integral converges.
\end{definition}

\begin{remark}
  Let $\psi \in \mathcal{S}_0 (\mathbb{R}^d)$ be admissible.
  \begin{enumerate}[(a)]
    \item If $f \in \mathcal{S}'_0 (\mathbb{R}^d)$ and $\phi \in \mathcal{S}_0 (\mathbb{R}^d)$,
          then the extended pairing $\langle f, \phi \rangle_{\psi}$ coincides with the standard
          conjugate linear pairing $\langle f, \phi \rangle := f (\overline{\phi})$
          by Lemma~\ref{lem:weak-continuity}.

    \item If both $f, \phi \in L^2 (\mathbb{R}^d)$, then $\langle f, \phi \rangle_{\psi}$ coincides
          with the $L^2$-inner product $\langle f, \phi \rangle$
          by Equation~\eqref{eq:ReconstructionFormula}.
  \end{enumerate}
\end{remark}

For showing that the extended pairing defined in \Cref{def:ExtendedDualPairing}
is well-defined, in the sense that it does not depend on the choice of admissible vectors,
the following approximation property will be used.

\begin{lemma}\label{lem:PseudoWeakStarDensity}
  Let $f \in \mathcal{S}_0' (\mathbb{R}^d)$ and let $\psi \in \mathcal{S}_0 (\mathbb{R}^d)$
  be admissible with $W_\psi f \in L_{1/w}^\infty(G_A)$,
  where $w : G \to [1,\infty)$ denotes the standard control weight provided by
  \Cref{lem:ControlWeights}.

  There exists a sequence $(f_n)_{n \in \N}$ of functions $f_n \in L^2(\R^d)$
  with the following properties:
  \begin{enumerate}[(i)]
    \item As $n \to \infty$, $f_n \to f$ with weak-$\ast$-convergence
          in $\mathcal{S}_0' (\mathbb{R}^d)$;

    \item For each $\varphi \in \mathcal{S}_0 (\mathbb{R}^d)$,
          there is a constant $C = C(\varphi,\psi,f) > 0$ such that
          \[ |W_\varphi f_n (g)| \leq C  w(g), \quad g \in G_A.\]
  \end{enumerate}
\end{lemma}

\begin{proof}
  For $n \in \N$, define $\Omega_n := [-n,n]^d \times [-n,n]$ and
  $F_n := W_\psi f \cdot \Indicator_{\Omega_n}$.
  Note that since $w$ is continuous and $\Omega_n \subset G_A$ is compact,
  for each $n \in \N$ there is $C_n > 0$ satisfying $w(x) \leq C_n$ for all $x \in \Omega_n$.
  This implies
  \(
    |F_n(\cdot)|
    \leq C_n
         \| W_\psi f \|_{L_{1/w}^\infty}
         \Indicator_{\Omega_n}(\cdot)
    \in L^1(G_A)
    .
  \)
  Since  $g \mapsto \pi(g) \psi$ is continuous from $G_A$ into $L^2(\R^d)$ and
  $\| F_n(\cdot) \, \pi(\cdot)\psi \|_{L^2} \leq \| \psi \|_{L^2}  |F_n(\cdot)| \in L^1(G_A)$,
  this shows that ${f_n := \int_{G_A} F_n(g) \, \pi(g) \psi \, d \mu_{G_A}(g) \in L^2(\R^d)}$
  is well-defined as a Bochner integral.

  Let $\varphi \in \mathcal{S}_0 (\mathbb{R}^d)$ be arbitrary.
  For $h \in G_A$, a direct calculation gives
  \begin{align*}
    |W_\varphi f_n (h)|
          & \leq \int_{G_A}
             |F_n(g)|  |\langle \pi(g) \psi, \pi(h) \varphi \rangle|
           \, d \mu_{G_A}(g)  \\
           &\leq \int_{G_A}
             |W_\psi f(g)|  |W_\psi \varphi (h^{-1} g)|
           \, d \mu_{G_A}(g) \\
    & \leq \| W_\psi f \|_{L_{1/w}^{\infty}} \!
           \int_{G_A} \!\!\!\!
             w(h) \, w(h^{-1} \! g) \, |W_\psi \varphi (h^{-1} g)|
           \, d \mu_{G_A}(g) \\
    & \leq w(h) \, \| W_\psi f \|_{L_{1/w}^{\infty}} \, \| W_\psi \varphi \|_{L_w^1} ,
  \end{align*}
  where $\| W_\psi \varphi \|_{L_w^1} < \infty$ by \Cref{lem:wavelet-Lpq}.
  This proves (ii).

  To prove (i), applying the dominated convergence theorem
  and Corollary~\ref{cor:reproducing-formula-distr} gives
  \begin{align*}
   \lim_{n \to \infty} W_\psi f_n (h)
    &= \lim_{n \to \infty} \int_{G_A}
          F_n (g)  \langle \pi(g) \psi, \pi(h) \psi \rangle
        \, d \mu_{G_A} (g) \\
    &=
      (W_\psi f \ast W_\psi \psi)(h)
      = W_\psi f (h).
  \end{align*}
  As shown above, $W_\psi f_n \to W_\psi f$
  pointwise and $|W_\psi f_n (g)| \leq C  w(g)$. On the other hand,
  given $\varphi \in \mathcal{S}_0 (\mathbb{R}^d)$,
  \Cref{lem:wavelet-Lpq} shows that $W_\psi \varphi \in L_w^1(G_A)$.
  Therefore, a combination of Corollary~\ref{cor:reproducing-formula-distr}
  with the dominated convergence theorem shows that
  \[
    \langle f, \varphi \rangle
    = \lim_{n \to \infty}
        \int_{G_A}
          W_\psi f_n (g)  \overline{W_\psi \varphi (g)}
        \, d \mu_{G_A} (g)
    = \lim_{n \to \infty}
        \langle f_n, \varphi \rangle ,
  \]
  proving that $f_n \to f$ with respect to the weak-$\ast$-topology on $\mathcal{S}'_0 (\R^d)$.
\end{proof}

\begin{lemma}\label{lem:ExtendedDualPairing}
  Let $w : G_A \to [1,\infty)$ be a standard control weight as in \Cref{lem:ControlWeights}.
  Let $\psi \in \mathcal{S}_0 (\mathbb{R}^d)$ be admissible.

 If $f \in \mathcal{S}_0'(\mathbb{R}^d)$ satisfies $W_\psi f \in L_{1/w}^\infty(G_A)$
 and $\phi \in L^2(\R^d)$ satisfies $W_\psi \phi \in L_w^1(G_A)$,
 then $\langle f, \phi \rangle_{\psi}$ is well-defined
 and independent of the choice of $\psi \in \mathcal{S}_0 (\mathbb{R}^d)$.
\end{lemma}

\begin{proof}
We first show for $Y = L_{1/w}^\infty(G_A)$ or $Y = L_w^1(G_A) $ that
if $f \in \mathcal{S}_0' (\mathbb{R}^d)$ satisfies $W_\psi f \in Y$,
then $W_\varphi f \in Y$ for every $\varphi \in \mathcal{S}_0(\mathbb{R}^d)$.
For this, first note that ${W_\varphi f = W_\psi f \ast W_\varphi \psi}$; see
Corollary~\ref{cor:reproducing-formula-distr}.
If $Y = L_w^1(G_A)$, the submultiplicativity of $w$  implies
$Y \ast L_w^1(G_A) \subset L_w^1(G_A)$, while Lemma~\ref{lem:wavelet-Lpq} shows
that $W_\varphi \psi \in L_w^1(G_A)$.
Thus, $W_\varphi f \in L^1_{w} (G_A)$.
In case of $Y = L_{1/w}^\infty (G_A)$, note that
$|(W_\varphi \psi) (h)| = |(W_\psi \varphi) (h^{-1})|$, and
\begin{align*}
  \frac{1}{w(g)} \, \bigl| W_\varphi f (g) \bigr|
   &\leq \int_{G_A}
           \frac{1}{w(h)} \, \bigl|W_\psi f (h)\bigr|
            \frac{w(h)}{w(g)} \, \bigl|W_\psi \varphi (g^{-1} h)\bigr|
         \, d \mu_{G_A} (h) \\
    &\leq \| W_\psi f \|_{L_{1/w}^\infty}
          \| W_\psi \varphi \|_{L_w^1}
\end{align*}
for all $g \in G_A$; here, we used that
\(
  \frac{w(h)}{w(g)}
  \leq \frac{w(g)w(g^{-1}h)}{w(g)}
  =    w(g^{-1} h)
  .
\)
Thus, also $W_\varphi f \in L_{1/w}^\infty(G_A)$.

Since $W_\psi f \in L_{1/w}^\infty(G_A) $ and $W_\psi \phi \in L_w^1(G_A)$ by assumption,
it is clear that $\langle f, \phi \rangle_{\psi} \in \CC$ is well-defined.
Next, let $\varphi \in \mathcal{S}_0 (\mathbb{R}^d)$ be admissible,
and let $(f_n)_{n \in \N} \subset L^2(\R^d)$ as provided by Lemma~\ref{lem:PseudoWeakStarDensity}.
Note that
\(
  W_\phi f_n (g)
  = \langle f_n, \pi(g) \phi \rangle \to \langle f, \pi(g) \phi \rangle
  = W_\phi f (g)
\)
for all $\phi \in \mathcal{S}_0 (\mathbb{R}^d)$ and $g \in G_A$.
Therefore, an application of the dominated convergence theorem implies
\begin{align*}
  \langle f, \phi \rangle_{\psi}
  & = \langle W_\psi f, W_\psi \phi \rangle
    = \lim_{n \to \infty}
        \langle W_\psi f_n , W_\psi \phi \rangle_{L^2}
   = \lim_{n \to \infty}
        \langle W_\varphi f_n , W_\varphi \phi \rangle_{L^2}
    = \langle f, \phi \rangle_{\varphi} ,
\end{align*}
where we used the isometry of $W_\varphi, W_\psi : L^2(\R^d) \to L^2(G_A)$.
\end{proof}

\subsection{Molecular decompositions}
\label{sub:MolecularDecompositions}

This section provides the proofs of Theorems~\ref{thm:intro2} and \ref{thm:intro3}.

We first show the following auxiliary claim which is implicit in the statements of
\Cref{thm:intro2,thm:intro3}.

\begin{lemma}\label{lem:SmoothWindowIndependence}
  Let $p \in (0,\infty)$, $q \in (0,\infty]$, $\alpha \in \R$, and $\beta > 0$.
  Let $w = w_{p,q}^{\alpha,\beta} : G_A \to [1,\infty)$ be a standard control weight
  as defined in \Cref{lem:ControlWeights} and let $r = \min \{ 1,p,q \}$.
  If $\psi \in L^2(\R^d)$ and if $\varphi \in \mathcal{S}_0(\R^d)$ is admissible with
  $W_{\varphi} \psi \in \WLwr$, then $W_\phi \psi \in \WLwr$
  for all $\phi \in \mathcal{S}_0 (\R^d)$.
\end{lemma}

\begin{proof}
  By \Cref{eq:reproducing-formula-L2}, the identity
  \(
    W_{\phi} \psi
    = W_{\varphi} \psi \ast W_{\phi} \varphi
  \)
  holds.
  Note that $W_\phi \varphi \in \WLwr$ by \Cref{lem:wavelet-Lpq} and $W_\varphi \psi \in \WLwr$
  by assumption.
  The weight $w := w_{p,q}^{\alpha,\beta}$ is continuous
  and submultiplicative with $w \geq 1$
  and satisfies $w(g) = w(g^{-1}) \Delta^{1/r}(g^{-1})$, meaning that it is an \emph{$r$-weight}
  in the terminology of \cite[Definition~3.1]{velthoven2022quasi}.
  Therefore, using the convolution relation $\WLwr \ast \WLwr \hookrightarrow \WLwr$
  from \cite[Corollary~3.9]{velthoven2022quasi}, we see that $W_\phi \psi \in \WLwr$, as claimed.
\end{proof}

\begin{proof}[Proof of \Cref{thm:intro2}]
 Let $\varphi \in \mathcal{S} (\R^d)$ be an admissible vector satisfying $\widehat{\varphi} \in C_c^{\infty} (\mathbb{R}^d \setminus \{0\})$ and the support condition
  \eqref{eq:analyzing_positive}; see \Cref{thm:AdmissibleVectorsExistence}.  Then an application of \Cref{prop:TL_coorbit} yields that $\TL = \Co_{\varphi} \bigl(\PTalt{-\alpha'}\bigr)$.
  Furthermore, since $\psi, \varphi \in L^2 (\mathbb{R}^d)$ are admissible and satisfy
  $W_\psi \psi, W_{\varphi} \varphi, W_{\varphi} \psi \in \WLwr$
  (see \Cref{lem:wavelet-Lpq,lem:SmoothWindowIndependence}), it follows by Lemma \ref{lem:abstract_identification} that $\TL$ can be identified with the abstract coorbit space
  $\Co_{\psi}^{\mathcal{H}} \bigl(\PTalt{-\alpha'}\bigr)$ used in \cite{velthoven2022quasi}.

  An application of \cite[Theorem~6.14]{velthoven2022quasi} yields
  a compact unit neighborhood $U \subset G_A$ such that for any $\Gamma \subset G_A$
  satisfying condition \eqref{eq:udense_relatively}, there exist molecular systems
  $(\phi_\gamma)_{\gamma \in \Gamma} \subset L^2(\R^d)$
  and $(f_\gamma)_{\gamma \in \Gamma} \subset L^2(\R^d)$,
  such that every $f \in \Co_\psi^{\mathcal{H}} (\PT[-\alpha'])$ can be represented as
  \begin{equation}
    f
    = \sum_{\gamma \in \Gamma}
        \langle f, \pi(\gamma) \psi \rangle_{\mathcal{R}_w, \mathcal{H}_w^1} \, \phi_\gamma
    = \sum_{\gamma \in \Gamma}
        \langle f, \phi_\gamma \rangle_{\mathcal{R}_w, \mathcal{H}_w^1} \, \pi(\gamma) \psi
    = \sum_{\gamma \in \Gamma}
        \langle f, f_\gamma \rangle_{\mathcal{R}_w, \mathcal{H}_w^1} \, f_\gamma
    ,
    \label{eq:FrameProofTargetIdentity}
  \end{equation}
  with unconditional convergence of the series in the weak-$\ast$-topology on the space
  $\mathcal{R}_w = \mathcal{R}_w (\psi)$ introduced in  \Cref{sec:abstract_reservoir}.
  By \Cref{lem:abstract_identification}, any $f \in \TL$ can be extended uniquely to an element
  $\widetilde{f} \in \Co_\psi^{\mathcal{H}} (\PT[-\alpha'])$.
  Since $(\phi_\gamma)_{\gamma \in \Gamma}$ and $(f_\gamma)_{\gamma \in \Gamma}$
  are molecules with respect to $\psi$, they also satisfy the molecule condition with respect to $\varphi$
  by \Cref{rem:molecule_independent}.
  Therefore,
  \Cref{eq:DualPairingClarification} shows that the dual pairings occurring in \Cref{eq:FrameProofTargetIdentity}
  coincide with the extended dual pairing from \Cref{def:ExtendedDualPairing}.
  Lastly, applying \Cref{eq:FrameProofTargetIdentity} to $\widetilde{f}$,
  using \Cref{eq:DualPairingClarification}, and restricting
  the domain of both sides of \Cref{eq:FrameProofTargetIdentity}
  to $\Schwartz_0(\R^d) \subset \mathcal{H}_w^1(\varphi) = \mathcal{H}_w^1(\psi)$,
  we see that \Cref{eq:FrameProofTargetIdentity} holds for all $f \in \TL$,
  with unconditional convergence of the series in the weak-$\ast$-topology on
  $\Schwartz'(\R^d) / \CalP(\R^d) = \Schwartz_0 ' (\R^d) \hookrightarrow \mathcal{R}_w (\psi)$.
\end{proof}

\begin{proof}[Proof of \Cref{thm:intro3}]
 As in the proof of \Cref{thm:intro2}, the Triebel-Lizorkin space $\TL$ can be identified with
 $\Co_{\psi}^{\mathcal{H}} (\PT[-\alpha'])$ of \Cref{lem:abstract_identification}. An application of
  \cite[Theorem~6.15]{velthoven2022quasi} yields a compact unit neighborhood $U \subset G_A$ such that
  for any $\Gamma \subset G_A$ satisfying condition \eqref{eq:separated_intro3}, there exists a molecular systems $(\phi_{\gamma})_{\gamma \in \Gamma}$ in $\overline{\Span \{ \pi (\gamma) \psi : \gamma \in \Gamma\} }$ such that, given $(c_{\gamma})_{\gamma \in \Gamma} \in \dot{\mathbf{p}}^{- \alpha', \beta}_{p, q}$, the vector $\widetilde{f} := \sum_{\gamma \in \Gamma} c_{\gamma} \phi_{\gamma} \in \Co_{\psi}^{\mathcal{H}} (\PT[-\alpha'])$ satisfies
  \begin{align} \label{eq:RieszProofTargetIdentity}
\langle \widetilde{f}, \pi (\gamma) \psi \rangle_{\mathcal{R}_w, \mathcal{H}^1_w} = c_{\gamma}, \quad \gamma \in \Gamma.
\end{align}
Arguing as in the proof of  \Cref{thm:intro2}, another application of \Cref{lem:abstract_identification} yields that the restriction $f = \widetilde{f}|_{\mathcal{S}_0} \in \TL$ satisfies $\langle f, \pi(\gamma) \psi \rangle_{\varphi} = \langle \widetilde{f}, \pi (\gamma) \psi \rangle_{\mathcal{R}_w, \mathcal{H}^1_w}$ for all $\gamma \in \Gamma$.
\end{proof}

\subsection{Explicit criteria}
\label{sec:criteria_molecules}

This section provides explicit criteria for coorbit molecules.
The proof relies on the following lemma concerning the standard envelope
from Definition~\ref{def:standard-env}.

\begin{lemma}\label{lem:StandardEnvelopeIntegrability}
  Let $r \in (0,1]$.
  If  $\sigma \in (0,\infty)^2$ satisfies $\sigma_1 < 1$, $\sigma_2 > |\det A|^{1/r}$
  and if $L > 1/r$, then $\Xi_{\sigma,L} \in L^r(G_A)$.
\end{lemma}

\begin{proof}
   Using Lemma~\ref{lem:HFunctionAlternative}, a change-of-variable yields
  \begin{align*}
    \int_{\R^d}
      |\eta_{L} (x,s)|r
    \, d x
    &\asymp \int_{\R^d}
             \bigl(1 + \rho_A (A^{-\PosPart{s}} \, x)\bigr)^{-Lr}
           \, d x
    = |\det A|^{\PosPart{s}}
      \int_{\R^d}
        \bigl(1 + \rho_A (y) \bigr)^{-Lr}
      \, d y \\
    &\asymp |\det A|^{\PosPart{s}},
  \end{align*}
  where the last step follows from Lemma~\ref{lem:QuasiNormIntegrability}
  and the assumption $Lr > 1$.
  Therefore,
  \begin{align*}
    \| \Xi_{\sigma,L} \|_{L^r(G_A)}
  & = \int_{\R}
        |\det A|^{-s}
        (\theta_\sigma (s))^r
        \int_{\R^d}
          (\eta_L (x,s))^r
        \, d x
      \, d s \\
    &\asymp \int_{\R}
             |\det A|^{\PosPart{s} - s}  (\theta_\sigma (s))^r
           \, d s \\
  & = \int_0^\infty
        \sigma_1^{rs}
      \, d s
      + \int_{-\infty}^0
          |\det A|^{-s}  \sigma_2^{rs}
        \, d s \\
        &= \int_0^\infty \!
        e^{s  r \ln \sigma_1}
      d s + \int_{-\infty}^0
        e^{s(r \ln \sigma_2 -  \ln |\det A|)}
      \, d s < \infty
  \end{align*}
  since $\ln \sigma_1 < 0$ and $r \ln \sigma_2 > \ln |\det A|$ by assumption.
\end{proof}

\begin{theorem}\label{thm:BetterVectorConditions}
  For $p \in (0, \infty)$, $q \in (0, \infty]$, let $r = \min \{p,q, 1\}$.
  Let $\alpha \in \R$, $\beta > 0$.
  Choose constants $L > 1$, $N \in \N_0$
  and $\delta \in (0,1)$ such that  $L \cdot (1 - \delta) > 1/r + \beta$ and
  \begin{align} \label{eq:lambdaN_condition}
    \lambda_-^{\delta N}
    > \max
      \bigg\{
        |\det A|^{\frac{1}{r} - \frac{1}{2} + |\alpha+\frac{1}{p}-\frac{1}{q}|}, \;
        |\det A|^{-\frac{1}{2} + \frac{1}{r} + \alpha + \beta - \frac{1}{q}}, \;
        |\det A|^{-\frac{1}{2} - (\alpha - \frac{1}{q})}
      \bigg\},
    \end{align}
  where $\lambda_{-} \in \mathbb{R}$ satisfies
  $1 < \lambda_{-} < \min_{\lambda \in \sigma(A)} | \lambda|$
  as in Section~\ref{sec:expansive}.

  Suppose $f \in L^2(\R^d) \cap C^N(\R^d)$ satisfies
  \begin{equation}\label{eq:assumption1}
    |f(x)| \lesssim (1 + \rho_A(x))^{-L},
    \quad
    \int_{\R^d} \|x\|^N |f(x)| \, d x < \infty,
    \quad \text{and} \quad
    \max_{|\alpha| \leq N}
      \sup_{x \in \R^d}
        |\partial^\alpha f (x)|
    < \infty,
  \end{equation}
  as well as
  \begin{equation}\label{eq:assumption2}
    \int_{\R^d} x^\alpha \, f(x) \, d x = 0
    \quad \text{for all } \alpha \in \N_0^d \text{ with } |\alpha| < N .
  \end{equation}
  Then $W_f f \in \WLwr$ for the control weight $w = w^{\alpha,\beta}_{p,q}$
  provided by Lemma~\ref{lem:ControlWeights}.
\end{theorem}

\begin{proof}
  We need to show that $M_Q ( W_f f) \in L_w^r (G_A)$.
  The proof is split into two steps.
\\\\
  \textbf{Step~1.}
  In this step, we show that $|W_ff(x,s)| \lesssim \Xi_{\tau,L(1-\delta)}(x,s)$,
  where $\tau = (\tau_1, \tau_2)$ with
  $\tau_1: =|\det A|^{-1/2}  \lambda_-^{-N \delta}$
  and $\tau_2:=|\det A|^{1/2}  \lambda_-^{N \delta}$.
  Assumptions \eqref{eq:assumption1} and \eqref{eq:assumption2}
  together with Lemma~\ref{lem:DiagonalWaveletDecay} imply that
  \begin{equation}
    \label{eq:bettervectors1}
    |W_ff(x,s)| \lesssim |\det A|^{-|s|/2}  \big(1+\rho_A(A^{-s^+}x)\big)^{-L},
    \quad x \in \R^d, s \in \R,
  \end{equation}
  and
  \(
    |W_ff(x,s)|
    \lesssim |\det A|^{-s/2} \|A^{-s}\|_{\infty}^N
    \lesssim |\det A|^{-s/2} \lambda_-^{-sN}
  \)
  for $ x \in \R^d, s \geq 0$, by Lemma~\ref{lem:expansive_cont_powers}.
  Applying this latter estimate to
  $W_ff(x,s) = \overline{ W_ff(-A^{-s}x,-s)}$ for $s<0$ yields immediately that
  $|W_ff(x,s)| = |W_ff(-A^{-s}x,-s)| \lesssim |\det A|^{s/2} \lambda_-^{sN}$ and therefore
  \begin{equation}\label{eq:bettervectors3}
    |W_ff(x,s)| \lesssim |\det A|^{-|s|/2} \lambda_-^{-|s|N},
    \quad x \in \R^d, s \in \R.
  \end{equation}
  Combining \eqref{eq:bettervectors1} and \eqref{eq:bettervectors3}
  gives
  \begin{align*}
    |W_ff(x,s)| &= |W_ff(x,s)|^{\delta} |W_ff(x,s)|^{1-\delta} \\
    &\lesssim |\det A|^{-|s|/2} \lambda_-^{-|s|N\delta} \big(1+\rho_A(A^{-s^+}x)\big)^{-L(1-\delta)},
  \end{align*}
  as desired.
 \\\\
  \textbf{Step 2.}
  Step 1 and Lemma~\ref{lem:StandardEnvelopeWienerMaximalFunction} yield
  $
    w  M_Q (W_f f)
    \lesssim w  M_Q (\Xi_{\tau,L(1-\delta)})
    \lesssim w  \Xi_{\tau,L(1-\delta)}.
  $
  Recall from Lemma~\ref{lem:ControlWeights}
  that the standard control weight satisfies
  $w \asymp \Xi_{\sigma,0} + \Xi_{\kappa,-\beta}$,
  where
  $\sigma, \kappa \in (0,\infty)^2$
  are as in the statement of Lemma~\ref{lem:ControlWeights}.
  Denote by $\tau \odot \sigma:=(\tau_1  \sigma_1, \tau_2  \sigma_2)$ component-wise multiplication. Then
  \begin{equation}
    \label{eq:bettervectors4}
    w  M_Q (W_f f)
    \lesssim w  \Xi_{\tau,L(1-\delta)}
    \lesssim \Xi_{\tau \odot \sigma,L(1-\delta)}
             + \Xi_{\tau \odot \kappa,L(1-\delta)-\beta}.
  \end{equation}

  It remains to verify the integrability conditions for standard
  envelopes of Lemma~\ref{lem:StandardEnvelopeIntegrability} for the
  right-hand side of \eqref{eq:bettervectors4}.
  The assumption $L \cdot (1 - \delta) > 1/r + \beta$ guarantees that
  \begin{equation*}
    \min\{L(1-\delta), L(1-\delta)-\beta\}  > 1/r ,
  \end{equation*}
  while the assumption \eqref{eq:lambdaN_condition} implies that both
  of the conditions
  \(
    \max\{\tau_1  \sigma_1, \tau_1  \kappa_1\}
    < 1
  \)
  and
  \(
    \min\{\tau_2 \sigma_2 , \tau_2  \kappa_2\}
    > |\det A|^{1/r}
  \)
  are satisfied.
  An application of Lemma~\ref{lem:StandardEnvelopeIntegrability}
  therefore yields $M_Q (W_f f) \in L_w^ r(G_A)$.
\end{proof}

\appendix

\section{The Peetre-type maximal function}%
\label{sec:PeetreMaximalFunction}

\begin{lemma}\label{lem:PeetreSupEssup}
  Let $A \in \GL(d,\R)$ be expansive and $\beta > 0$.
  Let either $s \in \Z$ or let $s \in \R$ and assume that $A$ is exponential.
  Let $f : \R^d \to [0,\infty)$ be continuous.
  Then
  \[
    \sup_{z \in \R^d} \frac{f(z)}{(1 + \rho_A (A^s z))^\beta}
    = \esssup_{z \in \R^d} \frac{f(z)}{(1 + \rho_A (A^s z))^\beta}
    .
  \]
\end{lemma}

\begin{proof}
  First, we see from the definition of $\rho_A$ (see Equation \eqref{eq:step_norm})
  for arbitrary $\lambda \in \R$ that
  \begin{equation}
    \{ x \in \R^d \colon \rho_A (x) < \lambda \}
    = \begin{cases}
        \emptyset,         & \text{if } \lambda \leq 0, \\
        A^{k+1} \Omega_A , & \text{if } |\det A|^k < \lambda \leq |\det A|^{k+1} \text{ for } k \in \Z .
      \end{cases}
    \label{eq:QuasiNormUpperSemicontinuous}
  \end{equation}
  Since $\Omega_A \subset \R^d$ is open,
  this shows that $\{ x \in \R^d \colon \rho_A(x) < \lambda \}$
  is always open.

  We now show for arbitrary $\theta \in \R$ that
  $W := \bigl\{ z \in \R^d \colon f(z) / (1 + \rho_A(A^s z))^\beta > \theta \bigr\}$
  is open; this then easily implies the claim of the lemma,
  since every non-empty open set has positive Lebesgue measure.
  First, if $\theta < 0$, then $W = \R^d$ is open.
  Next, if $\theta = 0$, then $W = \{ z \in \R^d \colon f(z) \neq 0 \}$ is open.
  Finally, let $\theta > 0$, $z_0 \in W$, $a := f(z_0)$, and $b := (1 + \rho_A(A^{s} z_0))^\beta$.
  Then $a/b > \theta > 0$ and hence $a > 0$.
  We can thus choose $0 < a' < a$ with $a'/b > \theta$, meaning that
  $\rho_A (A^s z_0) < (a'/\theta)^{1/\beta} - 1$.
  Note that $U := \{ z \in \R^d \colon f(z) > a'\}$ is an open neighborhood of $z_0$.
  Similar, the considerations from the beginning of the proof show that
  $V := \{ z \in \R^d \colon \rho_A(A^s z) < (a'/\theta)^{1/\beta} - 1 \}$
  is an open neighborhood of $z_0$.
  Finally, note that $U \cap V \subset W$.
  Since $z_0 \in W$ was arbitrary, this shows that $W \subset \R^d$ is indeed open.
\end{proof}

\section{Peetre-type maximal function on \texorpdfstring{$\mathbb{R}^d \times \mathbb{R}$}{âáµÃâ}}

Let $A \in \mathrm{GL}(d, \mathbb{R})$ be an expansive exponential matrix and $\beta > 0$.
For any measurable function $F : \mathbb{R}^d \times \mathbb{R} \to \mathbb{C}$, define
\[
  \GroupMaximal (x,s)
  := \esssup_{z \in \mathbb{R}^d}
       \frac{|F(x+z, s)|}{(1 + \rho_A (A^{-s} z))^{\beta}}
\]
for $(x, s) \in \mathbb{R}^d \times \mathbb{R}$.
We use the following basic properties repeatedly.

\begin{lemma}\label{lem:GroupMaximalProperties}
  If $F : \mathbb{R}^d \times \mathbb{R} \to \mathbb{C}$ is measurable,
  then $\GroupMaximal : \mathbb{R}^d \times \mathbb{R} \to [0, \infty]$ is measurable.
  Furthermore, there is a constant $C = C(\beta,A) \geq 1$ such that for each $s \in \R$,
  there is a null-set $N_s \subset \R^d$ such that for all $x, w \in \R^d$
  with $x + w \notin N_s$, we have
  \[
    \frac{|F(x+w, s)|}{(1 + \rho_A (A^{-s} w))^\beta}
    \leq C  \GroupMaximal (x,s).
  \]
  In particular, if $\GroupMaximal = 0$ almost everywhere, then $F = 0$ almost everywhere.
\end{lemma}

\begin{proof}
  Since the map $H : (x,z,s) \mapsto \frac{|F(x+z, s)|}{(1 + \rho_A(A^{-s} z))^\beta}$
  is measurable, it is well-known that $\GroupMaximal(x,s) = \esssup_{z \in \R^d} H(x,z,s)$
  is measurable as well; see, e.g.,~\cite[Lemma~B.4]{holighaus2020schur}.

  Let us fix $s \in \R$ and write $F_s (x) := F(x,s)$ and $Q_s := A^s \Omega_A$,
  with $\Omega_A \subset \R^d$ as in \Cref{lem:QuasiNorm}.
  For the open unit neighborhood $Q_s \subset \R^d$, we
  consider the local maximal function
  $M^L_{Q_s} f (x) = \esssup_{q \in Q_s} |f(x + q)|$ of a measurable $f : \R^d \to \CC$.

  Note that if $x \in \Omega_A \setminus \{ 0 \}$, then $x \in A^j \Omega_A$
  for some minimal $j \in \Z$, which necessarily satisfies $j \leq 0$.
  Hence, $x \in A^{(j-1)+1} \Omega_A \setminus A^{j-1} \Omega$,
  so that $\rho_A (x) = |\det A|^{j-1} \leq |\det A|^{-1} \leq 1$, by definition of $\rho_A$.
  Furthermore, \Cref{lem:QuasiNorm} yields $C' \geq 1$ with
  $\rho_A(x+y) \leq C' [\rho_A(x) + \rho_A(y)]$ for all $x,y \in \R^d$.
  For $q \in Q_s$, we then have $A^{-s} q \in \Omega_A$ and hence
  $\rho_A(A^{-s} q) \leq 1$.
  Thus, $1 + \rho_A (A^{-s} (z + q)) \leq (1 + C') (1 + \rho_A (A^{-s} z))$
  for $q \in Q_s$ and $z \in \R^d$.

  Now, by definition of $\GroupMaximal$, given $x \in \R^d$ and $s \in \R$, there is a null-set
  $N_{s,x} \subset \R^d$ with
  \[
    \frac{|F(x+z, s)|}{(1 + \rho_A(A^{-s} z))^\beta}
    \leq \GroupMaximal (x,s),
    \qquad  \, z \in \R^d \setminus N_{s,x} .
  \]
  Fix $x,z \in \R^d$ and $s \in \R$.
  Then, for $q \in Q_s \setminus (N_{s,x} - z)$,
  we have $z + q \in \R^d \setminus N_{s,x}$, and hence
  \[
    \GroupMaximal (x,s)
    \geq \frac{|F_s(x + z + q)|}{(1 + \rho_A(A^{-s} (z+q)))^\beta}
    \geq (1 + C')^{-\beta}
         \frac{|F_s (x + z + q)|}{(1 + \rho_A (A^{-s} z))^\beta} .
  \]
  Therefore,
  $M^L_{Q_s} F_s (x + z) \leq C  (1 + \rho_A(A^{-s} z))^\beta  \GroupMaximal (x,s)$
  for all $x,z \in \R^d$ and $s \in \R$, with $C := (1 + C')^\beta$.

  Lastly, it follows by \cite[Lemma~2.3.3]{VoigtlaenderPhDThesis} that there exists
  a null-set $N_s = N_{s,F} \subset \R^d$
  with $|F_s(x)| \leq M^L_{Q_s} F_s (x)$ for all $x \in \R^d \setminus N_s$.
  For $x,w \in \R^d$ with $x + w \notin N_s$, we then see that
  \begin{align*}
    |F(x+w, s)|
    &= |F_s (x + w)|
    \leq M^L_{Q_s} F_s (x + w) \\
    &\leq C  (1 + \rho_A (A^{-s} w))^\beta \, \GroupMaximal (x, s),
  \end{align*}
  which completes the proof.
\end{proof}

The following lemma allows to estimate maximal functions of the standard envelope defined in Definition~\ref{def:standard-env}.

\begin{lemma}\label{lem:StandardEnvelopeWienerMaximalFunction}
  For arbitrary $\sigma \in (0,\infty)^2$ and $L \geq 0$, let
  $\Xi_{\sigma,L}: G_A \to (0,\infty)$ denote the standard envelope
  defined in Definition~\ref{def:standard-env}.
  Then for any relatively compact unit-neighborhood $Q \subset G_A$,
  there exists a constant $C = C(Q,\sigma, L, A) > 0$
  such that $M_Q  \Xi_{\sigma,L} \leq C \cdot \Xi_{\sigma,L}$.
\end{lemma}

\begin{proof}
  Since $Q \subset G_A$ is relatively compact, we have $Q \subset [-N,N]^d \times [-N,N]$
  for some $N \geq 1$.
Recall that
  $\Xi_{\sigma,L}(x,s)= \theta_\sigma (s) \cdot \eta_L(x,s)$ with
  \begin{equation*}
    \eta_L(x,s)
    \asymp \big(
    1 + \rho_A(A^{-s^{+}} x) \big)^{-L}
    \qquad \text{and} \qquad
    \theta_\sigma (s)
    = \begin{cases}
      \sigma_1^s , & \text{if } s \geq 0, \\
      \sigma_2^s , & \text{if } s < 0,
    \end{cases}
  \end{equation*}
  where the first estimate is due to
  Lemma~\ref{lem:HFunctionAlternative} and $s^{+}:= \max\{0,s\}$.

  We split the proof into several steps and treat the two factors separately.
\\\\
  \textbf{Step 1.}  We show that, for all $s \in \R$ and $t \in [-N,N]$, we have
  \(
    \theta_{\sigma} (s + t)
    \leq c^4  \theta_{\sigma} (s),
  \)
  where
  \(
    \strut c
    := \max
       \bigl\{
         \sigma_1^N, \sigma_1^{-N}, \sigma_2^N, \sigma_2^{-N}
       \bigr\}
    \in [1,\infty)
    .
  \)
  Note that if $s > N$, then $s + t > 0$, and hence
  \(
    \theta_{\sigma} (s+t)
    = \sigma_1^{s+t}
    = \sigma_1^t  \theta_\sigma (s)
    \leq c  \theta_\sigma (s)
    .
  \)
  Likewise, if $s < -N$, then  $s + t < 0$ and thus
  \(
    \theta_\sigma (s+t)
    = \sigma_2^{s+t}
    = \sigma_2^t  \theta_\sigma (s)
    \leq c  \theta_\sigma (s) .
  \)
  Lastly, if ${s \in [-N,N]}$, then $s, s + t \in [-2N,2N]$.
  But for $x \in [-2N,2N]$, it follows that ${c^{-2} \leq \theta_\sigma (x) \leq c^2}$,
  and hence $\theta_\sigma (s+t) \leq c^{4}  \theta_\sigma(s)$.
\\\\
  \textbf{Step 2.}
  We show that $ M_Q^R \eta_{L} \lesssim \eta_{L}$, where
  $M_Q^R F (g) := \esssup_{u \in Q} |F(u g)|$ for any measurable $F: G_A \to \CC$.
  Let ${(x,s) \in G_A}$ and $(y,t) \in Q$ be arbitrary.
  Then Corollary~\ref{cor:quasi-norm_bound} implies
  \begin{align*}
    \rho_A(A^{-s^+} x)
     & \asymp |\det A|^{-s^+ -t} \rho_A( A^t x) \\
     &\lesssim |\det A|^{-s^+ -t}  \Big(\rho_A( A^t x + y) + \rho_A(-y) \Big) \\
    & \lesssim |\det A|^{-s^+ +(s+t)^+ -t}
                \Big(\rho_A\big( A^{-(s+t)^+}(A^t x + y)\big) + 1 \Big),
  \end{align*}
  where we used in the last step that $\rho_A(y) \lesssim 1$ for
  $ y \in [-N,N]^d$; see Lemma~\ref{lem:quasi-norm-bounds}.

  Note that $(s+t)^+ \leq s^+ + t^+ \leq s^+ + N$
  for $t \in [-N,N]$, and hence $-s^+ +(s+t)^+ -t \leq 2 N$ for all
  $s \in \R$ and $t \in [-N,N]$.
  Therefore, $|\det A|^{-s^+ +(s+t)^+ -t} \lesssim 1$
  and consequently
  \begin{align*}
    \eta_{L}\big( (y,t) (x,s) \big)
    &= \eta_{L}( A^t x + y, s+t)
      \asymp \big( 1 + \rho_A\big( A^{-(s+t)^+}(A^t x + y) \big) \big)^{-L} \\
    & \lesssim \big( 1+ \rho_A(A^{-s^+} x) \big)^{-L}
      \asymp \eta_{L}(x,s)
  \end{align*}
  for all $(x,s) \in G_A$ and $(y,t) \in Q$; here, we used that $L \geq 0$.
\\\\
\textbf{Step 3.}
  Similar to Step 2, we show that $M_Q^L \eta_{L} \lesssim \eta_{L}$, for
  $M^L_Q F (g) = \esssup_{v \in Q} |F(gv)|$.
  Let ${(x,s) \in G_A}$ and $(y,t) \in Q$ be arbitrary.
  Then Corollary~\ref{cor:quasi-norm_bound} again implies
  \begin{align*}
    \rho_A(A^{-s^+} x)
     & \asymp |\det A|^{-s^+} \rho_A (x) \\
      & \lesssim |\det A|^{-s^+}  \Big(\rho_A( x + A^s y) + \rho_A(-A^s y) \Big)\\
     & \lesssim |\det A|^{-s^+ +(s+t)^+} \rho_A\big(A^{-(s+t)^+}( x + A^sy) \big)
      + |\det A|^{-s^+ + s}\rho_A(y)  \\
    & \lesssim \rho_A\big(A^{-(s+t)^+}( x + A^sy) \big) + 1,
  \end{align*}
  since we have $|\det A|^{-s^+ +(s+t)^+}, |\det A|^{-s^+ + s} \lesssim 1$
  for all $s \in \R$ and $t \in [-N,N]$,
  and since $\rho_A(y) \lesssim 1$ for $y \in [-N,N]^d$ by Lemma~\ref{lem:quasi-norm-bounds}.
  Therefore,
  \begin{align*}
    \eta_{L}\big(  (x,s) (y,t) \big)
    &= \eta_{L}( x + A^sy, s+t)
      \asymp \big( 1 + \rho_A\big( A^{-(s+t)^+}(x + A^sy) \big) \big)^{-L} \\
    & \lesssim \big( 1+ \rho_A(A^{-s^+} x) \big)^{-L}
      \asymp \eta_{L}(x,s)
  \end{align*}
  for all $(x,s) \in G_A$ and $(y,t) \in Q$.

  \smallskip{}

  In combination, the obtained estimates easily imply the claim.
\end{proof}

\section{Proof of Lemma~\ref{lem:SpecialWeight}}%
\label{sub:SpecialWeightProof}

\begin{proof}
  \textbf{Step~1.}
  In this step, we prove the bound
  \begin{equation}
    v(y,t)
    \lesssim \max
             \{
               1,
               |\det A|^{-t}
             \} \,
             \big(
               1 + \min
                   \{
                     \rho_A(y),
                     \rho_A (A^{-t} y)
                   \}
             \big)
    ,
    \label{eq:SpecialWeightUpperBound}
  \end{equation}
  which will also imply that $v$ is well-defined.
  To this end, first note by the quasi-triangle inequality for $\rho_A$
  that there exists $H \geq 1$ such that
  \begin{align*}
    1 + \rho_A (x)
    &\leq\! 1 + H  \big( \rho_A (x-y) + \rho_A(y) \big) \\
   & \leq\! H  (1 + \rho_A(x-y) + \rho_A(y)) \\
   & \leq\! H  (1 + \rho_A(x-y))  (1 + \rho_A(y)) ,
  \end{align*}
  and hence $(1 + \rho_A(x-y))^{-1} \leq H \cdot \frac{1 + \rho_A(y)}{1 + \rho_A(x)}$.
  Next, we note as a consequence of \Cref{cor:quasi-norm_bound} that
  \(
    1 + \rho_A (A^{-(u - t)} z)
    \gtrsim 1 + |\det A|^t \rho_A (A^{-u} z)
    \geq \min\{ 1 , |\det A|^t \}  (1 + \rho_A(A^{-u} z)) ,
  \)
  so that
  \begin{align*}
    \bigl(1 + \rho_A (A^{-(u-t)} z)\bigr)^{-1}
    & \lesssim \big( \min \bigl\{ 1, |\det A|^t \bigr\}\big)^{-1}
                \bigl(1 + \rho_A (A^{-u} z)\bigr)^{-1} \\
    & =        \max \bigl\{ 1, |\det A|^{-t} \bigr\}
                \bigl(1 + \rho_A (A^{-u} z)\bigr)^{-1} .
  \end{align*}
  Combining this with the previous estimate, we see
  \[
    \bigl(1 + \rho_A (A^{-u} A^t z - y)\bigr)^{-1}
    \lesssim \frac{1 + \rho_A (y)}
                  {1 + \rho_A (A^{-u} A^t z)}
    \lesssim \max \{ 1, |\det A|^{-t} \}
             \frac{1 + \rho_A (y)}
                  {1 + \rho_A (A^{-u} z)} ,
  \]
  which shows that $v(y,t) \lesssim \max \{ 1, |\det A|^{-t}  \}  (1 + \rho_A (y))$.

  On the other hand, using \Cref{cor:quasi-norm_bound}, we see
  \begin{align*}
    \bigl(1 \!+\! \rho_A (A^{-(u - t)} z - y)\bigr)^{-1}
    & = \bigl(1 \!+\! \rho_A (A^t [A^{-u} z - A^{-t} y]) \bigr)^{-1} \\
     & \asymp \bigl(1 \!+\! |\det A|^t \rho_A (A^{-u} z - A^{-t} y) \bigr)^{-1} \\
    & \leq \frac{\max \{ 1, |\det A|^{- t} \}}
                {1 +\rho_A (A^{-u} z - A^{-t} y)} \\
    & \lesssim \max \bigl\{ 1, |\det A|^{- t} \bigr\}
                \frac{1 + \rho_A(A^{-t} y)}
                          {1 + \rho_A(A^{-u} z)} ,
  \end{align*}
  which implies $v(y,t) \lesssim \max \{ 1, |\det A|^{-t} \}  (1 + \rho_A (A^{-t} y))$.
  This shows \Cref{eq:SpecialWeightUpperBound}.
\\\\
  \textbf{Step 2.}
  As a consequence of \eqref{eq:QuasiNormUpperSemicontinuous},
  we see for arbitrary $\theta > 0$ and $\lambda \in \R$ that
  \[
    \Big\{
      x \in \R^d
      \colon
      \tfrac{\theta}{1 + \rho_A(x)} > \lambda
    \Big\}
    = \Big\{
        x \in \R^d
        \colon
        \rho_A(x) < \tfrac{\theta}{\lambda} - 1
      \Big\}
  \]
  is open, meaning that $\frac{\theta}{(1+\rho_A)^\beta}$ is lower semi-continuous.
  Based on this, it is not hard to see that $v$ is lower semi-continuous
  as a supremum of continuous functions; see \cite[Proposition~7.11]{FollandRA}.
  This means that $\{ (y,t) \in G_A \colon v(y,t) > \lambda \}$ is open for all $\lambda \in \R$,
  so that $v$ is measurable.
\\\\
\textbf{Step~3.}
  Define $\gamma(x,s) := 1 + \rho_A (A^{-s} x)$ for brevity.
  Note that
  \begin{align*}
    \gamma \big( (z,u) (y,t)^{-1} \big)
    &= \gamma \bigl( z - A^{u} A^{-t} y, \,\, u - t \bigr) \\
    &= 1 + \rho_A \bigl(A^{-(u - t)} (z - A^{u} A^{-t} y)\bigr) \\
    &= 1 + \rho_A \bigl( A^{-u} A^t z - y \bigr) .
  \end{align*}
  Thus,
  \(
    v(y,t)
    = \sup_{(z,u) \in G_A}
        \gamma(z,u) / \gamma \bigl( (z,u) (y,t)^{-1} \bigr)
  \)
  and hence $v(g) = \sup_{\kappa \in G_A} \gamma(\kappa) / \gamma(\kappa g^{-1})$ for $g \in G_A$.
  This easily implies that $v$ is submultiplicative; indeed, for $g,h \in G_A$, we see
  for $\widetilde{\kappa} := \kappa \, h^{-1}$ that
  \begin{align*}
    v(g h)
    &= \sup_{\kappa \in G_A}
        \frac{\gamma(\kappa)}
             {\gamma(\kappa h^{-1} g^{-1})}
    = \sup_{\kappa \in G_A}
        \frac{\gamma(\kappa)}
             {\gamma(\kappa h^{-1})}
        \frac{\gamma(\kappa h^{-1})}
             {\gamma(\kappa h^{-1} g^{-1})} \\
    &\leq v(h)
         \sup_{\widetilde{\kappa} \in G_A}
           \frac{\gamma(\widetilde{\kappa})}
                {\gamma(\widetilde{\kappa} g^{-1})}
    \leq v(h) v(g) ,
  \end{align*}
  as claimed.
  \\\\
  \textbf{Step~4.}
  Starting from the definition \eqref{eq:SpecialWeightDefinition} of $v$,
  the substitutions $a = A^{-u} z$ and $b = A^t a - y$ show
  by \Cref{cor:quasi-norm_bound} that
  \[
    v(y,t)
    = \sup_{a \in \R^d}
        \frac{1 + \rho_A(a)}{1 + \rho_A (A^t a - y)}
    = \sup_{b \in \R^d}
        \frac{1 + \rho_A(A^{-t} (b + y))}{1 + \rho_A (b)}
    \asymp \sup_{b \in \R^d}
             \frac{1 + |\det A|^{-t} \rho_A(b + y)}{1 + \rho_A (b)} .
  \]
  By using the second-to-last expression and setting $b = 0$,
  we see $v(y,t) \geq 1 + \rho_A (A^{-t} y)$.

  Next, note that
  \({
    \rho_A (b)
    \leq H  \big( \rho_A (b + y) + \rho_A(-y) \big)
  }\)
  and hence ${\rho_A(b + y) \geq H^{-1} \rho_A(b) \!-\! \rho_A(y)}$,
  by the symmetry of $\rho_A$.
  Furthermore, \Cref{lem:quasi-norm-bounds} shows $\rho_A (b) \to \infty$ as $\| b \| \to \infty$.
  Therefore, as $\| b \| \to \infty$,
  \begin{align*}
    v(y,t)
    &\!\gtrsim\! \sup_{b \in \R^d}
                  \frac{1 \!+\! |\det A|^{-t} \rho_A(b \!+\! y)}
                       {1 \!+\! \rho_A (b)} \\
    &\geq    \frac{1 \!+\! |\det A|^{-t} \!\cdot\! (H^{-1} \rho_A(b) \!-\! \rho_A(y))}
                 {1 \!+\! \rho_A (b)} \\
    &\to
            H^{-1} |\det A|^{-t} ,
  \end{align*}
  so that we also get $v(y,t) \gtrsim |\det A|^{-t}$.
  Overall, we see $v(y,t) \gtrsim 1 + |\det A|^{-t} + \rho_A(A^{-t} y)$.

  There are now two cases:
  If $t \geq 0$, then \Cref{cor:quasi-norm_bound} shows that
  \begin{align*}
    \max \bigl\{ 1, |\det A|^{-t} \bigr\}
    & \bigl(1 + \min \{ \rho_A (y), \rho_A (A^{-t} y) \}\bigr) \\
    & = 1 + \min \{ \rho_A (y), \rho_A (A^{-t} y) \} \\
    & \asymp 1 + \min \{ \rho_A (y), |\det A|^{-t} \rho_A (y) \} \\
    & \asymp 1 + \rho_A(A^{-t} y) \\
    & \asymp 1 + |\det A|^{-t} + \rho_A(A^{-t} y) \\
    &\lesssim v(y,t) .
  \end{align*}
  Otherwise, in case of $t < 0$, we see
  \begin{align*}
    \max \bigl\{ 1, |\det A|^{-t} \bigr\}
    & \bigl(1 \!+\! \min \{ \rho_A (y), \rho_A (A^{-t} y) \}\bigr) \\
    & \asymp |\det A|^{-t}
             \big( 1 \!+\! \min \{ \rho_A(y), |\det A|^{-t} \rho_A(y) \} \big) \\
    & =      |\det A|^{-t}  \big( 1 + \rho_A(y) \big) \\
    & \asymp |\det A|^{-t} + \rho_A(A^{-t} y) \\
    & \asymp 1 + |\det A|^{-t} + \rho_A(A^{-t} y) \\
     &\lesssim v(y,t).
  \end{align*}
  In combination with \Cref{eq:SpecialWeightUpperBound}, this proves \Cref{eq:VExplicitBound}.
\end{proof}

\section{Independence of coorbit reservoir} \label{sec:abstract_reservoir}
Let $\psi \in L^2 (\mathbb{R}^d)$ be an admissible vector satisfying $W_{\psi} \psi \in \WLwr$ for the standard control weight
$w = w^{\alpha, \beta}_{p,q} : G_A \to [1, \infty)$  provided by \Cref{lem:ControlWeights}.
Define the space
\[ \mathcal{H}^1_w (\psi) = \big\{ f \in L^2 (\mathbb{R}^d) : W_{\psi} f \in L^1_w (G_A) \big\}
\]
 and equip it with the norm $\| f \|_{\mathcal{H}^1_w (\psi)} := \| W_{\psi} f \|_{L^1_w}$.
 Let $\mathcal{R}_w (\psi) :=   ( \mathcal{H}^1_w (\psi) )^*$ be the anti-dual space of $\mathcal{H}^1_w (\psi)$ and write $V_{\phi} f := \langle f, \pi(\cdot) \phi \rangle_{\mathcal{R}_w , \mathcal{H}^1_w}$
 for $f \in \mathcal{R}_w (\psi)$ and $\phi \in \mathcal{H}_w^1 (\psi)$.

 The following lemma is a special case of \cite[Corollary 4.9]{velthoven2022quasi}.

\begin{lemma} \label{lem:abstract_identification}
Let $\alpha \in \R$, $\beta > 0$, and $p \in (0,\infty),q \in (0,\infty]$,
with $r := \min \{1, p ,q\}$.
Let $w = w^{\alpha, \beta}_{p,q} : G_A \to [1, \infty)$ be a standard control weight
for $\PT$ as provided by \Cref{lem:ControlWeights}.
Suppose $\varphi \in \mathcal{S}_0 (\mathbb{R}^d)$ is admissible and $\psi \in L^2 (\mathbb{R}^d)$ is admissible satisfying $W_{\psi} \psi \in \WLwr$ and $W_{\varphi} \psi \in \WLwr$. Then
\[
\Co_{\psi}^{\mathcal{H}} (\PT) := \big\{ f  \in \mathcal{R}_w (\psi) : M^L_Q V_{\psi} f \in  \PT \big\} = \Co_{\varphi} (\PT)
\]
in the sense that the restriction $f \mapsto f|_{\mathcal{S}_0}$ is a well-defined bijection.
Furthermore, given the unique extension $\widetilde{f} \in \Co_{\psi}^{\mathcal{H}} (\PT)$  of $f \in \Co_{\psi} (\PT)$, then
\[
\langle \widetilde{f}, \phi \rangle_{\mathcal{R}_w, \mathcal{H}^1_w} = \langle f, \phi \rangle_{\varphi} \quad \phi \in \mathcal{H}^1_w (\varphi),
\]
where $\langle \cdot , \cdot \rangle_{\varphi}$ denotes the extended pairing of \Cref{def:ExtendedDualPairing}.
\end{lemma}
\begin{proof}
We first verify that the Peetre-type spaces $\PT$ satisfy the standing assumptions of \cite{velthoven2022quasi}. As shown in \Cref{lem:TranslationNormBounds},
  the Peetre space $\PT$
  is a solid, translation-invariant quasi-Banach space, and \Cref{lem:PeetreNormIsRNorm}
  shows that $\| \cdot \|_{\PT}$ is an $r$-norm.
  Moreover, the standard control weight $w := w_{p,q}^{\alpha, \beta} : G_A \to [1,\infty)$
  defined in \Cref{lem:ControlWeights} is continuous, submultiplicative  and satisfies $w(g) = w(g^{-1}) \Delta^{1/r}(g^{-1})$.
  Furthermore, \Cref{lem:ControlWeights} shows that
  $\| L_{h^{-1}} \|_{\PT \to \PT} \leq w(h)$ and $\| R_h \|_{\PT \to \PT} \leq w(h)$
  for all $h \in G_A$. Together, this shows that $w$ is a \emph{strong control weight}
  for $\PT$ in the terminology of \cite[Definition~3.1]{velthoven2022quasi}.
  By \cite[Corollary~3.9]{velthoven2022quasi}, this implies that the pair
  $(\PT, w)$ is $L_w^r$-compatible
  in the sense of \cite[Definition~3.5]{velthoven2022quasi}.

  Since $\psi, \varphi \in L^2 (\mathbb{R}^d)$ are admissible and satisfy
  $W_\psi \psi, W_{\varphi} \varphi, W_{\varphi} \psi \in \WLwr$,
  it follows from \cite[Lemma~4.3 and Proposition~4.8]{velthoven2022quasi}
  that $\mathcal{H}_w^1(\psi) = \mathcal{H}_w^1(\varphi)$
  and hence also $\mathcal{R}_w(\psi) = \mathcal{R}_w (\varphi)$. Therefore,
  \begin{align} \label{eq:independence_coorbit}
    \Co_\psi^{\mathcal{H}} (\PT)
     = \bigl\{
          f \in \mathcal{R}_w (\psi)
          \colon
          M^L_Q V_\psi f \in \PT
        \bigr\}
     = \bigl\{
          f \in \mathcal{R}_w (\varphi)
          \colon
          M^L_Q V_{\varphi} f \in \PT
        \bigr\}
    .
  \end{align}
   \Cref{lem:wavelet-Lpq} easily shows that
  $\mathcal{S}_0(\R^d) \hookrightarrow \mathcal{H}_w^1(\varphi)$,
  and \Cref{lem:weak-continuity} shows that \cite[Equation~(4.14)]{velthoven2022quasi}
  is satisfied for $ \mathcal{S}_0(\R^d)$.
  Therefore, invoking \cite[Corollary~4.9]{velthoven2022quasi}, it follows that the restriction map
  $
    f \mapsto f|_{\mathcal{S}_0(\R^d)}
  $ is a bijection from  $\Co_{\varphi}^{\mathcal{H}} \bigl(\PT \bigr)$ onto $\Co_{\varphi} \bigl(\PT \bigr)$.
  Combining this with \Cref{eq:independence_coorbit}
 yields that
  \[
    \Co_\psi \bigl(\PT\bigr)
    \to \Co_{\varphi}\bigl(\PT\bigr), \quad
    f \mapsto f|_{\Schwartz_0 (\R^d)}
  \]
  is a well-defined bijection.

  By the above, any $f \!\in\! \Co_{\varphi} (\PT)$ uniquely extends to an element of
  \( \widetilde{f} \in
    \Co_\psi \bigl( \PT)
    \!\subset\! \mathcal{R}_w (\psi)
    ,
  \)
  denoted by $\widetilde{f}$.
  Note that $V_{\varphi} \widetilde{f} = W_{\varphi} f$.
  Then, a combination of \cite[Lemma~4.6(iii)]{velthoven2022quasi}
  and \Cref{def:ExtendedDualPairing} shows for
  any $\phi \in L^2(\R^d)$ with $W_{\varphi} \phi \in L_w^1 (\R^d)$
  (i.e., $\phi \in \mathcal{H}_w^1 (\varphi)$) that
  \begin{equation}
    \langle \widetilde{f}, \phi \rangle_{\mathcal{R}_w, \mathcal{H}_w^1}
    = \langle W_{\varphi} f, W_{\varphi} \phi \rangle_{L_{1/w}^\infty, L_w^1}
    = \langle f, \phi \rangle_{\varphi}
    .
    \label{eq:DualPairingClarification}
  \end{equation}
  This completes the proof.
  \end{proof}

\section*{Acknowledgements}
S.~K. ~was supported by projects P 30148 and I 3403 of the Austrian Science Fund (FWF).
J.~v.~V. ~gratefully acknowledges the support
from the Austrian Science Fund (FWF) projects P 29462 and J-4445
and the Research Foundation - Flanders (FWO) Odysseus 1 grant G.0H94.18N

\end{document}